\documentclass[a4paper,10pt]{amsart}

\usepackage{amsmath,amsfonts,amssymb,amsthm,amscd}
\usepackage{graphicx}
\usepackage{perpage}
\usepackage{url}
\usepackage{color}
\usepackage[T1]{fontenc}
\usepackage{hyperref}
\usepackage[a4paper,scale={0.72,0.74},marginratio={1:1},footskip=7mm,headsep=10mm]{geometry}
\usepackage[latin9]{inputenc}
\usepackage{mathrsfs}
\usepackage{tikz}
\usepackage{bbm}
\usepackage{tikz}
\usepackage{pgfplots}

\numberwithin{equation}{section}

\newtheorem{theorem}{Theorem}[section]
\newtheorem{lemma}[theorem]{Lemma}
\newtheorem{proposition}[theorem]{Proposition}
\newtheorem{corollary}[theorem]{Corollary}
\newtheorem{remark}[theorem]{Remark}
\newtheorem{definition}[theorem]{Definition}

\newcommand{\R}{\mathbb{R}}

\newcommand{\N}{\mathbb{N}}

\newcommand{\cH}{\mathcal{H}}
\newcommand{\cP}{\mathcal{P}}
\newcommand{\cC}{\mathcal{C}}
\newcommand{\cL}{\mathcal{L}}

\newcommand{\stab}{\mathrm{stab}}
\newcommand{\meta}{\mathrm{meta}}

\usetikzlibrary{calc}

\makeatother

\begin{document}


\title{Metastability on the hierarchical lattice}

\author{Frank den Hollander, Oliver Jovanovski}

\date{\today}

\begin{abstract}
We study metastability for Glauber spin-flip dynamics on the $N$-dimensional 
hierarchical lattice with $n$ hierarchical levels. Each vertex carries an Ising spin 
that can take the values $-1$ or $+1$. Spins interact with an external magnetic 
field $h>0$. Pairs of spins interact with each other according to a ferromagnetic pair 
potential $\vec{J}=\{J_i\}_{i=1}^n$, where $J_i>0$ is the strength of the interaction 
between spins at hierarchical distance $i$. Spins flip according to a 
Metropolis dynamics at inverse temperature $\beta$. In the limit as $\beta\to\infty$, 
we analyse the crossover time from the metastable state $\boxminus$ (all spins 
$-1$) to the stable state $\boxplus$ (all spins $+1$). Under the assumption that 
$\vec{J}$ is non-increasing, we identify the mean transition time up to a multiplicative 
factor $1+o_\beta(1)$. On the scale of its mean, the transition time is exponentially 
distributed. We also identify the set of configurations representing the gate for the 
transition. For the special case where $J_i = \tilde{J}/N^i$, $1 \leq i \leq n$, with 
$\tilde{J}>0$ the relevant formulas simplify considerably. Also the hierarchical 
mean-field limit $N\to\infty$ can be analysed in detail. 
\end{abstract}

\maketitle



\section{Introduction}
\label{S1}

Interacting particle systems evolving according to a \emph{Metropolis dynamics} 
associated with an energy functional, called the \emph{Hamiltonian}, may end up 
being trapped for a long time near a state that is a local minimum but not a global 
minimum. Just how long it takes for the system to escape from the energy valley 
around a local minimum and reach the global minimum depends on how deep 
this valley is. The deepest local minima are called \emph{metastable states}, the 
global minimum is called the \emph{stable state}. While being trapped near a 
metastable state, the system is said to be in a quasi-equilibrium. The transition 
to the stable state marks the relaxation of the system to equilibrium. To describe 
this relaxation, it is of interest to compute the transition time and to identify the 
set of critical configurations the system has to cross in order to achieve the 
transition. The critical configurations constitute the lowest saddle points in 
the energy landscape encountered along paths that achieve the crossover.

Metastability for interacting particle systems on \emph{lattices} has been 
studied intensively in the past three decades. Various different approaches 
have been proposed, which are summarised in the monographs by Olivieri 
and Vares~\cite{OV05}, Bovier and den Hollander~\cite{BdH15}. Recently, 
there has been interest in metastability for interacting particle systems on 
\emph{random graphs}, which is much more challenging because the 
transition time depends in a delicate manner on the realisation of the graph.

In the present paper we are interested in metastability for Glauber spin-flip 
dynamics on the $N$-dimensional \emph{hierarchical lattice} at low temperature. 
We obtain a full description of both the transition time and the set of critical 
configurations representing the gate for the transition. Our results are part 
of a larger enterprise in which the goal is to understand metastability on large 
graphs. The hierarchical lattice is interesting because it has a non-trivial 
geometric structure and allows for a rich variability in the choice of the interaction 
parameters. 

The outline of the paper is as follows. In Section~\ref{S1.1} we recall the definition 
of Glauber spin-flip dynamics on an arbitrary finite connected graph. In Section~\ref{S1.2} 
we recall the basic geometric definitions that are needed for the description of 
metastability and recall three key theorems from the literature that are valid in the limit 
of low temperature. These theorems, which are based on a certain \emph{key hypothesis}
but are otherwise model-independent, state that the mean transition time equals 
$[1+o_\beta(1)]\,K^\star\,e^{\beta\Gamma^\star}$, with $\beta$ the inverse temperature, 
and that the gate for the transition is $\cC^\star$, where $(\Gamma^\star,\cC^\star,K^\star)$ 
is a model-dependent triple. The theorems also show that the spectral gap of the generator 
of the dynamics scales like the inverse of the mean transition time and that the transition 
time divided by its mean is exponentially distributed asymptotically. In Section~\ref{S1.3} 
we recall that the prefactor $K^\star$ is given by a variational formula. In Section~\ref{S1.4} 
we define the hierarchical lattice. In Section~\ref{S1.5} we verify the key hypothesis for 
Glauber spin-flip dynamics on the hierarchical lattice and state \emph{five assumptions} 
on the interaction parameters. In Section~\ref{S1.6} we state our main theorems, which 
identify the triple $(\Gamma^\star,\cC^\star,K^\star)$ for the hierarchical lattice subject to 
these assumptions. In Section~\ref{S1.7} we close with a discussion and point to related 
literature. The proofs of the main theorems are given in Sections~\ref{S2}--\ref{S4}. The 
framework that is recalled in Sections~\ref{S1.1}--\ref{S1.3} is taken from Bovier and den 
Hollander~\cite[Chapter 16]{BdH15}.


\subsection{Ising model and Glauber spin-flip dynamics}
\label{S1.1}

Given a finite connected graph $G=\left(V,E\right)$, let $\Omega=\{-1,+1\}^{V}$ be 
the set of configurations $\sigma=\left\{ \sigma(v)\colon\, v\in V\right\}$ that assigns to each 
vertex $v\in V$ a spin-value $\sigma(v)\in\{-1,+1\}$. Two configurations that will be of 
particular interest to us are those where all spins point up, respectively, down: 
\begin{equation}
\boxplus\equiv+1,\qquad \boxminus\equiv-1.
\end{equation}
For $\beta \geq 0$, playing the role of \emph{inverse temperature}, we define the 
Gibbs measure 
\begin{equation}
\mu_{\beta}\left(\sigma\right) = \frac{1}{Z_{\beta}}\, e^{-\beta\cH\left(\sigma\right)},
\qquad \sigma\in\Omega,
\label{eq:Gibbs}
\end{equation}
where $\cH\colon\,\Omega\to\mathbb{R}$ is the \emph{Hamiltonian} that assigns an 
energy to each configuration given by
\begin{equation}
\cH\left(\sigma\right) = -\frac{1}{2} \sum_{(v,w)\in E} J_{\left(v,w\right)}\,
\sigma(v)\sigma(w)
-\frac{h}{2} \sum_{v\in V} \sigma(v),
\qquad \sigma\in\Omega,
\label{eq:hamiltonian}
\end{equation}
where $\vec{J}=\{J_e\} _{e \in E}$ is the \emph{ferromagnetic pair potential} acting along 
edges, satisfying $J_e \geq 0$ for all $e\in E$, and $h>0$ is the \emph{external magnetic 
field}. 

For two configurations $\sigma,\eta\in\Omega$, we write $\sigma\sim\eta$ when $\sigma$ and $\eta$ 
agree at all but one vertex. A transition from $\sigma$ to $\eta$ corresponds to a flip of a single
spin, and is referred to as an \emph{allowed move}. Glauber spin-flip dynamics on $\Omega$ 
is the continuous-time Markov process $(\sigma_t)_{t \geq 0}$ defined by the transition rates 
\begin{equation}
c_{\beta}\left(\sigma,\eta\right)=\begin{cases}
e^{-\beta[\cH\left(\eta\right)-\cH\left(\sigma\right)]_{+}}, & \sigma\sim\eta,\\
0, & \mbox{otherwise}.
\end{cases}
\end{equation}
The Gibbs measure in $\eqref{eq:Gibbs}$ is the \emph{reversible equilibrium} of this dynamics. 
We write $P_\sigma^{G,\beta}$ to denote the law of $(\sigma_t)_{t \geq 0}$ given $\sigma_{0}=\sigma$, 
$\cL^{G,\beta}$ to denote the associated generator, and $\lambda^{G,\beta}$ to denote the 
principal eigenvalue of $\cL^{G,\beta}$. The upper indices $G,\beta$ exhibit the dependence 
on the underlying graph $G$ and the interaction strength $\beta$ between neighbouring spins. 
For $A\subseteq\Omega$, we write 
\begin{equation}
\tau_A = \inf\big\{t>0\colon\,\sigma_t \in A,\,\exists\,0<s<t\colon\,\sigma_s \neq \sigma_0\big\}
\end{equation}
to denote the first hitting time of the set $A$ after the starting configuration is left.


\subsection{Metastability}
\label{S1.2}

To describe the metastable behaviour of our dynamics we need the following 
geometric definitions.

\begin{definition}
{\rm (a)} The communication height between two distinct configurations $\sigma,\eta\in\Omega$ is 
\begin{equation}
\Phi(\sigma,\eta) = \min_{\gamma\colon\,\sigma\to\eta}\max_{\xi\in\gamma}\cH(\xi),
\end{equation}
where the minimum is taken over all paths $\gamma\colon\,\sigma\to\eta$ consisting of 
allowed moves only. The communication height between two non-empty disjoint sets 
$A,B\subset\Omega$ is 
\begin{equation}
\Phi(A,B) = \min_{\sigma\in A,\eta\in B} \Phi(\sigma,\eta).
\end{equation}
{\rm (b)} The stability level of $\sigma\in\Omega$ is 
\begin{equation}
V_{\sigma} = \min_{ {\eta\in\Omega:} \atop {\cH(\eta)<\cH(\sigma)}} \Phi(\sigma,\eta)-\cH(\sigma).
\label{eq:Stability}
\end{equation}
{\rm (c)} The set of stable configurations is 
\begin{equation}
\Omega_{\stab} = \left\{\sigma\in\Omega\colon\,\cH(\sigma) = \min_{\eta\in\Omega}\cH(\eta)\right\}.
\end{equation}
{\rm (d)} The set of metastable configurations is 
\begin{equation}
\Omega_{\meta} = \left\{\sigma\in\Omega\backslash\Omega_{\stab}\colon\, 
V_\sigma = \max_{\eta\in\Omega\backslash\Omega_{\stab}} V_\eta\right\}.
\end{equation}
\end{definition}

It is easy to check that $\Omega_\stab=\{\boxplus\}$ for all $G$ because $h>0$ and 
$J_e \geq 0$ for all $e \in E$. In general, $\Omega_\meta$ is not a singleton. In order 
to proceed, we need the following \emph{key hypothesis}:
\begin{equation}
\mathrm{(H)} \quad \Omega_\meta = \{\boxminus\}.
\label{eq:Hhyp}
\end{equation}
Hypothesis (H) states that $\left\lbrace \boxminus, \boxplus \right\rbrace$ is a metastable pair. The energy barrier between $\boxminus$ and $\boxplus$ is 
\begin{equation}
\Gamma^\star = \Phi(\boxminus,\boxplus)-\cH(\boxminus),
\label{eq:defGammastar}
\end{equation}
which is a key quantity for the description of the metastable behaviour of our dynamics. 
We will say that a path $\gamma\colon\,\boxminus\to\boxplus$ is an \emph{optimal path}
when 
\begin{equation}
\max_{\eta\in\gamma} \cH(\eta) 
= \Phi\left(\boxminus,\boxplus\right). 
\label{eq:optimal-path}
\end{equation}

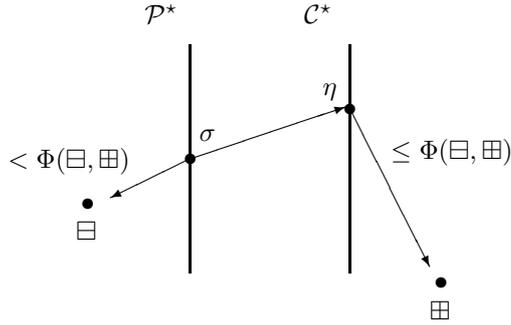
\begin{figure}[htbp]
\vspace{-0.5cm}
\begin{center}
\setlength{\unitlength}{0.3cm}
\begin{picture}(15,15)(0,-1)
{\thicklines
\qbezier(0,0)(0,5)(0,10)
\qbezier(7,0)(7,5)(7,10) 
}
\put(-2,11){$\cP^{\star}$}
\put(5,11){$\cC^\star$}
\put(0.4,5.8){$\sigma$}
\put(5.8,7.8){$\eta$}
\put(-5,1.4){$\boxminus$}
\put(10.5,-2.1){$\boxplus$} 
\put(-8,4.6){$<\Phi(\boxminus,\boxplus)$}
\put(8.8,5){$\leq \Phi(\boxminus,\boxplus)$} 
\put(0,5){\circle*{.45}}
\put(7,7.2){\circle*{.45}}
\put(-4.5,3){\circle*{.45}}
\put(11,-.5){\circle*{.45}}
\put(0,5){\vector(3,1){6.8}}
\put(7,7.2){\vector(1,-2){3.5}}
\put(0,5){\vector(-2,-1){3.5}}
\end{picture}
\end{center}
\vspace{0.1cm}
\caption{\small Schematic picture of the protocritical set and the critical set.}
\label{fig-protocrcr}
\end{figure}

\begin{definition}
\label{prcrcrdef}
Let $(\cP^{\star},\cC^{\star})$ be the unique maximal subset of $\Omega\times\Omega$
with the following properties (see Fig.~{\rm \ref{fig-protocrcr}}):
\begin{enumerate} 
\item 
$\forall\,\sigma\in\cP^\star\,\exists\,\eta\in\cC^\star\colon\,\sigma\sim\eta$,\\ 
$\forall\,\eta\in\cC^\star\,\exists\,\sigma\in\cP^\star\colon\,\eta\sim\sigma$. 
\item 
$\forall\,\sigma\in\cP^\star\colon\,\Phi(\sigma,\boxminus)<\Phi(\sigma,\boxplus)$. 
\item 
$\forall\sigma\,\in\cC^\star\,\exists\,\gamma\colon\,\sigma\to\boxplus\colon\,\\ 
{\rm (i)} \max_{\eta\in\gamma} \cH(\eta) \leq \Phi(\boxminus,\boxplus)$.\\ 
{\rm (ii)} $\gamma \cap\{\eta\in\Omega\colon\,\Phi(\eta,\boxminus)
<\Phi(\eta,\boxplus)\} = \emptyset$. 
\end{enumerate}
\end{definition}

\noindent
Think of $\cP^{\star}$ as the set of configurations where the dynamics, on its way from 
$\boxminus$ to $\boxplus$, is `almost at the top', and of $\cC^{\star}$ as the set of 
configurations where it is `at the top and capable of crossing over'. We refer to 
$\cP^{\star}$ as the \emph{protocritical set} and to $\cC^{\star}$ as the \emph{critical
set}. Uniqueness follows from the observation that if $(\cP_{1}^{\star},\cC_{1}^{\star})$
and $(\cP_{2}^{\star},\cC_{2}^{\star})$ both satisfy conditions (1)--(3), then so does 
$(\cP_{1}^{\star}\cup\cP_{2}^{\star},\cC_{1}^{\star}\cup\cC_{2}^{\star})$. Note that 
\begin{equation} 
\begin{array}{lll} 
&\cH(\sigma)<\Phi(\boxminus,\boxplus) 
&\forall\,\sigma \in \cP^\star,\\[0.1cm] 
&\cH(\sigma)=\Phi(\boxminus,\boxplus) 
&\forall\,\sigma \in \cC^\star. 
\end{array}
\end{equation}

It is shown in Bovier and den Hollander \cite[Chapter 16]{BdH15}
that \emph{subject to hypothesis} (H) the following three theorems
hold.

\begin{theorem} 
\label{thm:critdrop} 
$\lim_{\beta\to\infty} P^{G,\beta}_\boxminus(\tau_{\cC^\star}<\tau_\boxplus 
\mid \tau_\boxplus < \tau_\boxminus)=1$. 
\end{theorem}

\begin{theorem} 
\label{thm:nucltime} 
There exists a $K^\star \in (0,\infty)$ such that 
\begin{equation} 
\lim_{\beta\to\infty} e^{-\beta\Gamma^\star}\,E^{G,\beta}_\boxminus(\tau_\boxplus) 
= K^\star. 
\end{equation} 
\end{theorem}

\begin{theorem} 
\label{thm:explaw} 
{\rm (a)} $\lim_{\beta\to\infty} \lambda^{G,\beta}\,E^{G,\beta}_\boxminus(\tau_\boxplus)=1$.\\ 
{\rm (b)} $\lim_{\beta\to\infty} P^{G,\beta}_\boxminus(\tau_\boxplus/ E^{G,\beta}_\boxminus
(\tau_\boxplus)>t) = e^{-t}$ for all $t \geq 0$. 
\end{theorem}

\begin{figure}[htbp]
\vspace{1.3cm}
\begin{center}
\setlength{\unitlength}{0.4cm}
\begin{picture}(8,6)(4,0)
\put(0,0){\line(11,0){11}}
\put(0,0){\line(0,9){9}}
\qbezier[25](3.3,3.8)(3.3,1.9)(3.3,0)
\qbezier[30](4.8,5.6)(4.8,3)(4.8,0)
\qbezier[35](6.8,3.0)(6.8,1.5)(6.8,0)
\qbezier[30](0,5.7)(2.4,5.7)(4.8,5.7)
\qbezier[25](0,3.9)(1.65,3.9)(3.3,3.9)
{\thicklines
\qbezier(2,8)(3,2)(4,5)
\qbezier(4,5)(5,7)(6,4)
\qbezier(6,4)(7,2)(8,5)
}
\put(-.3,4.8){\vector(0,1){.9}}
\put(-.3,4.8){\vector(0,-1){.9}}
\put(-1.6,4.4){$\Gamma^\star$}
\put(2.9,-1){$\boxminus$}
\put(4.55,-1){$\cC^\star$}
\put(6.5,-1){$\boxplus$}
\put(11.5,-.3){$\sigma$}
\put(-1.3,9.5){$\cH(\sigma)$}
\put(3.3,4){\circle*{.35}}
\put(4.8,5.75){\circle*{.35}}
\put(6.8,3.20){\circle*{.35}} 
\end{picture}
\end{center}
\vspace{0.3cm}
\caption{\small Schematic picture of $\cH$, $\boxminus$, $\boxplus$, $\Gamma^\star$ 
and $\cC^\star$. Lemma~\ref{lem:variational-lemma} shows that $1/K^\star$ is in essence 
proportional to $\vert\cC^\star\vert$.}
\label{fig-abstractdoublewell}
\vspace{0.cm}
\end{figure}
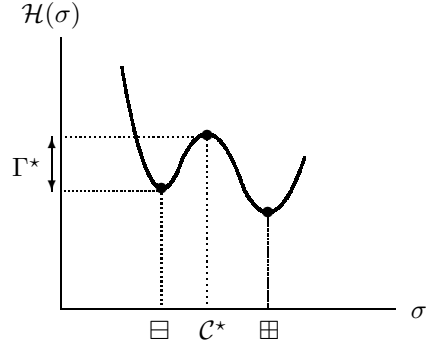

\noindent
The proofs of Theorems~\ref{thm:critdrop}--\ref{thm:explaw} in \cite{BdH15} do not rely 
on the details of the graph $G$, provided it is finite, connected and non-oriented. For 
concrete choices of $G$, the task is to verify hypothesis (H) and to identify the triple 
\begin{equation} 
\label{eq:triple} 
\big(\Gamma^\star,\cC^\star,K^\star\big). 
\end{equation} 
A schematic picture of the role of these quantities is given in Fig.~\ref{fig-abstractdoublewell}.


\subsection{Variational formula for the prefactor}
\label{S1.3}

The prefactor $K^{\star}$ in Theorem~\ref{thm:nucltime} is given by a variational formula
(see \cite[Lemma 16.17]{BdH15}):
\begin{equation}
\frac{1}{K^\star} = \min_{C_1,\ldots,C_I} \min_{ {f\colon\,S^\star\to [0,1]:} \atop 
{f|_{S_\boxminus} \equiv 1,\, f|_{S_\boxplus} \equiv 0,\, f|_{S_k}=C_k} }
\frac{1}{2}\sum_{\sigma,\eta\in S^\star}\,
\mathbf{1}_{\{\sigma\sim\eta\}}\,
\left[f\left(\sigma\right)-f\left(\eta\right)\right]^{2}.
\label{eq:variationalform}
\end{equation}
Here, $\left\{S_k\right\} _{k=1}^I$ is the unique sequence of maximally connected 
disjoint sets $S_k\subseteq \Omega$ defined by
\begin{equation}
\sigma \in S_k \quad \Longleftrightarrow \quad 
\cH\left(\sigma\right)<\Phi\left(\boxminus,\boxplus\right),
\,\Phi\left(\sigma,\boxminus\right)=\Phi\left(\sigma,\boxplus\right)
=\Phi\left(\boxminus,\boxplus\right).
\label{eq:wells}
\end{equation}
Think of $\left\{S_k\right\} _{k=1}^I$ as `wells at the top' (see Fig.~\ref{fig-wells}). 
The sets $S_\boxminus,S_\boxplus$ are defined by 
\begin{equation}
\begin{aligned}
S_\boxminus &= \left\{\sigma\in\Omega\colon\,
\Phi\left(\sigma,\boxminus\right)<\Phi\left(\boxminus,\boxplus\right)\right\},\\
S_\boxplus &= \left\{\sigma\in\Omega\colon\,
\Phi\left(\sigma,\boxplus\right)<\Phi\left(\boxminus,\boxplus\right)\right\},
\end{aligned}
\end{equation}
and are to be thought of as the `valleys' around $\boxminus$ and $\boxplus$. 
The set $S^\star$ is defined by
\begin{equation}
S^\star = \left\{\sigma\in\Omega\colon\,\Phi\left(\sigma,\boxminus\right) \vee 
\Phi\left(\sigma,\boxplus\right) \leq  \Phi\left(\boxminus,\boxplus\right)\right\}, 
\end{equation}
i.e., the maximally connected set with energy $\leq \Phi(\boxminus,\boxplus)$
containing $\boxminus$ and $\boxplus$. Note that $\left\{S_k\right\} _{k=1}^I$,
$S_\boxminus, S_\boxplus \subseteq S^\star$. 

\begin{figure}[htbp]
\vspace{0.5cm}
\begin{center}
\setlength{\unitlength}{0.25cm}
\begin{picture}(15,5)(-3,-1)
{\thicklines
\qbezier(0,0)(0.5,5)(1,5)
\qbezier(1,5)(2,5)(3,5)
\qbezier(3,5)(3.5,4)(4,5)
\qbezier(4,5)(5,5)(6,5)
\qbezier(6,5)(6.5,4)(7,5)
\qbezier(7,5)(8,5)(9,5) 
\qbezier(9,5)(9.5,0)(10,-2)
}
\put(-.4,-1.2){$\boxminus$}
\put(9.6,-3.2){$\boxplus$}
\put(-1.9,2){$S_{\boxminus}$}
\put(10,2){$S_{\boxplus}$}
\put(0.8,5){\circle*{0.6}}
\put(0.5,5.7){$\cC^{\star}$}
\put(2.8,5.7){$S_1$}
\end{picture}
\vspace{.5cm}
\end{center}
\caption{\small Schematic picture of the wells $\{S_k\}_{k=1}^I$. Note that $\cC^{\star}
\subseteq S \backslash (S_\boxminus \cup S_\boxplus)$.}
\label{fig-wells}
\end{figure}
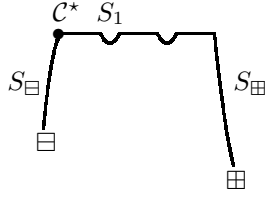

The variational problem in \eqref{eq:variationalform} has the interpretation 
of the \emph{capacity} between $S_\boxminus$ and $S_\boxplus$ for 
\emph{simple random walk} on $S^\star$ jumping at rate 1 after the sets 
$\left\{S_k\right\} _{k=1}^I,S_\boxminus,S_\boxplus$ are wired. If we impose 
\emph{additional constraints} on the optimal paths and their behaviour near 
the set $\cC^{\star}$, then \eqref{eq:variationalform} simplifies considerably, 
as is shown in the following lemma.

\begin{lemma}
\label{lem:variational-lemma}
Suppose that there exists a $k^{\star}\in\mathbb{N}$ such that the following are 
true:\\ 
(i) $\cC^\star=\lbrace \sigma \in S^{\star}\colon \vert \sigma\vert = k^{\star}\rbrace$.\\ 
(ii) For all $\sigma\in\cC^{\star}$ the sets 
\begin{equation}
\label{Udefs}
\begin{aligned}
&U_{\sigma}^{-} = \left\{ \eta \in S^{\star}\colon\,\eta\sim\sigma,\,\left|\eta\right|
= \left|\sigma\right|-1\right\},\\
&U_{\sigma}^{-} = \left\{ \eta \in S^{\star}\colon\,\eta\sim\sigma,\,\left|\eta\right|
= \left|\sigma\right|+1\right\},
\end{aligned} 
\end{equation}
satisfy 
\begin{equation}
\Phi\left(\eta,\boxminus\right)<\Phi\left(\boxminus,\boxplus\right) 
\quad \forall\,\eta\in U_{\sigma}^{-},
\qquad \Phi\left(\eta,\boxplus\right)<\Phi\left(\boxminus,\boxplus\right)  
\quad \forall\,\eta\in U_{\sigma}^{+}.
\label{eq:condition-lemma}
\end{equation}
Then \eqref{eq:variationalform} simplifies to 
\begin{equation}
\frac{1}{K^\star} = \sum_{\sigma\in\cC^{\star}}
\frac{|U_{\sigma}^{-}|\,|U_{\sigma}^{+}|}{|U_{\sigma}^{-}|+|U_{\sigma}^{+}|}.
\label{eq:reduced-varform-0}
\end{equation}
\end{lemma}

\begin{proof}
The proof is analogous to that in \cite[Section 17.5]{BdH15}. The variational problem 
in \eqref{eq:variationalform} simplifies because of the following two facts that are specific 
to Glauber dynamics:
\begin{itemize}
\item
$S^\star\backslash [S_\boxminus \cup S_\boxplus]
= \cC^\star$, i.e.,
there are no wells inside $\cC^\star$.
\item 
There are no allowed moves within $\cC^\star$, i.e., critical configurations cannot 
transform into each other via single spin-flips.
\end{itemize} 
Consequently, \eqref{eq:variationalform} reduces to
\begin{equation}
\label{varfor1}
\frac{1}{K^\star} = \min_{h\colon\,\cC^\star\to [0,1]} \sum_{\sigma \in \cC^\star}
[1-h(\sigma)]\, 2|U^-_\sigma| + [h(\sigma)]\, 2|U^+_\sigma|, 
\end{equation}
where $U^-_\sigma$ and $U^+_\sigma$ consist of the configurations in $S_\boxminus$ 
and $S_\boxplus$, respectively, that can reached from $\sigma \in \cC^\star$ by a 
single spin-flip. The solution of \eqref{varfor1} is computed easily to obtain 
\eqref{eq:reduced-varform-0}
\end{proof}

\begin{remark}
\label{rem:prefactor-lemma} 
{\rm An immediate consequence of the additional assumptions in 
Lemma~\ref{lem:variational-lemma} is that $I = 0$ (`no wells at the top') 
and that all configurations in $S^{\star}$ that are neighbours of configurations in $\cC^{\star}$ 
have an energy that is strictly below $\Phi(\boxminus,\boxplus)$ (`the top is not flat'). 
Consequently, only transitions from $\cC^\star$ to $S_\boxminus$ and $S_\boxplus$ 
(`down from the top') contribute to the prefactor (see Fig.~\ref{fig-wellsalt}).}
\end{remark}

\begin{figure}[htbp]
\vspace{0.5cm}
\begin{center}
\setlength{\unitlength}{0.25cm}
\begin{picture}(15,5)(-3,-1)
\qbezier[20](0,0)(0.5,5)(1,5)
\qbezier[30](1,5)(5.5,0)(6,-2)
\put(-.5,-1.4){$\boxminus$}
\put(5.7,-3.3){$\boxplus$}
\put(-2.1,2){$S_\boxminus$}
\put(4.3,2){$S_\boxplus$}
\put(0.8,5){\circle*{0.3}}
\put(0.5,5.7){$\cC^{\star}$}
\end{picture}
\vspace{.5cm}
\end{center}
\caption{\small Configurations in $\cC^{\star}$ are strict maxima in the energy profile 
of an optimal path. No plateau or wells are present.}
\label{fig-wellsalt}
\end{figure}
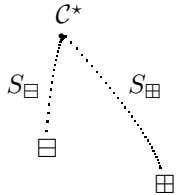

\subsection{The hierarchical lattice}
\label{S1.4}

Let $N\in\mathbb{N}\backslash\{1\}$, and define the \emph{$N$-dimensional 
hierarchical lattice} $\Lambda_{N}$ to be the metric space $\left(\mathbb{N},d\right)$ 
with $\N$ the set of positive integers and $d$ the ultrametric defined by 
\begin{equation}
d\left(a,b\right)
=\max\left\{ k\in\N_0\colon\, a\,\mathrm{mod}\,N^{k} \neq 
b\,\mathrm{mod}\,N^{k}\right\},\quad a,b\in \mathbb{N}, 
\label{eq:metric-d}
\end{equation}
which is called the \emph{hierarchical distance}. We say that $A\subseteq\mathbb{N}$ is a 
$k$-block of $\Lambda_{N}$ when $|A|=N^{k}$ and $d\left(a,b\right)\leq k$ for all $a,b\in A$. 
In particular, we define $\Lambda_{N}^{n}$ to be the $n$-block
\begin{equation}
\label{orderdef}
\Lambda_{N}^{n}=\left\{1,2,\ldots,N^{n}\right\}, 
\end{equation}
which is the $N$-dimensional hierarchical lattice with $n$ \emph{hierarchical levels} (see 
Fig.~\ref{fig:boxes}).

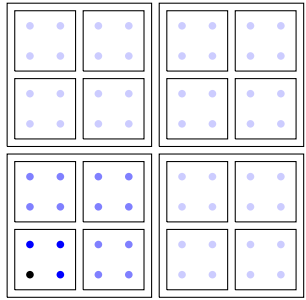
\begin{figure}[htbp]
\vspace{0.2cm}
\begin{tikzpicture} [scale=0.5]
\draw (0.1,0.1) -- (3.9,0.1) -- (3.9,3.9) -- (0.1,3.9) -- (0.1,0.1); 		
\draw (0.3,0.3) -- (1.9,0.3) -- (1.9,1.9) -- (0.3,1.9) -- (0.3,0.3); 		
\draw (2.1,0.3) -- (3.7,0.3) -- (3.7,1.9) -- (2.1,1.9) -- (2.1,0.3); 		
\draw (2.1,2.1) -- (3.7,2.1) -- (3.7,3.7) -- (2.1,3.7) -- (2.1,2.1); 		
\draw (0.3,2.1) -- (1.9,2.1) -- (1.9,3.7) -- (0.3,3.7) -- (0.3,2.1); 	
\draw (4.1,0.1) -- (7.9,0.1) -- (7.9,3.9) -- (4.1,3.9) -- (4.1,0.1); 		
\draw (0.3+4,0.3) -- (1.9+4,0.3) -- (1.9+4,1.9) -- (0.3+4,1.9) -- (0.3+4,0.3); 		
\draw (2.1+4,0.3) -- (3.7+4,0.3) -- (3.7+4,1.9) -- (2.1+4,1.9) -- (2.1+4,0.3); 		
\draw (2.1+4,2.1) -- (3.7+4,2.1) -- (3.7+4,3.7) -- (2.1+4,3.7) -- (2.1+4,2.1); 		
\draw (0.3+4,2.1) -- (1.9+4,2.1) -- (1.9+4,3.7) -- (0.3+4,3.7) -- (0.3+4,2.1); 	
\draw (0.1,4.1) -- (3.9,4.1) -- (3.9,7.9) -- (0.1,7.9) -- (0.1,4.1); 		
\draw (0.3,0.3+4) -- (1.9,0.3+4) -- (1.9,1.9+4) -- (0.3,1.9+4) -- (0.3,0.3+4); 		
\draw (2.1,0.3+4) -- (3.7,0.3+4) -- (3.7,1.9+4) -- (2.1,1.9+4) -- (2.1,0.3+4); 		
\draw (2.1,2.1+4) -- (3.7,2.1+4) -- (3.7,3.7+4) -- (2.1,3.7+4) -- (2.1,2.1+4); 		
\draw (0.3,2.1+4) -- (1.9,2.1+4) -- (1.9,3.7+4) -- (0.3,3.7+4) -- (0.3,2.1+4); 	
\draw (4.1,4.1) -- (7.9,4.1) -- (7.9,7.9) -- (4.1,7.9) -- (4.1,4.1); 		
\draw (0.3+4,0.3+4) -- (1.9+4,0.3+4) -- (1.9+4,1.9+4) -- (0.3+4,1.9+4) -- (0.3+4,0.3+4); 		
\draw (2.1+4,0.3+4) -- (3.7+4,0.3+4) -- (3.7+4,1.9+4) -- (2.1+4,1.9+4) -- (2.1+4,0.3+4); 		
\draw (2.1+4,2.1+4) -- (3.7+4,2.1+4) -- (3.7+4,3.7+4) -- (2.1+4,3.7+4) -- (2.1+4,2.1+4); 		
\draw (0.3+4,2.1+4) -- (1.9+4,2.1+4) -- (1.9+4,3.7+4) -- (0.3+4,3.7+4) -- (0.3+4,2.1+4);
\fill[black!] (0.7,0.7) circle (0.1cm); 	
\fill[blue!100] (1.5,0.7) circle (0.1cm); 	
\fill[blue!100] (1.5,1.5) circle (0.1cm); 	
\fill[blue!100] (0.7,1.5) circle (0.1cm);
\fill[blue!50] (0.7+1.8,0.7) circle (0.1cm); 	
\fill[blue!50] (1.5+1.8,0.7) circle (0.1cm); 	
\fill[blue!50] (1.5+1.8,1.5) circle (0.1cm); 	
\fill[blue!50] (0.7+1.8,1.5) circle (0.1cm);
\fill[blue!50] (0.7,0.7+1.8) circle (0.1cm); 	
\fill[blue!50] (1.5,0.7+1.8) circle (0.1cm); 	
\fill[blue!50] (1.5,1.5+1.8) circle (0.1cm); 	
\fill[blue!50] (0.7,1.5+1.8) circle (0.1cm);
\fill[blue!50] (0.7+1.8,0.7+1.8) circle (0.1cm); 	
\fill[blue!50] (1.5+1.8,0.7+1.8) circle (0.1cm); 	
\fill[blue!50] (1.5+1.8,1.5+1.8) circle (0.1cm); 	
\fill[blue!50] (0.7+1.8,1.5+1.8) circle (0.1cm);
\fill[blue!20] (0.7+4,0.7) circle (0.1cm); 	
\fill[blue!20] (1.5+4,0.7) circle (0.1cm); 	
\fill[blue!20] (1.5+4,1.5) circle (0.1cm); 	
\fill[blue!20] (0.7+4,1.5) circle (0.1cm);
\fill[blue!20] (0.7+1.8+4,0.7) circle (0.1cm); 	
\fill[blue!20] (1.5+1.8+4,0.7) circle (0.1cm); 	
\fill[blue!20] (1.5+1.8+4,1.5) circle (0.1cm); 	
\fill[blue!20] (0.7+1.8+4,1.5) circle (0.1cm);
\fill[blue!20] (0.7+4,0.7+1.8) circle (0.1cm); 	
\fill[blue!20] (1.5+4,0.7+1.8) circle (0.1cm); 	
\fill[blue!20] (1.5+4,1.5+1.8) circle (0.1cm); 	
\fill[blue!20] (0.7+4,1.5+1.8) circle (0.1cm);
\fill[blue!20] (0.7+1.8+4,0.7+1.8) circle (0.1cm); 
\fill[blue!20] (1.5+1.8+4,0.7+1.8) circle (0.1cm); 	
\fill[blue!20] (1.5+1.8+4,1.5+1.8) circle (0.1cm); 
\fill[blue!20] (0.7+1.8+4,1.5+1.8) circle (0.1cm);
\fill[blue!20] (0.7,0.7+4) circle (0.1cm); 	
\fill[blue!20] (1.5,0.7+4) circle (0.1cm); 	
\fill[blue!20] (1.5,1.5+4) circle (0.1cm); 	
\fill[blue!20] (0.7,1.5+4) circle (0.1cm);
\fill[blue!20] (0.7+1.8,0.7+4) circle (0.1cm); 	
\fill[blue!20] (1.5+1.8,0.7+4) circle (0.1cm); 	
\fill[blue!20] (1.5+1.8,1.5+4) circle (0.1cm); 	
\fill[blue!20] (0.7+1.8,1.5+4) circle (0.1cm);
\fill[blue!20] (0.7,0.7+1.8+4) circle (0.1cm); 	
\fill[blue!20] (1.5,0.7+1.8+4) circle (0.1cm); 	
\fill[blue!20] (1.5,1.5+1.8+4) circle (0.1cm); 	
\fill[blue!20] (0.7,1.5+1.8+4) circle (0.1cm);
\fill[blue!20] (0.7+1.8,0.7+1.8+4) circle (0.1cm); 	
\fill[blue!20] (1.5+1.8,0.7+1.8+4) circle (0.1cm); 	
\fill[blue!20] (1.5+1.8,1.5+1.8+4) circle (0.1cm); 
\fill[blue!20] (0.7+1.8,1.5+1.8+4) circle (0.1cm);
\fill[blue!20] (0.7+4,0.7+4) circle (0.1cm); 	
\fill[blue!20] (1.5+4,0.7+4) circle (0.1cm); 	
\fill[blue!20] (1.5+4,1.5+4) circle (0.1cm); 	
\fill[blue!20] (0.7+4,1.5+4) circle (0.1cm);
\fill[blue!20] (0.7+1.8+4,0.7+4) circle (0.1cm); 	
\fill[blue!20] (1.5+1.8+4,0.7+4) circle (0.1cm); 	
\fill[blue!20] (1.5+1.8+4,1.5+4) circle (0.1cm); 	
\fill[blue!20] (0.7+1.8+4,1.5+4) circle (0.1cm);
\fill[blue!20] (0.7+4,0.7+1.8+4) circle (0.1cm); 	
\fill[blue!20] (1.5+4,0.7+1.8+4) circle (0.1cm); 	
\fill[blue!20] (1.5+4,1.5+1.8+4) circle (0.1cm); 	
\fill[blue!20] (0.7+4,1.5+1.8+4) circle (0.1cm);
\fill[blue!20] (0.7+1.8+4,0.7+1.8+4) circle (0.1cm); 	
\fill[blue!20] (1.5+1.8+4,0.7+1.8+4) circle (0.1cm); 	
\fill[blue!20] (1.5+1.8+4,1.5+1.8+4) circle (0.1cm); 	
\fill[blue!20] (0.7+1.8+4,1.5+1.8+4) circle (0.1cm);
\end{tikzpicture}
\caption{\small Schematic representation of $\Lambda^{3}_{4}$. The distance from 
the vertex in the lower-left corner to any vertex in the lower-left 1-block different from 
that vertex equals 1, to any vertex in the lower-left 2-block that is not in the lower-left 
1-block equals 2, and to any vertex in the lower-left 3-block that is not in the lower-left 
2-block equals 3. Note that, with this interpretation, for any two vertices $v$ and $w$ 
the size of the smallest box containing both $v$ and $w$ is $N^{d\left(v,w\right)}$.}
\label{fig:boxes}
\end{figure}

The set $\Lambda_{N}^{n}$ is the underlying graph from which we build our state space 
$\Omega=\{-1,+1\}^{\Lambda_{N}^{n}}$.  We may alternatively write $\Lambda^{n}_{N}
=\left\lbrace v_{1},\ldots,v_{N^n}\right\rbrace$ with $v_{a}$ the vertex corresponding to 
the integer $a$. Note that $d(v_{a},v_{b})=d(a,b)$. We define $\gamma\colon\,\boxminus
\to\boxplus$ to be the path $\gamma=\left(\gamma_{0},\ldots,\gamma_{N^{n}}\right)$, 
where $\gamma_{k}$ is the configuration with $\gamma_{k}\left(v_a\right)=+1$ for $a\leq k$ 
and $\gamma_{k}\left(v_a\right)=-1$ for $a>k$, i.e., spins are flipped upward in the order 
in which they are labelled. We refer to $\gamma$ as the \emph{reference path}, and it will 
play a crucial role in our analysis.

Whenever convenient, we may think of $\Omega$ as the power set of $\Lambda_{N}^n$ 
and of configurations $\sigma\in\Omega$ as subsets of $\Lambda_{N}^{n}$. Thus, 
we may identify a configuration $\sigma\in\{-1,+1\}^{\Lambda_{N}^{n}}$ with the set 
$\{v\in\Lambda_{N}^{n}\colon\,\sigma(v)=$ $+1\}$ and its flipped image 
$\overline{\sigma}$ with the set $\{ v\in\Lambda_{N}^{n}\colon\,\sigma(v)=-1\}$. 

To define the interaction, we make $\Lambda_{N}^{n}$ into a complete graph by 
placing an edge between all pairs $v,w\in\Lambda_{N}^{n}$ with $v\neq w$. The 
ferromagnetic pair potential between such pairs equals $J_{d\left(v,w\right)}$, where
\begin{equation}
\vec{J}=\left\{J_i\right\} _{i=1}^{n}
\end{equation} 
is chosen such that $J_i > 0$ for $1\leq i\leq n$. Hence the Hamiltonian in \eqref{eq:hamiltonian} 
becomes 
\begin{equation}
\cH\left(\sigma\right)=-\frac{1}{2}\sum_{ {v,w\in\Lambda_N^n:} \atop {v \neq w} }
J_{d\left(v,w\right)}\,\sigma(v)\sigma(w)
-\frac{h}{2}\sum_{v\in\Lambda_N^n} \sigma(v).
\label{eq:Hamiltonian}
\end{equation}


\subsection{Hypothesis and Assumptions}
\label{S1.5}

We want to apply the theory behind Theorems~\ref{thm:critdrop}--\ref{thm:explaw}, 
for which we need to verify Hypothesis (H) in \eqref{eq:Hhyp}. In the sequel we will need
five assumptions on the interaction parameters of our model.

\medskip\noindent
{\bf Assumption (A1):}
 \begin{eqnarray}
\left(1-\frac{1}{N}\right) \sum_{i=1}^{n}J_{i}N^{i} >h. 
\label{eq:A0-MetastableRegime}
\end{eqnarray}
(A1) guarantees that $\boxminus$ is a local minimum and corrresponds to the range of
parameters for which the system is in the \emph{metastable regime}. 

\begin{theorem}
\label{thrm:hyp-H} 
Suppose that $\vec{J}$ is monotone, i.e. either non-increasing or non-decreasing, and that {\rm (A1)} holds. Then hypothesis {\rm (H)} is verified. 
\end{theorem}

\noindent
We will see from the proof of Theorem ~\ref{thrm:hyp-H} that without (A1) there are no local minima in the energy landscape.

Our main task is to identify the triplet $\left(\Gamma^{\star},\cC^{\star},K^{\star}\right)$
in \eqref{eq:triple}. To do so, we require \emph{four assumptions} on $\vec{J}$, which 
we list below.

\medskip\noindent
{\bf Assumption (A2):}
\begin{eqnarray}
&\mathrm{(a)} & \exists\,\delta>0,\, M\in\mathbb{N}\colon\,\quad
1-\delta\geq\left\lceil \hat{s}\right\rceil -\hat{s} \geq \delta\quad \forall\,N\geq M,
\label{eq:A1noninteger}\\
&\mathrm{(b)} & \liminf_{N\to\infty}
\left|\sum_{i=\hat{m}+1}^{n}J_{i}N^{i}-h\right|>0,\nonumber 
\end{eqnarray}
where
\begin{eqnarray}
\hat{m} &=& \max\left\{ 0\leq m\leq n-1\colon\,\left(1-\frac{1}{N}\right)
\sum_{i=m+1}^{n}J_{i}N^{i}>h\right\},
\label{eq:defm}\\
\hat{s} &=& \frac{N}{2}(J_{\hat{m}+1}N^{\hat{m}+1})^{-1}\left[\left(1-\frac{1}{N}\right)
\sum_{i=\hat{m}+1}^{n}J_{i}N^{i}-h\right].
\label{eq:sineq}
\end{eqnarray}
(A2)(a) guarantees that $\hat{s}$ is not an integer when $N$ is sufficiently large, 
and does not approach an integer either as $N\to\infty$. (A2)(b) guarantees that the 
interaction is not `conspiring' to allow $|\sum_{i=\hat{m}+1}^{n} J_{i}N^{i}-h|$ to 
vanish as $N\to\infty$. Both assumptions are made to avoid certain degeneracies. 
These would not pose an essential problem, but would complicate our analysis 
unnecessarily. 

\medskip\noindent
{\bf Assumption (A3):}
\begin{equation}
\begin{aligned}
&\mbox{ For all } 1 \leq k \leq N^{\hat{m}} \mbox{ with } N\mbox{-ary 
decomposition } k=a_{\hat{m}-1} N^{\hat{m}-1}+\ldots+a_{0}\colon \\
&\lim_{N\to\infty} \frac{\sum\limits_{i=0}^{\hat{m}-1}J_{i+1}N^{i}
\Big[(N-a_{i}-1)\Big(\sum\limits_{j=0}^{i} a_{j}N^{j}\Big)
+a_{i}\Big(N^{i}-\sum\limits_{j=0}^{i-1} a_{j}N^{j}\Big)\Big]
+k\sum\limits_{i=\hat{m}+1}^{n} J_{i}N^{i}}{\lceil \hat{s}\rceil 
(2\hat{s}-\lceil \hat{s}\rceil +1) J_{\hat{m}+1}N^{2\hat{m}}} = 0.
\label{eq:A2}
\end{aligned}
\end{equation}
This assumption has a somewhat unappealing form. Its purpose is to ensure that, in the 
limit as $N\to\infty$, the energy along optimal paths fluctuates by relatively small amounts 
over short distances. We will see that it is satisfied when $J_i =o(N^{-i+1})$ as $N\to\infty$. 

\medskip\noindent
{\bf Assumption (A4):}
\begin{equation}
\frac{J_{i+1}}{J_{i}} = O\left(\frac{1}{N}\right) \qquad \forall\,1 \leq i \leq \hat{m}.
\label{eq:A3decayrateofJi}
\end{equation}
This assumption guarantees that the total interaction between a given spin and all the 
spins at a given hierarchical level remains bounded as $N\to\infty$.

\medskip\noindent
{\bf Assumption (A5):}
\begin{equation}
\mbox{No linear combination of \ensuremath{J_1,\ldots,J_n} is
a multiple of \ensuremath{h}.}
\label{eq:AssumptionUniquemax}
\end{equation}
This assumption again avoids certain degeneracies, and is valid for all but countably 
many choices of $h$ and $\vec{J}$.


\subsection{Main theorems}
\label{S1.6}

We are now ready to state our main results. The seven theorems and two corollaries 
given below identify the triple in \eqref{eq:triple}, consisting of the communication 
height $\Gamma^\star$, the set of critical configurations $\cC^\star$ and the prefactor
$K^\star$. Formulas simplify as more constraints are placed on $\vec{J}$. 


\medskip\noindent
\paragraph{$\bullet$ {\bf Communication height}} 

Recall the definition of $\Gamma^\star$ in \eqref{eq:defGammastar}.

\begin{theorem}
\label{thm:Gamma-case1}
Suppose that $\vec{J}$ is non-increasing and that {\rm (A1)} and {\rm (A3)} hold. Then 
\begin{equation}
\Gamma^{\star} = \left[1+o_{N}\left(1\right)\right]\,\frac{1}{4}
\left(J_{\hat{m}+1}\right)^{-1}\left(\sum_{i=\hat{m}+1}^{n} J_{i}N^{i}-h\right)^{2}, \quad N \to \infty.
\label{eq:thrm-Gamma value}
\end{equation}
\end{theorem}

\begin{corollary}
\label{cor:Gamma-case2}
Suppose that $J_i=\tilde{J}_i/N^i$ with $\tilde{J}_i=o(N)$ and that 
{\rm (A2)(b)} holds. Then {\rm (A3)} holds and 
\begin{equation}
\Gamma^{\star} = \left[1+o_{N}\left(1\right)\right]\,\frac{1}{4}(\tilde{J}_{\hat{m}+1})^{-1}
\left(\sum_{i=\hat{m}+1}^{n}\tilde{J}_{i}-h\right)^{2}N^{\hat{m}+1}, \quad N \to \infty.
\end{equation}
\end{corollary}

Our next result gives a formula for $\Gamma^{\star}$ when $J_i=\tilde{J}/N^{i}$ for 
some $\tilde{J}>0$. Let
\begin{equation}
\mathbb{I} = \left\lbrace \left(m,s\right)\colon\, 0\leq m\leq n-1,\,1\leq s\leq N-1\right\rbrace 
\cup \left\lbrace\left(n-1,N\right)\right\rbrace \subseteq \N^2,
\label{eq:indexI}
\end{equation}
and for $\left(m,s\right)\in \mathbb{I}$ define
\begin{equation}
h^{\left(m,s\right)} = \tilde{J}\left[\left(1-\frac{1}{N}\right)\left(n-m\right)
-\left(s-1\right)\frac{1}{N}\right] \in \left[0,\tilde{J}\left(1-\frac{1}{N}\right)n\right].
\label{eq:hdagger}
\end{equation}

\begin{theorem}
\label{thm:Gamma-case3}
Suppose that $J_i=\tilde{J}/N^{i}$ for some $\tilde{J}>0$. Let $\left(m,s\right)\in \mathbb{I}$ 
be such that $h$ satisfies
\begin{equation}
h^{\left(m,s\right)}\leq h < h^{\left(m,s-1\right)}.
\end{equation}
{\rm (1)} If $N$ is odd, then
\begin{equation}
\begin{aligned}
\Gamma^{\star} 
&= \frac{\tilde{J}}{4N}\left[N^{m}\left(2s\left(N-\frac{s}{2}
+s\,\mathrm{mod}\,2\right)-N-s\,\mathrm{mod}\,2\right)+N-2s
-\left(-1\right)^{s\,\mathrm{mod}\,2}\right]\\
&\qquad +\frac{1}{2}\left[\tilde{J}\left(1-\frac{1}{N}\right)\left(n-m-1\right)-h\right]
\left(N^{m}\big(s-s\,\mathrm{mod}\,2\right)+1\big). 
\label{eq:thrm-Gamma-case3a}
\end{aligned}
\end{equation}
{\rm (2)} If $N$ is even and $s$ is odd, then
\begin{equation}
\Gamma^{\star} = \Gamma_{m,s}^{\star}
\end{equation}
with
\begin{eqnarray}
\Gamma_{m,s}^{\star} 
& = & \frac{\tilde{J}}{2}N^{-m\,\mathrm{mod}\,2} (A_m-1)
+ \tilde{J}\left[\frac{1}{2}B_{m}-N^{m\,\mathrm{mod}\,2} A_{m}\right](N-s)\\
& + & \tilde{J}\left[\frac{N}{4}B_{m}
- N^{m\,\mathrm{mod}\,2}A_{m}
+ N^{m-1}\left(\frac{s-1}{2}\right)\left(N-\frac{s-1}{2}\right)\right] \nonumber \\
& + & \left[\left(\frac{s-1}{2}\right)N^{m}+\frac{N}{2}B_{m}
-N^{1+m\,\mathrm{mod}\,2}A_{m}\right]
\left[\tilde{J}\left(1-\frac{1}{N}\right)\left(n-m-1\right)-h\right], \nonumber
\label{eq:thrm-Gamma-case3b}
\end{eqnarray}
where $A_{m}=\left(\frac{N^{m-m\mathrm{mod}2}-1}{N^{2}-1}\right)$
and $B_{m}=\left(\frac{N^{m}-1}{N-1}\right)$.\\
{\rm (3)} If $N$ is even and $s$ is even, then
\begin{equation}
\begin{aligned}
&\Gamma^{\star}=\Gamma_{m,s-1}^{\star} + \left(h^{\left(m,s-1\right)}-h\right)
\left[sN^{m}-\left(\frac{s-1}{2}\right)N^{m}-\left(\frac{N}{2}\right)
B_{m}+N^{1+m\,\mathrm{mod}\,2}
A_{m}\right].
\label{eq:thrm-Gamma-case3c}
\end{aligned}
\end{equation}
\end{theorem}

\begin{corollary}
\label{cor:Gamma-case4} 
Suppose that $J_i=\tilde{J}/N^i$ for some $\tilde{J}>0$. Let $\alpha\in\left(0,1\right)$ 
and $0\leq m\leq n-1$ be such that $h=\tilde{J}\left(n-m-\alpha\right)$. Then 
\begin{equation}
\Gamma^{\star}=\left[1+o_{N}\left(1\right)\right]\,\frac{\tilde{J}}{4}\,\alpha^{2}\,N^{m+1}.
\label{eq:thrm-Gamma-case4}
\end{equation}
\end{corollary}


\bigskip\noindent
\paragraph{$\bullet$ {\bf Critical configurations}}

Recall the definition of $\cC^{\star}$ in Definition~\ref{prcrcrdef}. Recall from Section~\ref{S1.4}
that every integer $a\in\Lambda_{N}^{n}$ corresponds to a vertex $v_a$ in such a way
that $d\left(a,b\right)=d\left(v_a,v_b\right)$, and that $\gamma\colon\,\boxminus\to\boxplus$ 
is the reference path $\gamma=\left(\gamma_{0},\ldots,\gamma_{N^{n}}\right)$, where $\gamma_{k}$ 
is the configuration with $\gamma_{k}\left(v_a\right)=+1$ for $a\leq k$ and $\gamma_{k}
\left(v_a\right)=-1$ for $a>k$. If $\vec{J}$ is monotone, then $\gamma$ is an optimal path
as defined in \eqref{eq:optimal-path}.

\begin{theorem}
\label{thm:Cstar-case1} 
Suppose that $\vec{J}$ is strictly monotone. Then there exists a $1\leq M\leq N^{n}$ 
such that $\cC^{\star}$ is the set of isometric translations of $\gamma_{M}$. Furthermore,
if {\rm (A1)},  {\rm (A2)} and {\rm (A4)} hold, then the $N$-ary decomposition $M=a_{n-1}N^{n-1}
+\ldots+a_{0}$ 
satisfies 
\begin{equation}
\lim_{N\to\infty} \frac{1}{N}\sum_{i=0}^{n-1}\left|a_{i}-\eta_{i}\right|=0,
\label{eq:thrm - K decomp}
\end{equation}
where the coordinates $\eta_{0},\ldots,\eta_{n-1}$ are as follows: $\eta_{i}=0$ for $\hat{m}< i 
\leq n-1$, $\eta_{\hat{m}}=\left\lceil \hat{s}\right\rceil$, and $\eta_{\hat{m}-1},\ldots,\eta_{0}$ 
are defined recursively in \eqref{eq:zetamhat-1} and \eqref{eq:zetai+1} below.
\end{theorem}

\noindent
By isometric translation we mean any bijection $\phi\colon\,\Lambda_{N}^{n}
\to\Lambda_{N}^{n}$ such that $d\left(v_a,v_b\right)=d\left(\phi\left(v_a\right),
\phi\left(v_b\right)\right)$, $1\leq a,b\leq N^{n}$.

\begin{theorem}
\label{thm:Cstar-case2} 
Suppose that $\vec{J}$ is strictly monotone and that $J_i=\tilde{J}_i/N^i$ 
with $\tilde{J}_i=o\left(N\right)$. If {\rm (A1)},  {\rm (A2)} and {\rm (A4)}  hold, then the 
coordinates $\eta_{0},\ldots,\eta_{n-1}$ in Theorem~{\rm \ref{thm:Cstar-case1}} 
are as follows:
\begin{equation}
\eta_{i} = 
\begin{cases}
0, & \hat{m}< i \leq n-1,\\
\left\lceil \hat{s}\right\rceil, & i=\hat{m},\\
\frac{N}{2}, & i=\hat{m}-1,\\
\frac{N}{2}\left[\sum_{j=1}^{i+1}
\left(\frac{\tilde{J}_{\hat{m}-i+j}}{\tilde{J}_{\hat{m}-i}}\right)
\left(1-\frac{2\eta_{\hat{m}-i}}{N}\right)+\sum_{j=2}^{n-\hat{m}}
\left(\frac{\tilde{J}_{\hat{m}+j}}{\tilde{J}_{\hat{m}-i}}\right)
-\frac{h}{\tilde{J}_{\hat{m}-i}}+1\right], & 1\leq i\leq\hat{m}-1.
\end{cases}
\end{equation}
 
\end{theorem}

\begin{theorem}
\label{thm:Cstar-case3} 
Suppose that $J_i=\tilde{J}/N^i$ for some $\tilde{J}>0$. Let $\left(m,s\right) \in \mathbb{I}$ 
be such that $h$ satisfies 
\begin{equation}
h^{\left(m,s\right)}\leq h<h^{\left(m,s-1\right)}.
\end{equation}
Then $\cC^{\star}$ is the set of all isometric translations of the configuration $\gamma_{M}$, 
where 
\small
\begin{equation}
M=\begin{cases}
\left\lceil \frac{s}{2}N^{m}\right\rceil,  
&N\mbox{ is odd and }s\mbox{ is odd},\\
\left\lceil \left(\frac{s-1}{2}\right)N^{m}\right\rceil +1,
&N\mbox{ is odd and } s \mbox{ is even},\\
\left(\frac{s-1}{2}\right)N^{m}+\sum_{j=1}^{m-1}\left(\frac{N}{2}
-\left(s+j+1\right)\mathrm{mod}\,2\right)N^{m-j}+\frac{N}{2}, 
&N\mbox{ is even and }s\mbox{ is odd},\\
\left(\frac{s-2}{2}\right)N^{m}+\sum_{j=1}^{m-1}\left(\frac{N}{2}
-\left(s+j+1\right)\mathrm{mod}\,2\right)N^{m-j}+\frac{N}{2}, 
&N\mbox{ is even and } s \mbox{ is even}.
\end{cases}
\label{eq: K in Cstar3}
\end{equation}
\normalsize
\end{theorem}


\bigskip\noindent
\paragraph{$\bullet$ {\bf Prefactor}}

We finally turn to the prefactor $K^{\star}$ defined in \eqref{eq:variationalform}. 

\begin{theorem}
\label{thm:Kstar-case1} 
Suppose that $\vec{J}$ is strictly monotone and that {\rm (A1)--(A5)} 
hold. Then 
\begin{equation}
\begin{aligned}
&\frac{1}{K^\star} =  \left[1+o_{N}\left(1\right)\right]\\
&\times \frac{\Big[\sum_{i\in B_{d}}\eta_{i-1}N^{i-1}\Big]
\Big[\sum_{i\in B_{u}}\left(N^{i}-\eta_{i-1}N^{i-1}\right)\Big]}
{\Big[\sum_{i\in B_{d}}\eta_{i-1}N^{i-1}\Big]
+\Big[\sum_{i\in B_{u}}\left(N^{i}-\eta_{i-1}N^{i-1}\right)\Big]}\,
\frac{N^{n-\hat{m}-1}}{N-\eta_{0}}\prod_{i=0}^{\hat{m}}{N \choose \eta_{i}}
\left(N-\eta_{i}\right),
\label{eq:thrm Kstar-case1}
\end{aligned}
\end{equation}
where $\eta_0,\ldots,\eta_{n-1}$ are the coordinates defining the critical configurations in 
Theorem~{\rm \ref{thm:Cstar-case1}}, and the integer sets $B_{d}$ and $B_{u}$ are 
defined in \eqref{eq:Bd Bu} below.
\end{theorem}

\begin{theorem}
\label{thm:Kstar-case2} 
Suppose that $J_i=\tilde{J}/N^i$ for some $\tilde{J}>0$ and that $h$ satisfies 
\begin{equation}
h^{\left(m,s\right)}<h<h^{\left(m,s-1\right)}
\end{equation}
for some $\left(m,s\right)\in \mathbb{I}$. If $N$ is odd, $N\neq 2,4$ and $m\geq1$, then
\begin{equation}
\frac{1}{K^\star} = a_{0}N^{n-m-2}\prod_{i=0}^{m} {N \choose a_{i}}\left(N-a_{i}\right),
\label{eq:thrm Kstar-case2}
\end{equation}
where $a_{0}=\frac{N-1}{2}+1$, $a_{i}=\frac{N-1}{2}$ for $i=1,\ldots,m-1$, and
$a_{m}=\frac{s-1-\left(s+1\right)\mathrm{mod}2}{2}$.
\end{theorem}


\subsection{Discussion}
\label{S1.7}

The theorems and corollaries in Section~\ref{S1.6} provide a full description of
the metastable behaviour of Glauber spin-flip dynamics on the hierarchical 
lattice, for any dimension $N$ and any number of hierarchical levels $n$. The 
formulas are somewhat complicated for general $\vec{J}$, but simplify considerably
as more restrictions are imposed on $\vec{J}$, such as $J_i=\tilde{J}/N^i$,
$1 \leq i \leq n$ and $\tilde{J} > 0$, and in the hierarchical mean-field limit $N\to\infty$. 
The formulas even allow us to investigate the limit $n\to\infty$ towards the infinite 
hierarchical lattice. 

The case of `standard' interaction, defined by $J_{i}=\tilde{J}/N^{i}$ and treated in 
Section ~\ref{S4}, is the easiest to interpret. The magnetic field $h$ defines the integer 
pair $\left(m,s\right)$ through the inequality 
\begin{equation}
\tilde{J}\left[\left(1-\frac{1}{N}\right)\left(n-m\right)-\left(s-1\right)\frac{1}{N}\right]
\leq h<\tilde{J}\left[\left(1-\frac{1}{N}\right)\left(n-m\right)-\left(s-2\right)\frac{1}{N}\right].
\end{equation}
It turns out that the pair $\left(m,s\right)$ captures the size of a critical configuration. Indeed, 
from Theorem \ref{thm:Cstar-case3} we see that if $N$ is odd, then every critical configuration 
is of size $M=\lceil \frac{sN^{m}}{2}\rceil$ when $s$ is odd and $M=\lceil\frac{(s-1)N^{m}}
{2}\rceil $ when $s$ is even, with similar results for $N$ even. In particular, the set of critical 
configurations corresponds precisely to the set of all configurations of said size that are an 
isometric translation of $\gamma_{M}$. 

Equations \eqref{eq:thrm-Gamma-case3a} and \eqref{eq:thrm-Gamma-case3c} in 
Theorem ~\ref{thm:Gamma-case3} are not particularly elegant, but in the hierarchical 
mean-field limit, and with $\alpha\in\left(0,1\right)$ and $1\leq m\leq n-1$ defined through 
the equation $h=\tilde{J}\left(n-m-\alpha\right)$, we find that 
\begin{equation}
\lim_{N\to\infty}\frac{\Gamma^{\star}}{N^{m+1}}=\frac{\tilde{J}\alpha^{2}}{4},
\label{eq:Gamma limit}
\end{equation}
while for $\alpha=0$ we have $\lim_{N\to\infty}\frac{\Gamma^{\star}}{N^m}=\tfrac14\tilde{J}$. 

The prefactor $K^{\star}$ in Theorem~\ref{thm:Kstar-case2} in the hierarchical mean-field 
limit scales like
\begin{equation}
\frac{1}{K^{\star}} \sim \left(\frac{1-\alpha}{2}\right)2^{m\left(N-\frac{1}{2}\right)}
N^{n}{N \choose \alpha N},
\end{equation}
in which the dominant term is exponential in $N$.

Our results are part of a larger enterprise in which the goal is to understand 
metastability on large graphs. Jovanovski~\cite{Jpr} analysed the case of the
\emph{hypercube}, Dommers~\cite{Dpr} the case of the \emph{random regular 
graph}, and Dommers, den Hollander, Jovanovski and Nardi~\cite{DdHJNpr} the 
case of the \emph{configuration model}. Each requires carrying out a detailed 
combinatorial analysis that is model-specific, even though the metastable behaviour 
expressed in Theorems~\ref{thm:critdrop}--\ref{thm:explaw} is universal. For 
lattices like the hypercube and the hierarchical lattice a full identification of the 
triple in \eqref{eq:triple} is possible, while for random graphs like the random regular 
graph and the configuration model so far only the communication height is well 
understood, while the set of critical configurations and the prefactor remain 
somewhat elusive.


\section{Monotone pair potentials}
\label{S2}

In Section~\ref{S2.1} we study the change in energy when all spins in two hierarchical 
blocks are switched (Lemma~\ref{lem:switching} below). In Section~\ref{S2.2} we show 
that the reference path $\gamma$ is an optimal paths for monotone pair potentials  
(Lemma~\ref{lem: unif opt. path} below). In Section~\ref{S2.3} we give the proof of 
Theorem~\ref{thrm:hyp-H}.


\subsection{Energy landscape}
\label{S2.1}

Let $m\leq n-1$, let $U$ be an $m+1$-block in $\Lambda_{N}^{n}$, and let $U_{1}$ and 
$U_{2}$ be two disjoint $m$-blocks in $U$. Suppose that $U_{1}'\subset U_{1}$ is 
a $k$-block in $U_{1}$ and $U_{2}'\subset U_{2}$ is a $k$-block in $U_{2}$, for 
some $k<m$. Let $\sigma\in\Omega$ be any configuration, and let $\sigma'$ be 
the result of switching the values of $\sigma$ at $U_{1}'$ and $U_{2}'$. More 
precisely, let $\varphi\colon\,U_{1}' \to U_{2}'$ be any isometric (with respect 
to $d$) bijection, and set 
\begin{equation}
\sigma'(v)=\begin{cases}
\sigma(v), & v\notin U_{1}'\cup U_{2}',\\
\sigma(\varphi(v)), & v\in U_{1}',\\
\sigma(\varphi^{-1}(v)), & v\in U_{2}'.
\end{cases}
\end{equation}
For $k+1\leq i\leq m$, let $A_{i}=\left\{ x\in U_{1}\cap\overline{\sigma}\colon\, d\left(x,U_{1}
'\right)=i\right\}$ (which is well defined because all $v\in U_{1}'$ are at the same
distance from any $x\in U_{1}\backslash U_{1}'$), $B_{i}=\left\{ x\in U_{1}\cap\sigma
\colon\, d\left(x,U_{1}'\right)=i\right\}$, $C_{i}=\left\{ x\in U_{2}\cap\overline{\sigma}
\colon\, d\left(x,U_{2}'\right)=i\right\}$ and $D_{i}=\left\{ x\in U_{2}\cap\sigma\colon\, 
d\left(x,U_{2}'\right)=i\right\}$. 

\begin{lemma}
\label{lem:switching}
For any $\sigma\in\Omega$,
\begin{equation}
\cH\left(\sigma'\right)-\cH\left(\sigma\right)
=\sum_{i=k+1}^{m}2\left(J_{i}-J_{m+1}\right)\left(\left|A_{i}\right|-\left|C_{i}\right|\right)
\,\left(\left|U_{2}'\cap\sigma\right|-\left|U_{1}'\cap\sigma\right|\right).
\end{equation}
\end{lemma}

\begin{proof}
Note that the number of interacting pairs (i.e., pairs $\left(v,w\right)$ such that 
$\sigma(v)=-\sigma(w)$) in $U_{1}'\times U_{2}'$ in $\sigma$ is the 
same as in $\sigma'$. Hence
\begin{equation}
-\sum_{\substack{v\in U_{1}'\\w\in U_{2}'}}
J_{d\left(v,w\right)}\sigma(v)\sigma(w)
= -\sum_{\substack{v\in U_{1}'\\w\in U_{2}'}}
J_{d\left(v,w\right)}\sigma'(v)\sigma'(w).
\end{equation}
The same is true for interacting pairs in $\left(\overline{U_{1}'\cup U_{2}'}\right)
\times\left(\overline{U_{1}'\cup U_{2}'}\right)$, $U_{1}'\times U_{1}'$, 
$U_{2}'\times U_{2}'$, as well as $\overline{U}\times\Lambda_{N}^{n}$, where 
$\overline{U}$ is the complement of $U$. Thus, we only need to consider interacting pairs 
coming from $U_{1}'\times\left(U_{1}\backslash U_{1}'\right)$, $U_{1}'
\times\left(U_{2}\backslash U_{2}'\right)$, $U_{2}'\times\left(U_{2}\backslash 
U_{2}'\right)$ and $U_{2}'\times\left(U_{1}\backslash U_{1}'\right)$. 
The contribution to $\cH\left(\sigma\right)-\cH\left(\boxminus\right)$ of interacting pairs 
in $U_{1}'\times\left(U_{1}\backslash U_{1}'\right)$ is given by 
\begin{equation}
-\sum_{\substack{v\in U_{1}'\\w\in U_{1}\backslash U_{1}'}}
J_{d\left(v,w\right)}\sigma(v)\sigma(w)
=\sum_{i=k+1}^{m}J_{i}\left(\left|A_{i}\right|\left|U_{1}'\cap\sigma\right|
+\left|B_{i}\right|\left|U_{1}'\cap\overline{\sigma}\right|\right).
\end{equation}
Thus by moving the set $U_{1}'\cap\sigma$ from $U_{1}$ to $U_{2}$, this contribution 
is replaced by 
\begin{equation}
-\sum_{\substack{v\in U_{2}'\\w\in U_{1}\backslash U_{1}'}}
J_{d\left(v,w\right)}\sigma'(v)\sigma'(w)
= \sum_{i=k+1}^{m}J_{m+1}\left(\left|A_{i}\right|\left|U_{1}'\cap\sigma\right|
+\left|B_{i}\right|\left|U_{1}'\cap\overline{\sigma}\right|\right).
\end{equation}
Similarly, the contribution to $\cH\left(\sigma\right)-\cH\left(\boxminus\right)$ of interacting 
pairs in $U_{1}'\times\left(U_{2}\backslash U_{2}'\right)$ is given by 
\begin{equation}
\sum_{i=k+1}^{m}J_{m+1}\left(\left|C_{i}\right|\left|U_{1}'\cap\sigma\right|
+\left|D_{i}\right|\left|U_{1}'\cap\overline{\sigma}\right|\right),
\end{equation}
which is subsequently replaced by 
\begin{equation}
\sum_{i=k+1}^{m}J_{i}\left(\left|C_{i}\right|\left|U_{1}'\cap\sigma\right|
+\left|D_{i}\right|\left|U_{1}'\cap\overline{\sigma}\right|\right).
\end{equation}
Similar observations follow for interacting pairs in $U_{2}'\times\left(U_{2}\backslash 
U_{2}'\right)$ and $U_{2}'\times\left(U_{1}\backslash U_{1}'\right)$.
Hence 
\begin{equation}
\begin{aligned}
&\cH\left(\sigma'\right)-\cH\left(\sigma\right)
= \sum_{i=k+1}^{m}\left(J_{i}-J_{m+1}\right)\\
&\times \Big(\left[\left|A_{i}\right|-\left|C_{i}\right|\right]
\left(\left|U_{2}'\cap\sigma\right|-\left|U_{1}'\cap\sigma\right|\right)
+\left[\left|B_{i}\right|-\left|D_{i}\right|\right]
\left(\left|U_{2}'\cap\overline{\sigma}\right|
-\left|U_{1}'\cap\overline{\sigma}\right|\right)\Big).
\end{aligned}
\end{equation}
Noting that $\left|B_{i}\right|+\left|A_{i}\right|=\left(N-1\right)N^{i-1}=\left|D_{i}\right|
+\left|C_{i}\right|$, we get 
\begin{equation}
\begin{aligned}
&\cH\left(\sigma'\right)-\cH\left(\sigma\right)
= \sum_{i=k+1}^{m}\left(J_{i}-J_{m+1}\right)\\
&\times \Big(\left[\left|A_{i}\right|
-\left|C_{i}\right|\right]\left(\left|U_{2}'\cap\sigma\right|
-\left|U_{1}'\cap\sigma\right|\right)
+\left[\left|C_{i}\right|-\left|A_{i}\right|\right]\left(\left|U_{2}'\cap\overline{\sigma}\right|
-\left|U_{1}'\cap\overline{\sigma}\right|\right)\Big)\\
&\qquad\qquad\qquad  = \sum_{i=k+1}^{m}\left(J_{i}-J_{m+1}\right)\\
&\times \Big(\left[\left|A_{i}\right|
-\left|C_{i}\right|\right]\left(\left|U_{2}'\cap\sigma\right|
-\left|U_{1}'\cap\sigma\right|\right)
+\left[\left|A_{i}\right|-\left|C_{i}\right|\right]\left(\left|U_{1}'\cap\overline{\sigma}\right|
-\left|U_{2}'\cap\overline{\sigma}\right|\right)\Big).
\end{aligned}
\end{equation}
Finally, noting that $\left|U_{1}'\cap\overline{\sigma}\right|=N^{k}-\left|U_{1}'
\cap\sigma\right|$ and $\left|U_{2}'\cap\overline{\sigma}\right|=N^{k}-\left|U_{2}'
\cap\sigma\right|$, we complete the proof.
\end{proof}


\subsection{Optimal paths}
\label{S2.2}

Recall the definition of an optimal path from (\ref{eq:optimal-path}). We call a path 
$\gamma\colon\,\boxminus\to\boxplus$, denoted by $\left\{\gamma_{i}\right\} _{i=0}^{M}$ 
for some $M\geq N^{n}$, \emph{uniformly optimal} when, for all $0\leq i\leq M$, 
\begin{equation}
\cH\left(\gamma_{i}\right)=\min_{ {\sigma\in\Omega:} \atop {\left|\sigma\right|
=\left|\gamma_{i}\right|}} \cH\left(\gamma_{i}\right),
\label{eq:unif-optimal}
\end{equation}
and \emph{strictly optimal} when the minimum in the right-hand side of (\ref{eq:optimal-path}) 
is only attained by configurations that belong to some uniformly optimal path. We think of a 
path $\gamma$ between two configurations in $\Omega$ both as a sequence of configurations
denoted by $\left\{\gamma_{i}\right\}_{i=1}^M$ and as a sequence of vertices denoted by 
$\left\{\gamma\left(i\right)\right\}_{i=1}^{M}$, where $\gamma\left(i\right)$ is the single 
vertex in the symmetric difference $\gamma_{i-1}\triangle\gamma_{i}$.

Order the vertices $\left\{v_{i}\right\} _{i=1}^{N^{n}}$ in $\Lambda_N^{n}$ in a natural 
order so that, for all $1\leq k\leq n-1$ and for all $0\leq j\leq N^{n}/N^{k}$, $\{v_{jN^{k}+1},
\ldots,v_{\left(j+1\right)N^{k}}\}$ belong to the same $k$-block. Let 
$\gamma^{\mathrm{MD}}\colon\,\boxminus\to\boxplus$ be the path defined by 
$\gamma^{\mathrm{MD}}\left(i\right)=v_{i}$ for $1\leq i\leq N^{n}$. Let $\gamma^{\mathrm{MI}}
\colon\,\boxminus\to\boxplus$ be defined by $\gamma^{\mathrm{MI}}\left(k\right)
=v_{\theta\left(k\right)}$ and
\begin{equation}
\theta\left(k\right)=1+\sum_{i=0}^{n-1}N^{n-1-i}
\left(\left\lfloor \frac{k-1}{N^{i}}\right\rfloor \mbox{mod}\,N\right).
\end{equation}
Thus, the vertex $\gamma^{\mathrm{MI}}\left(k\right)$ belongs to the $\left(\left(k-1\right)\,
\mbox{mod}\,N\right)^{\mbox{th}}$ $\left(n-1\right)$-block, and within that block it belongs 
to the $(\lfloor\frac{k-1}{N^{2}}\rfloor \mbox{mod}\,N)^{\mbox{th}}$ $(n-2)$-block, etc. We can 
now use Lemma~\ref{lem:switching} to draw the following conclusions.

\begin{lemma}
\label{lem: unif opt. path}
{\rm (1)} If $\vec{J}$ is non-increasing, then $\gamma^{\mathrm{MD}}$ is a uniformly optimal path.\\ 
{\rm (2)} If $\vec{J}$ is non-decreasing, then $\gamma^{\mathrm{MI}}$ is a uniformly optimal path.\\ 
{\rm (3)} If $\vec{J}$ is strictly decreasing or strictly increasing, then $\gamma^{\mathrm{MD}}$ or 
$\gamma^{\mathrm{MI}}$ is strictly optimal. 
\end{lemma}

\begin{proof}
We treat the non-increasing case and the non-decreasing case separately.
 
\bigskip\noindent
\underline{Non-increasing case}: Let $\sigma\in\Omega$ be given. We will construct 
a sequence of configurations $\left\{\psi_{i}\right\} _{i=1}^{n}$, all of volume $\left|\sigma\right|$ 
and with $\psi_{n}=\gamma_{\left|\sigma\right|}^{\mathrm{MD}}$, such that $\cH\left(\sigma\right)\geq
\cH\left(\psi_{1}\right)\geq\ldots\geq\cH\left(\psi_{n}\right)$, and the inequalities being strict 
whenever $\vec{J}$ is strictly decreasing. This will prove the claim for the non-increasing case. 
 
For $1\leq k\leq n$, define $\psi_{k}$ to be the (unique) configuration that satisfies the following 
two conditions:
\begin{itemize}
\item[1.] 
For every $k$-block $U\subset\Lambda_{N}^{n}$, $\left|U\cap\sigma\right|=\left|U\cap\psi_{k}\right|$.
\item[2.] 
For all $i<j$ with $v_{i}$ and $v_{j}$ belonging in the same $k$-block, $v_{j}\in\psi_{k}$ 
implies $v_{i}\in\psi_{k}$.
\end{itemize}
In particular, note that $\psi_{1}$ is obtained from $\sigma$ by ``shifting'' the $+1$ 
values of $\sigma$ found inside every $1$-block as far left as possible (i.e., with the 
lowest possible index) within the same $1$-block. It is obvious that $\cH\left(\psi_{1}\right)
=\cH\left(\sigma\right)$. It is also clear from this recursive definition that $\psi_{n}
=\gamma_{\left|\sigma\right|}^{\mathrm{MD}}$. 

Starting with $\psi_{k}$, we will show how to obtain $\psi_{k+1}$ by a series of transformations 
that are non-increasing in $\cH$. Let $U$ be the first $k+1$ block of $\Lambda_{n}^{N}$, and let $U_{1},
\ldots,U_{N}$ be its $k$-blocks, arranged so that $\left|U_{i}\cap\sigma\right|\geq\left|U_{i+1}
\cap\sigma\right|$. Note that this may be achieved by re-arranging (or re-labeling) the blocks 
$U_{1},\ldots,U_{N}$, and any such re-arranging is an $\cH$-preserving operation. Let 
$a=\min\left\{i\colon\,\left|U_{i}\cap\sigma\right|<N^{k}\right\}$ and $b=\max\left\{i\colon\,
\left|U_{i}\cap\sigma\right|>1\right\}$. Note that if $a=b$, then $U\cap\sigma$ is already 
in the correct form, satisfying the definition of $\psi_{k+1}$. Thus, we may assume that
$a\neq b$. Find a maximal block $\tilde{U}_{b}\subsetneq U_{b}$ with $|\tilde{U}_{b}
\cap\sigma|>0$ such that, for some block of equal size $\tilde{U}_{a}\subsetneq U_{a}$, 
$|\tilde{U}_{b}\cap\sigma|>|\tilde{U}_{a}\cap\sigma|$. To do this, take the 
first $k-1$-block $U_{b}'$ in $U_{b}$ and the last $k-1$-block $U_{a}'$ in $U_{a}$ 
that satisfies $\left|U_{a}'\cap\overline{\sigma}\right|>0$, and check whether $\left|U_{b}
'\cap\sigma\right|>\left|U_{a}'\cap\sigma\right|$. If not, then proceed by taking 
the first $k-2$-block in $U_{b}$, etc. By the definition of $a$ and $b$, this constructive search 
for $\tilde{U}_{a}$ and $\tilde{U}_{b}$ always yields two such blocks. Once these are found,
perform the switching operation in Lemma \ref{lem:switching} on the blocks $\tilde{U}_{a}$ 
and $\tilde{U}_{b}$, and denote the resulting configuration by $\psi_{k}'$ (see 
Fig.~\ref{fig:mapPsi}). Then, by Lemma \ref{lem:switching}, with $s$ denoting the size 
of the blocks $\tilde{U}_{a}$ and $\tilde{U}_{b}$,
\begin{equation}
\cH\left(\psi_{k}'\right)-\cH\left(\psi_{k}\right)
=\sum_{i=s+1}^{k}2\left(J_{i}-J_{k+1}\right)\left[\left|A_{i}\right|-\left|C_{i}\right|\right]
\left(|\tilde{U}_{b}\cap\sigma|-|\tilde{U}_{a}\cap\sigma|\right),
\end{equation}
where we recall that $A_{i}=\{x\in U_{a}\cap\overline{\sigma}\colon\, d(x,\tilde{U}_{a})=i\}$ 
and $C_{i}=\{x\in U_{b}\cap\overline{\sigma}\colon\, d(x,\tilde{U}_{b})=i\}$. By definition, 
we have $|\tilde{U}_{b}\cap\sigma|-|\tilde{U}_{a} \cap\sigma|>0$, and from the monotonicity 
we get that $J_{i}-J_{m+1} \geq 0$. Lastly, by the fact that $\left|U_{a}\cap\sigma\right|
\geq\left|U_{b}\cap\sigma\right|$ and the construction of $\psi_{k}$, as well as the 
definition of $\tilde{U}_{b}$ and $\tilde{U}_{a}$, it also follows that $\left|A_{i}\right|
-\left|C_{i}\right| \leq 0$ for all $s+1\leq i\leq k$. Therefore $\cH\left(\psi_{k}'\right)
-\cH\left(\psi_{k}\right)\leq0$. Repeating this construction until $\min\left\{i\colon\,\left|U_{i}
\cap\sigma\right|<N^{k}\right\} =\max\left\{i\colon\,\left|U_{i}\cap\sigma\right|>1\right\}$ 
(which happens in a finite number of moves), and repeating the same construction for all 
other $k+1$-blocks, we get the configuration $\psi_{k+1}$, and hence $\cH\left(\psi_{k+1}\right)
-\cH\left(\psi_{k}\right) \leq 0$. 

\begin{figure}[h]
\begin{tikzpicture} [scale=0.30] 	
\draw[gray, very thin] (0.0,0.0) -- (8,0) -- (8,8) -- (0,8) -- (0,0); 	
\draw[gray, very thin] (0.1,0.1) -- (3.9,0.1) -- (3.9,3.9) -- (0.1,3.9) -- (0.1,0.1); 		
\draw[gray, very thin] (0.3,0.3) -- (1.9,0.3) -- (1.9,1.9) -- (0.3,1.9) -- (0.3,0.3); 		
\draw[gray, very thin] (2.1,0.3) -- (3.7,0.3) -- (3.7,1.9) -- (2.1,1.9) -- (2.1,0.3); 		
\draw[gray, very thin] (2.1,2.1) -- (3.7,2.1) -- (3.7,3.7) -- (2.1,3.7) -- (2.1,2.1); 		
\draw[gray, very thin] (0.3,2.1) -- (1.9,2.1) -- (1.9,3.7) -- (0.3,3.7) -- (0.3,2.1); 	
\draw[gray, very thin] (4.1,0.1) -- (7.9,0.1) -- (7.9,3.9) -- (4.1,3.9) -- (4.1,0.1); 		
\draw[gray, very thin] (0.3+4,0.3) -- (1.9+4,0.3) -- (1.9+4,1.9) -- (0.3+4,1.9) -- (0.3+4,0.3); 		
\draw[gray, very thin] (2.1+4,0.3) -- (3.7+4,0.3) -- (3.7+4,1.9) -- (2.1+4,1.9) -- (2.1+4,0.3); 		
\draw[gray, very thin] (2.1+4,2.1) -- (3.7+4,2.1) -- (3.7+4,3.7) -- (2.1+4,3.7) -- (2.1+4,2.1); 		
\draw[gray, dashed, very thin] (0.3+4,2.1) -- (1.9+4,2.1) -- (1.9+4,3.7) -- (0.3+4,3.7) -- (0.3+4,2.1); 	
\draw[gray, very thin] (0.1,4.1) -- (3.9,4.1) -- (3.9,7.9) -- (0.1,7.9) -- (0.1,4.1); 		
\draw[gray, very thin] (0.3,0.3+4) -- (1.9,0.3+4) -- (1.9,1.9+4) -- (0.3,1.9+4) -- (0.3,0.3+4); 		
\draw[gray, very thin] (2.1,0.3+4) -- (3.7,0.3+4) -- (3.7,1.9+4) -- (2.1,1.9+4) -- (2.1,0.3+4); 		
\draw[gray, very thin] (2.1,2.1+4) -- (3.7,2.1+4) -- (3.7,3.7+4) -- (2.1,3.7+4) -- (2.1,2.1+4); 		
\draw[gray, very thin] (0.3,2.1+4) -- (1.9,2.1+4) -- (1.9,3.7+4) -- (0.3,3.7+4) -- (0.3,2.1+4); 	
\draw[gray, very thin] (4.1,4.1) -- (7.9,4.1) -- (7.9,7.9) -- (4.1,7.9) -- (4.1,4.1); 		
\draw[gray, very thin] (0.3+4,0.3+4) -- (1.9+4,0.3+4) -- (1.9+4,1.9+4) -- (0.3+4,1.9+4) -- (0.3+4,0.3+4); 		\draw[gray, very thin] (2.1+4,0.3+4) -- (3.7+4,0.3+4) -- (3.7+4,1.9+4) -- (2.1+4,1.9+4) -- (2.1+4,0.3+4); 		\draw[gray, very thin] (2.1+4,2.1+4) -- (3.7+4,2.1+4) -- (3.7+4,3.7+4) -- (2.1+4,3.7+4) -- (2.1+4,2.1+4); 		\draw[gray, very thin] (0.3+4,2.1+4) -- (1.9+4,2.1+4) -- (1.9+4,3.7+4) -- (0.3+4,3.7+4) -- (0.3+4,2.1+4); 	\fill[black!] (0.7,0.7) circle (0.1cm);	 	
\fill[black!] (1.5,0.7) circle (0.1cm); 	
\fill[black!] (1.5,1.5) circle (0.1cm); 	
\fill[black!] (0.7,1.5) circle (0.1cm); 	
\fill[black!] (0.7+1.8,0.7) circle (0.1cm); 	
\fill[black!] (1.5+1.8,0.7) circle (0.1cm); 	
\fill[black!] (1.5+1.8,1.5) circle (0.1cm); 	
\fill[black!] (0.7+1.8,1.5) circle (0.1cm); 	
\fill[black!] (0.7,0.7+1.8) circle (0.1cm); 	
\fill[black!] (1.5,0.7+1.8) circle (0.1cm); 	
\fill[black!] (1.5,1.5+1.8) circle (0.1cm); 	
\fill[black!] (0.7,1.5+1.8) circle (0.1cm); 	
\fill[black!] (0.7+1.8,0.7+1.8) circle (0.1cm); 	
\fill[black!] (1.5+1.8,0.7+1.8) circle (0.1cm); 	
\fill[black!] (1.5+1.8,1.5+1.8) circle (0.1cm); 	
\fill[black!] (0.7+1.8,1.5+1.8) circle (0.1cm); 	
\fill[black!] (0.7+4,0.7) circle (0.1cm); 	
\fill[black!] (1.5+4,0.7) circle (0.1cm); 	
\fill[black!] (1.5+4,1.5) circle (0.1cm); 	
\fill[black!] (0.7+4,1.5) circle (0.1cm); 	
\draw[very thin] (0.7+1.8+4,0.7) circle (0.1cm); 	
\draw[very thin] (1.5+1.8+4,0.7) circle (0.1cm); 	
\draw[very thin] (1.5+1.8+4,1.5) circle (0.1cm); 	
\draw[very thin] (0.7+1.8+4,1.5) circle (0.1cm); 	
\fill[black!] (0.7+4,0.7+1.8) circle (0.1cm); 	
\draw[very thin] (1.5+4,0.7+1.8) circle (0.1cm); 	
\draw[very thin] (1.5+4,1.5+1.8) circle (0.1cm); 	
\fill[black!] (0.7+4,1.5+1.8) circle (0.1cm); 	
\draw[very thin] (0.7+1.8+4,0.7+1.8) circle (0.1cm); 	
\draw[very thin] (1.5+1.8+4,0.7+1.8) circle (0.1cm); 	
\draw[very thin] (1.5+1.8+4,1.5+1.8) circle (0.1cm); 	
\draw[very thin] (0.7+1.8+4,1.5+1.8) circle (0.1cm); 	
\fill[black!] (0.7,0.7+4) circle (0.1cm); 	
\fill[black!] (1.5,0.7+4) circle (0.1cm); 	
\fill[black!] (1.5,1.5+4) circle (0.1cm); 	
\fill[black!] (0.7,1.5+4) circle (0.1cm); 	
\fill[black!] (0.7+1.8,0.7+4) circle (0.1cm); 	
\fill[black!] (1.5+1.8,0.7+4) circle (0.1cm); 	
\fill[black!] (1.5+1.8,1.5+4) circle (0.1cm); 	
\fill[black!] (0.7+1.8,1.5+4) circle (0.1cm); 	
\fill[black!] (0.7,0.7+1.8+4) circle (0.1cm); 	
\fill[black!] (1.5,0.7+1.8+4) circle (0.1cm); 	
\fill[black!] (1.5,1.5+1.8+4) circle (0.1cm); 	
\fill[black!] (0.7,1.5+1.8+4) circle (0.1cm); 	
\fill[black!] (0.7+1.8,0.7+1.8+4) circle (0.1cm); 	
\fill[black!] (1.5+1.8,0.7+1.8+4) circle (0.1cm); 	
\fill[black!] (1.5+1.8,1.5+1.8+4) circle (0.1cm); 	
\fill[black!] (0.7+1.8,1.5+1.8+4) circle (0.1cm); 	
\fill[black!] (0.7+4,0.7+4) circle (0.1cm); 	
\fill[black!] (1.5+4,0.7+4) circle (0.1cm); 	
\fill[black!] (1.5+4,1.5+4) circle (0.1cm); 	
\fill[black!] (0.7+4,1.5+4) circle (0.1cm); 	
\fill[black!] (0.7+1.8+4,0.7+4) circle (0.1cm); 	
\fill[black!] (1.5+1.8+4,0.7+4) circle (0.1cm); 	
\fill[black!] (1.5+1.8+4,1.5+4) circle (0.1cm); 	
\fill[black!] (0.7+1.8+4,1.5+4) circle (0.1cm); 	
\fill[black!] (0.7+4,0.7+1.8+4) circle (0.1cm); 	
\fill[black!] (1.5+4,0.7+1.8+4) circle (0.1cm); 	
\fill[black!] (1.5+4,1.5+1.8+4) circle (0.1cm); 	
\fill[black!] (0.7+4,1.5+1.8+4) circle (0.1cm); 	
\fill[black!] (0.7+1.8+4,0.7+1.8+4) circle (0.1cm); 	
\fill[black!] (1.5+1.8+4,0.7+1.8+4) circle (0.1cm); 	
\fill[black!] (1.5+1.8+4,1.5+1.8+4) circle (0.1cm); 	
\fill[black!] (0.7+1.8+4,1.5+1.8+4) circle (0.1cm); 
\node[text width=3cm] at (3.5,4) {\tiny {$U_{a}$}};    
\begin{scope}[shift={(11,0)}]         
\draw[gray, very thin] (0.0,0.0) -- (8,0) -- (8,8) -- (0,8) -- (0,0); 	
\draw[gray, very thin] (0.1,0.1) -- (3.9,0.1) -- (3.9,3.9) -- (0.1,3.9) -- (0.1,0.1); 		
\draw[gray, dashed, very thin] (0.3,0.3) -- (1.9,0.3) -- (1.9,1.9) -- (0.3,1.9) -- (0.3,0.3); 		
\draw[gray, very thin] (2.1,0.3) -- (3.7,0.3) -- (3.7,1.9) -- (2.1,1.9) -- (2.1,0.3); 		
\draw[gray, very thin] (2.1,2.1) -- (3.7,2.1) -- (3.7,3.7) -- (2.1,3.7) -- (2.1,2.1); 		
\draw[gray, very thin] (0.3,2.1) -- (1.9,2.1) -- (1.9,3.7) -- (0.3,3.7) -- (0.3,2.1); 	
\draw[gray, very thin] (4.1,0.1) -- (7.9,0.1) -- (7.9,3.9) -- (4.1,3.9) -- (4.1,0.1); 		
\draw[gray, very thin] (0.3+4,0.3) -- (1.9+4,0.3) -- (1.9+4,1.9) -- (0.3+4,1.9) -- (0.3+4,0.3); 		
\draw[gray, very thin] (2.1+4,0.3) -- (3.7+4,0.3) -- (3.7+4,1.9) -- (2.1+4,1.9) -- (2.1+4,0.3); 		
\draw[gray, very thin] (2.1+4,2.1) -- (3.7+4,2.1) -- (3.7+4,3.7) -- (2.1+4,3.7) -- (2.1+4,2.1); 		
\draw[gray, very thin] (0.3+4,2.1) -- (1.9+4,2.1) -- (1.9+4,3.7) -- (0.3+4,3.7) -- (0.3+4,2.1); 	
\draw[gray, very thin] (0.1,4.1) -- (3.9,4.1) -- (3.9,7.9) -- (0.1,7.9) -- (0.1,4.1); 		
\draw[gray, very thin] (0.3,0.3+4) -- (1.9,0.3+4) -- (1.9,1.9+4) -- (0.3,1.9+4) -- (0.3,0.3+4); 		
\draw[gray, very thin] (2.1,0.3+4) -- (3.7,0.3+4) -- (3.7,1.9+4) -- (2.1,1.9+4) -- (2.1,0.3+4); 		
\draw[gray, very thin] (2.1,2.1+4) -- (3.7,2.1+4) -- (3.7,3.7+4) -- (2.1,3.7+4) -- (2.1,2.1+4); 		
\draw[gray, very thin] (0.3,2.1+4) -- (1.9,2.1+4) -- (1.9,3.7+4) -- (0.3,3.7+4) -- (0.3,2.1+4); 	
\draw[gray, very thin] (4.1,4.1) -- (7.9,4.1) -- (7.9,7.9) -- (4.1,7.9) -- (4.1,4.1); 		
\draw[gray, very thin] (0.3+4,0.3+4) -- (1.9+4,0.3+4) -- (1.9+4,1.9+4) -- (0.3+4,1.9+4) -- (0.3+4,0.3+4); 		\draw[gray, very thin] (2.1+4,0.3+4) -- (3.7+4,0.3+4) -- (3.7+4,1.9+4) -- (2.1+4,1.9+4) -- (2.1+4,0.3+4); 		\draw[gray, very thin] (2.1+4,2.1+4) -- (3.7+4,2.1+4) -- (3.7+4,3.7+4) -- (2.1+4,3.7+4) -- (2.1+4,2.1+4); 		\draw[gray, very thin] (0.3+4,2.1+4) -- (1.9+4,2.1+4) -- (1.9+4,3.7+4) -- (0.3+4,3.7+4) -- (0.3+4,2.1+4); 	\fill[black!] (0.7,0.7) circle (0.1cm);	 	
\draw[very thin] (1.5,0.7) circle (0.1cm); 	
\fill[black!] (1.5,1.5) circle (0.1cm); 	
\fill[black!] (0.7,1.5) circle (0.1cm); 	
\draw[very thin] (0.7+1.8,0.7) circle (0.1cm); 	
\draw[very thin] (1.5+1.8,0.7) circle (0.1cm); 	
\draw[very thin] (1.5+1.8,1.5) circle (0.1cm); 	
\draw[very thin] (0.7+1.8,1.5) circle (0.1cm); 	
\draw[very thin] (0.7,0.7+1.8) circle (0.1cm); 	
\draw[very thin] (1.5,0.7+1.8) circle (0.1cm); 	
\draw[very thin] (1.5,1.5+1.8) circle (0.1cm); 	
\draw[very thin] (0.7,1.5+1.8) circle (0.1cm); 	
\draw[very thin] (0.7+1.8,0.7+1.8) circle (0.1cm); 	
\draw[very thin] (1.5+1.8,0.7+1.8) circle (0.1cm); 	
\draw[very thin] (1.5+1.8,1.5+1.8) circle (0.1cm); 	
\draw[very thin] (0.7+1.8,1.5+1.8) circle (0.1cm); 	
\draw[very thin] (0.7+4,0.7) circle (0.1cm); 	
\draw[very thin] (1.5+4,0.7) circle (0.1cm); 	
\draw[very thin] (1.5+4,1.5) circle (0.1cm); 	
\draw[very thin] (0.7+4,1.5) circle (0.1cm); 	
\draw[very thin] (0.7+1.8+4,0.7) circle (0.1cm); 	
\draw[very thin] (1.5+1.8+4,0.7) circle (0.1cm); 	
\draw[very thin] (1.5+1.8+4,1.5) circle (0.1cm); 	
\draw[very thin] (0.7+1.8+4,1.5) circle (0.1cm); 	
\draw[very thin] (0.7+4,0.7+1.8) circle (0.1cm); 	
\draw[very thin] (1.5+4,0.7+1.8) circle (0.1cm); 	
\draw[very thin] (1.5+4,1.5+1.8) circle (0.1cm); 	
\draw[very thin] (0.7+4,1.5+1.8) circle (0.1cm); 	
\draw[very thin] (0.7+1.8+4,0.7+1.8) circle (0.1cm); 	
\draw[very thin] (1.5+1.8+4,0.7+1.8) circle (0.1cm); 	
\draw[very thin] (1.5+1.8+4,1.5+1.8) circle (0.1cm); 	
\draw[very thin] (0.7+1.8+4,1.5+1.8) circle (0.1cm); 	
\draw[very thin] (0.7,0.7+4) circle (0.1cm); 	
\draw[very thin] (1.5,0.7+4) circle (0.1cm); 	
\draw[very thin] (1.5,1.5+4) circle (0.1cm); 	
\draw[very thin] (0.7,1.5+4) circle (0.1cm); 	
\draw[very thin] (0.7+1.8,0.7+4) circle (0.1cm); 	
\draw[very thin] (1.5+1.8,0.7+4) circle (0.1cm); 	
\draw[very thin] (1.5+1.8,1.5+4) circle (0.1cm); 	
\draw[very thin] (0.7+1.8,1.5+4) circle (0.1cm); 	
\draw[very thin] (0.7,0.7+1.8+4) circle (0.1cm); 	
\draw[very thin] (1.5,0.7+1.8+4) circle (0.1cm); 	
\draw[very thin] (1.5,1.5+1.8+4) circle (0.1cm); 	
\draw[very thin] (0.7,1.5+1.8+4) circle (0.1cm); 	
\draw[very thin] (0.7+1.8,0.7+1.8+4) circle (0.1cm); 	
\draw[very thin] (1.5+1.8,0.7+1.8+4) circle (0.1cm); 	
\draw[very thin] (1.5+1.8,1.5+1.8+4) circle (0.1cm); 	
\draw[very thin] (0.7+1.8,1.5+1.8+4) circle (0.1cm); 	
\draw[very thin] (0.7+4,0.7+4) circle (0.1cm); 	
\draw[very thin] (1.5+4,0.7+4) circle (0.1cm); 	
\draw[very thin] (1.5+4,1.5+4) circle (0.1cm); 	
\draw[very thin] (0.7+4,1.5+4) circle (0.1cm); 	
\draw[very thin] (0.7+1.8+4,0.7+4) circle (0.1cm); 	
\draw[very thin] (1.5+1.8+4,0.7+4) circle (0.1cm); 	
\draw[very thin] (1.5+1.8+4,1.5+4) circle (0.1cm); 	
\draw[very thin] (0.7+1.8+4,1.5+4) circle (0.1cm); 	
\draw[very thin] (0.7+4,0.7+1.8+4) circle (0.1cm); 	
\draw[very thin] (1.5+4,0.7+1.8+4) circle (0.1cm); 	
\draw[very thin] (1.5+4,1.5+1.8+4) circle (0.1cm); 	
\draw[very thin] (0.7+4,1.5+1.8+4) circle (0.1cm); 	
\draw[very thin] (0.7+1.8+4,0.7+1.8+4) circle (0.1cm); 	
\draw[very thin] (1.5+1.8+4,0.7+1.8+4) circle (0.1cm); 
\draw[very thin] (1.5+1.8+4,1.5+1.8+4) circle (0.1cm); 	
\draw[very thin] (0.7+1.8+4,1.5+1.8+4) circle (0.1cm); 
\node[text width=2cm] at (0.7,4) {\tiny{$\cdots \, U_{b}$}}; 
\end{scope} 
\begin{scope}[shift={(19.5,0)}] 
\draw [thick, ->] (0,4) -- (3,4); 
\end{scope} 
\begin{scope}[shift={(25,0)}] 
\draw[gray, very thin] (0.0,0.0) -- (8,0) -- (8,8) -- (0,8) -- (0,0); 	
\draw[gray, very thin] (0.1,0.1) -- (3.9,0.1) -- (3.9,3.9) -- (0.1,3.9) -- (0.1,0.1); 		
\draw[gray, very thin] (0.3,0.3) -- (1.9,0.3) -- (1.9,1.9) -- (0.3,1.9) -- (0.3,0.3); 		
\draw[gray, very thin] (2.1,0.3) -- (3.7,0.3) -- (3.7,1.9) -- (2.1,1.9) -- (2.1,0.3); 		
\draw[gray, very thin] (2.1,2.1) -- (3.7,2.1) -- (3.7,3.7) -- (2.1,3.7) -- (2.1,2.1); 		
\draw[gray, very thin] (0.3,2.1) -- (1.9,2.1) -- (1.9,3.7) -- (0.3,3.7) -- (0.3,2.1); 	
\draw[gray, very thin] (4.1,0.1) -- (7.9,0.1) -- (7.9,3.9) -- (4.1,3.9) -- (4.1,0.1); 		
\draw[gray, very thin] (0.3+4,0.3) -- (1.9+4,0.3) -- (1.9+4,1.9) -- (0.3+4,1.9) -- (0.3+4,0.3); 		
\draw[gray, very thin] (2.1+4,0.3) -- (3.7+4,0.3) -- (3.7+4,1.9) -- (2.1+4,1.9) -- (2.1+4,0.3); 		
\draw[gray, very thin] (2.1+4,2.1) -- (3.7+4,2.1) -- (3.7+4,3.7) -- (2.1+4,3.7) -- (2.1+4,2.1); 		
\draw[gray, dashed, very thin] (0.3+4,2.1) -- (1.9+4,2.1) -- (1.9+4,3.7) -- (0.3+4,3.7) -- (0.3+4,2.1); 	
\draw[gray, very thin] (0.1,4.1) -- (3.9,4.1) -- (3.9,7.9) -- (0.1,7.9) -- (0.1,4.1); 		
\draw[gray, very thin] (0.3,0.3+4) -- (1.9,0.3+4) -- (1.9,1.9+4) -- (0.3,1.9+4) -- (0.3,0.3+4); 		
\draw[gray, very thin] (2.1,0.3+4) -- (3.7,0.3+4) -- (3.7,1.9+4) -- (2.1,1.9+4) -- (2.1,0.3+4); 		
\draw[gray, very thin] (2.1,2.1+4) -- (3.7,2.1+4) -- (3.7,3.7+4) -- (2.1,3.7+4) -- (2.1,2.1+4); 		
\draw[gray, very thin] (0.3,2.1+4) -- (1.9,2.1+4) -- (1.9,3.7+4) -- (0.3,3.7+4) -- (0.3,2.1+4); 	
\draw[gray, very thin] (4.1,4.1) -- (7.9,4.1) -- (7.9,7.9) -- (4.1,7.9) -- (4.1,4.1); 		
\draw[gray, very thin] (0.3+4,0.3+4) -- (1.9+4,0.3+4) -- (1.9+4,1.9+4) -- (0.3+4,1.9+4) -- (0.3+4,0.3+4); 		\draw[gray, very thin] (2.1+4,0.3+4) -- (3.7+4,0.3+4) -- (3.7+4,1.9+4) -- (2.1+4,1.9+4) -- (2.1+4,0.3+4); 		\draw[gray, very thin] (2.1+4,2.1+4) -- (3.7+4,2.1+4) -- (3.7+4,3.7+4) -- (2.1+4,3.7+4) -- (2.1+4,2.1+4); 		\draw[gray, very thin] (0.3+4,2.1+4) -- (1.9+4,2.1+4) -- (1.9+4,3.7+4) -- (0.3+4,3.7+4) -- (0.3+4,2.1+4); 	\fill[black!] (0.7,0.7) circle (0.1cm);	 	
\fill[black!] (1.5,0.7) circle (0.1cm); 	
\fill[black!] (1.5,1.5) circle (0.1cm); 	
\fill[black!] (0.7,1.5) circle (0.1cm); 	
\fill[black!] (0.7+1.8,0.7) circle (0.1cm); 	
\fill[black!] (1.5+1.8,0.7) circle (0.1cm); 	
\fill[black!] (1.5+1.8,1.5) circle (0.1cm); 	
\fill[black!] (0.7+1.8,1.5) circle (0.1cm); 	
\fill[black!] (0.7,0.7+1.8) circle (0.1cm); 	
\fill[black!] (1.5,0.7+1.8) circle (0.1cm); 	
\fill[black!] (1.5,1.5+1.8) circle (0.1cm); 	
\fill[black!] (0.7,1.5+1.8) circle (0.1cm); 	
\fill[black!] (0.7+1.8,0.7+1.8) circle (0.1cm); 	
\fill[black!] (1.5+1.8,0.7+1.8) circle (0.1cm); 	
\fill[black!] (1.5+1.8,1.5+1.8) circle (0.1cm); 	
\fill[black!] (0.7+1.8,1.5+1.8) circle (0.1cm); 	
\fill[black!] (0.7+4,0.7) circle (0.1cm); 	
\fill[black!] (1.5+4,0.7) circle (0.1cm); 	
\fill[black!] (1.5+4,1.5) circle (0.1cm); 	
\fill[black!] (0.7+4,1.5) circle (0.1cm); 	
\draw[very thin] (0.7+1.8+4,0.7) circle (0.1cm); 	
\draw[very thin] (1.5+1.8+4,0.7) circle (0.1cm); 	
\draw[very thin] (1.5+1.8+4,1.5) circle (0.1cm); 	
\draw[very thin] (0.7+1.8+4,1.5) circle (0.1cm); 	
\fill[black!] (0.7+4,0.7+1.8) circle (0.1cm); 	
\draw[very thin] (1.5+4,0.7+1.8) circle (0.1cm); 	
\fill[black!] (1.5+4,1.5+1.8) circle (0.1cm); 	
\fill[black!] (0.7+4,1.5+1.8) circle (0.1cm); 	
\draw[very thin] (0.7+1.8+4,0.7+1.8) circle (0.1cm); 	
\draw[very thin] (1.5+1.8+4,0.7+1.8) circle (0.1cm); 	
\draw[very thin] (1.5+1.8+4,1.5+1.8) circle (0.1cm); 	
\draw[very thin] (0.7+1.8+4,1.5+1.8) circle (0.1cm); 	
\fill[black!] (0.7,0.7+4) circle (0.1cm); 	
\fill[black!] (1.5,0.7+4) circle (0.1cm); 	
\fill[black!] (1.5,1.5+4) circle (0.1cm); 	
\fill[black!] (0.7,1.5+4) circle (0.1cm); 	
\fill[black!] (0.7+1.8,0.7+4) circle (0.1cm); 	
\fill[black!] (1.5+1.8,0.7+4) circle (0.1cm); 	
\fill[black!] (1.5+1.8,1.5+4) circle (0.1cm); 	
\fill[black!] (0.7+1.8,1.5+4) circle (0.1cm); 	
\fill[black!] (0.7,0.7+1.8+4) circle (0.1cm); 	
\fill[black!] (1.5,0.7+1.8+4) circle (0.1cm); 	
\fill[black!] (1.5,1.5+1.8+4) circle (0.1cm); 	
\fill[black!] (0.7,1.5+1.8+4) circle (0.1cm); 	
\fill[black!] (0.7+1.8,0.7+1.8+4) circle (0.1cm); 	
\fill[black!] (1.5+1.8,0.7+1.8+4) circle (0.1cm); 	
\fill[black!] (1.5+1.8,1.5+1.8+4) circle (0.1cm); 	
\fill[black!] (0.7+1.8,1.5+1.8+4) circle (0.1cm); 	
\fill[black!] (0.7+4,0.7+4) circle (0.1cm); 	
\fill[black!] (1.5+4,0.7+4) circle (0.1cm); 	
\fill[black!] (1.5+4,1.5+4) circle (0.1cm); 	
\fill[black!] (0.7+4,1.5+4) circle (0.1cm); 	
\fill[black!] (0.7+1.8+4,0.7+4) circle (0.1cm); 	
\fill[black!] (1.5+1.8+4,0.7+4) circle (0.1cm); 	
\fill[black!] (1.5+1.8+4,1.5+4) circle (0.1cm); 	
\fill[black!] (0.7+1.8+4,1.5+4) circle (0.1cm); 	
\fill[black!] (0.7+4,0.7+1.8+4) circle (0.1cm); 	
\fill[black!] (1.5+4,0.7+1.8+4) circle (0.1cm); 	
\fill[black!] (1.5+4,1.5+1.8+4) circle (0.1cm); 	
\fill[black!] (0.7+4,1.5+1.8+4) circle (0.1cm); 	
\fill[black!] (0.7+1.8+4,0.7+1.8+4) circle (0.1cm); 	
\fill[black!] (1.5+1.8+4,0.7+1.8+4) circle (0.1cm); 	
\fill[black!] (1.5+1.8+4,1.5+1.8+4) circle (0.1cm); 	
\fill[black!] (0.7+1.8+4,1.5+1.8+4) circle (0.1cm); 
\node[text width=2cm] at (2.1,4) {\tiny{$U_{a}$}};    
\begin{scope}[shift={(11,0)}]         
\draw[gray, very thin] (0.0,0.0) -- (8,0) -- (8,8) -- (0,8) -- (0,0); 	
\draw[gray, very thin] (0.1,0.1) -- (3.9,0.1) -- (3.9,3.9) -- (0.1,3.9) -- (0.1,0.1); 		
\draw[gray, dashed, very thin] (0.3,0.3) -- (1.9,0.3) -- (1.9,1.9) -- (0.3,1.9) -- (0.3,0.3); 		
\draw[gray, very thin] (2.1,0.3) -- (3.7,0.3) -- (3.7,1.9) -- (2.1,1.9) -- (2.1,0.3); 		
\draw[gray, very thin] (2.1,2.1) -- (3.7,2.1) -- (3.7,3.7) -- (2.1,3.7) -- (2.1,2.1); 		
\draw[gray, very thin] (0.3,2.1) -- (1.9,2.1) -- (1.9,3.7) -- (0.3,3.7) -- (0.3,2.1); 	
\draw[gray, very thin] (4.1,0.1) -- (7.9,0.1) -- (7.9,3.9) -- (4.1,3.9) -- (4.1,0.1); 		
\draw[gray, very thin] (0.3+4,0.3) -- (1.9+4,0.3) -- (1.9+4,1.9) -- (0.3+4,1.9) -- (0.3+4,0.3); 		
\draw[gray, very thin] (2.1+4,0.3) -- (3.7+4,0.3) -- (3.7+4,1.9) -- (2.1+4,1.9) -- (2.1+4,0.3); 		
\draw[gray, very thin] (2.1+4,2.1) -- (3.7+4,2.1) -- (3.7+4,3.7) -- (2.1+4,3.7) -- (2.1+4,2.1); 		
\draw[gray, very thin] (0.3+4,2.1) -- (1.9+4,2.1) -- (1.9+4,3.7) -- (0.3+4,3.7) -- (0.3+4,2.1); 	
\draw[gray, very thin] (0.1,4.1) -- (3.9,4.1) -- (3.9,7.9) -- (0.1,7.9) -- (0.1,4.1); 		
\draw[gray, very thin] (0.3,0.3+4) -- (1.9,0.3+4) -- (1.9,1.9+4) -- (0.3,1.9+4) -- (0.3,0.3+4); 		
\draw[gray, very thin] (2.1,0.3+4) -- (3.7,0.3+4) -- (3.7,1.9+4) -- (2.1,1.9+4) -- (2.1,0.3+4); 		
\draw[gray, very thin] (2.1,2.1+4) -- (3.7,2.1+4) -- (3.7,3.7+4) -- (2.1,3.7+4) -- (2.1,2.1+4); 		
\draw[gray, very thin] (0.3,2.1+4) -- (1.9,2.1+4) -- (1.9,3.7+4) -- (0.3,3.7+4) -- (0.3,2.1+4); 	
\draw[gray, very thin] (4.1,4.1) -- (7.9,4.1) -- (7.9,7.9) -- (4.1,7.9) -- (4.1,4.1); 		
\draw[gray, very thin] (0.3+4,0.3+4) -- (1.9+4,0.3+4) -- (1.9+4,1.9+4) -- (0.3+4,1.9+4) -- (0.3+4,0.3+4); 		\draw[gray, very thin] (2.1+4,0.3+4) -- (3.7+4,0.3+4) -- (3.7+4,1.9+4) -- (2.1+4,1.9+4) -- (2.1+4,0.3+4); 		\draw[gray, very thin] (2.1+4,2.1+4) -- (3.7+4,2.1+4) -- (3.7+4,3.7+4) -- (2.1+4,3.7+4) -- (2.1+4,2.1+4); 		\draw[gray, very thin] (0.3+4,2.1+4) -- (1.9+4,2.1+4) -- (1.9+4,3.7+4) -- (0.3+4,3.7+4) -- (0.3+4,2.1+4); 	\fill[black!] (0.7,0.7) circle (0.1cm);	 	
\draw[very thin] (1.5,0.7) circle (0.1cm); 	
\draw[very thin] (1.5,1.5) circle (0.1cm); 	
\fill[black!] (0.7,1.5) circle (0.1cm); 	
\draw[very thin] (0.7+1.8,0.7) circle (0.1cm); 	
\draw[very thin] (1.5+1.8,0.7) circle (0.1cm); 	
\draw[very thin] (1.5+1.8,1.5) circle (0.1cm); 	
\draw[very thin] (0.7+1.8,1.5) circle (0.1cm); 	
\draw[very thin] (0.7,0.7+1.8) circle (0.1cm); 	
\draw[very thin] (1.5,0.7+1.8) circle (0.1cm); 	
\draw[very thin] (1.5,1.5+1.8) circle (0.1cm); 	
\draw[very thin] (0.7,1.5+1.8) circle (0.1cm); 	
\draw[very thin] (0.7+1.8,0.7+1.8) circle (0.1cm); 	
\draw[very thin] (1.5+1.8,0.7+1.8) circle (0.1cm); 	
\draw[very thin] (1.5+1.8,1.5+1.8) circle (0.1cm); 	
\draw[very thin] (0.7+1.8,1.5+1.8) circle (0.1cm); 	
\draw[very thin] (0.7+4,0.7) circle (0.1cm); 	
\draw[very thin] (1.5+4,0.7) circle (0.1cm); 	
\draw[very thin] (1.5+4,1.5) circle (0.1cm); 	
\draw[very thin] (0.7+4,1.5) circle (0.1cm); 	
\draw[very thin] (0.7+1.8+4,0.7) circle (0.1cm); 	
\draw[very thin] (1.5+1.8+4,0.7) circle (0.1cm); 	
\draw[very thin] (1.5+1.8+4,1.5) circle (0.1cm); 	
\draw[very thin] (0.7+1.8+4,1.5) circle (0.1cm); 	
\draw[very thin] (0.7+4,0.7+1.8) circle (0.1cm); 	
\draw[very thin] (1.5+4,0.7+1.8) circle (0.1cm); 	
\draw[very thin] (1.5+4,1.5+1.8) circle (0.1cm); 	
\draw[very thin] (0.7+4,1.5+1.8) circle (0.1cm); 	
\draw[very thin] (0.7+1.8+4,0.7+1.8) circle (0.1cm); 	
\draw[very thin] (1.5+1.8+4,0.7+1.8) circle (0.1cm); 	
\draw[very thin] (1.5+1.8+4,1.5+1.8) circle (0.1cm); 	
\draw[very thin] (0.7+1.8+4,1.5+1.8) circle (0.1cm); 	
\draw[very thin] (0.7,0.7+4) circle (0.1cm); 	
\draw[very thin] (1.5,0.7+4) circle (0.1cm); 	
\draw[very thin] (1.5,1.5+4) circle (0.1cm); 	
\draw[very thin] (0.7,1.5+4) circle (0.1cm); 	
\draw[very thin] (0.7+1.8,0.7+4) circle (0.1cm); 	
\draw[very thin] (1.5+1.8,0.7+4) circle (0.1cm); 	
\draw[very thin] (1.5+1.8,1.5+4) circle (0.1cm); 	
\draw[very thin] (0.7+1.8,1.5+4) circle (0.1cm); 	
\draw[very thin] (0.7,0.7+1.8+4) circle (0.1cm); 	
\draw[very thin] (1.5,0.7+1.8+4) circle (0.1cm); 	
\draw[very thin] (1.5,1.5+1.8+4) circle (0.1cm); 
\draw[very thin] (0.7,1.5+1.8+4) circle (0.1cm); 
\draw[very thin] (0.7+1.8,0.7+1.8+4) circle (0.1cm); 	
\draw[very thin] (1.5+1.8,0.7+1.8+4) circle (0.1cm); 	
\draw[very thin] (1.5+1.8,1.5+1.8+4) circle (0.1cm); 	
\draw[very thin] (0.7+1.8,1.5+1.8+4) circle (0.1cm); 	
\draw[very thin] (0.7+4,0.7+4) circle (0.1cm); 	
\draw[very thin] (1.5+4,0.7+4) circle (0.1cm); 	
\draw[very thin] (1.5+4,1.5+4) circle (0.1cm); 	
\draw[very thin] (0.7+4,1.5+4) circle (0.1cm); 	
\draw[very thin] (0.7+1.8+4,0.7+4) circle (0.1cm); 	
\draw[very thin] (1.5+1.8+4,0.7+4) circle (0.1cm); 	
\draw[very thin] (1.5+1.8+4,1.5+4) circle (0.1cm); 	
\draw[very thin] (0.7+1.8+4,1.5+4) circle (0.1cm); 	
\draw[very thin] (0.7+4,0.7+1.8+4) circle (0.1cm); 	
\draw[very thin] (1.5+4,0.7+1.8+4) circle (0.1cm); 	
\draw[very thin] (1.5+4,1.5+1.8+4) circle (0.1cm); 	
\draw[very thin] (0.7+4,1.5+1.8+4) circle (0.1cm); 	
\draw[very thin] (0.7+1.8+4,0.7+1.8+4) circle (0.1cm); 	
\draw[very thin] (1.5+1.8+4,0.7+1.8+4) circle (0.1cm); 	
\draw[very thin] (1.5+1.8+4,1.5+1.8+4) circle (0.1cm); 	
\draw[very thin] (0.7+1.8+4,1.5+1.8+4) circle (0.1cm); 
\node[text width=2cm] at (0.8,4) {\tiny{$\cdots \, U_{b}$}}; 
\end{scope} 
\end{scope}         
 \end{tikzpicture} 
\caption{\small The transformation $\psi_{k} \to \psi_{k}'$. The blocks $\tilde{U}_{a}$ and 
$\tilde{U}_{b}$ are drawn with a dashed outline. Solid black circles represent elements of 
$\psi_{k}$ (i.e., vertices on which the configuration $\psi_{k}$ takes the value $+1$), while 
blank circles are elements of $\overline{\psi_{k}}$.} 
\label{fig:mapPsi} 
\end{figure}
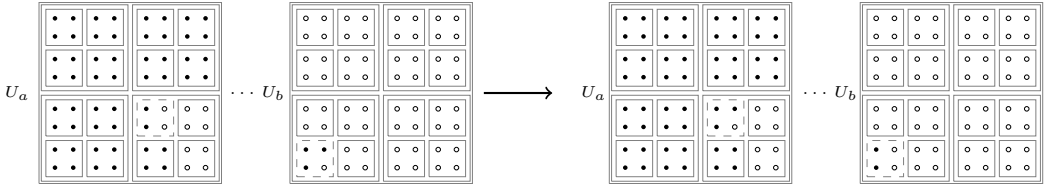

\bigskip\noindent
\underline{Non-decreasing case}: Given a configuration $\sigma$, we again apply 
a series of transformations involving switching and re-arranging of blocks in $\sigma$ 
(all of which are non-increasing in $\cH$) and ending with the configuration 
$\gamma_{\left|\sigma\right|}^{\mathrm{MI}}$. Firstly, through a series of re-arrangements, we 
may assume that $\sigma$ is \emph{left-aligned}: for any $0\leq k\leq n-1$ and any 
$k$-blocks $U_{i}$ and $U_{i+1}$ contained in the same $\left(k+1\right)$-block (a 
lower index on a block implies that it contains vertices that also have a lower index), 
we have $\left|U_{i}\cap\sigma\right|\geq\left|U_{i+1}\cap\sigma\right|$. It is clear that 
these re-arrangements are $\cH$-invariant.

Start with $k=n-1$ and check whether $\left|U_{1}\cap\sigma\right|\geq\left|U_{N}\cap
\sigma\right|+2$. If so, then switch the value at $v_{1}\in U_{1}$ (equal to $+1$) with the value at 
$v_{N^{n}}\in U_{N}$ (equal to $-1$). Denote the result of this switch by $\sigma'$. 
From Lemma~\ref{lem:switching} we have 
\begin{equation}
\cH\left(\sigma'\right)-\cH\left(\sigma\right)=\sum_{i=1}^{n-1}2
\left(J_{i}-J_{n}\right)\left[\left|A_{i}\right|-\left|C_{i}\right|\right]\left(0-1\right).
\end{equation}
Since $\sigma$ is left-aligned, we know that $\left|A_{n-1}\right|\leq\left|C_{n-1}\right|$.
Inductively it follows that $\left|A_{i}\right|\leq\left|C_{i}\right|$ for all $1\leq i\leq n-1$. 
Since, by the monotonicity, we also have $J_{i}-J_{n}\leq0$ for all $1\leq i\leq n-1$, it 
follows that $\cH\left(\sigma'\right)-\cH\left(\sigma\right)\leq0$. 

Next re-arrange $\sigma'$ to make it left-aligned (at no cost in $\cH$), and repeat 
this construction until $\left|U_{N}\cap\sigma\right|\leq\left|U_{1}\cap\sigma\right|\leq
\left|U_{N}\cap\sigma\right|+1$. Note that this takes a finite number of steps. Once this 
is accomplished, resume by recursively repeating the construction for $k=n-2$, within
each $n-1$-block, etc. This terminates with $\gamma_{\left|\sigma\right|}^{\mathrm{MI}}$.
\end{proof}


\subsection{Proof of Theorem \ref{thrm:hyp-H}}
\label{S2.3}

The proof is analogous to that given in \cite[Section 17.3.1]{BdH15}, and relies on the 
existence of a uniformly optimal path.

\begin{proof}
Let $\sigma\in\Omega\backslash\left\lbrace\boxminus,\boxplus\right\rbrace$. Find two vertices $v_{i},v_{j}\in  \Lambda_{N}^{n}$ such 
that $v_{i}\in\sigma$ and $v_{j}\notin\sigma$. By translation invariance, we can 
construct a uniformly optimal reference path $\gamma$ that is a translation (via some $d$-preserving bijection of $\Lambda_{N}^{n}$) of the path 
$\gamma^{\mathrm{MD}}$ in the non-increasing case and $\gamma^{\mathrm{MI}}$ in the non-decreasing case, and that satisfies 
$\gamma\left(1\right)=v_{j}$ and $\gamma\left(2\right)=v_{i}$. Note that in both cases 
\begin{equation}
\begin{aligned}
&\sigma\cap\gamma_{1}=\boxminus,\\
&1\leq\left|\sigma\cap\gamma_{k}\right|<k \quad \forall\, k\geq2.
\label{eq:intersectedpath}
\end{aligned}
\end{equation}
Furthermore, 
\begin{equation}
\cH\left(\sigma\cup\gamma_{1}\right)-\cH\left(\sigma\right)
=\sum_{\substack{w\neq v_{j}\\w\notin\sigma}}
J_{d\left(w,v_{j}\right)}-\sum_{\substack{w\neq v_{j}\\w\in\sigma}}
J_{d\left(w,v_{j}\right)}-h<\sum_{\substack{w\neq v_{j}}}
J_{d\left(w,v_{j}\right)}-h=\cH\left(\gamma_{1}\right)-\cH\left(\boxminus\right)
\label{eq:refpathcompare}
\end{equation}
where we use the fact that $J_{i}>0$, $1\leq i \leq n$. Similarly, if we let $k'=\min\left\{k\in\N\colon\,
\cH\left(\gamma_{k}\right) \leq\cH\left(\boxminus\right)\right\}$, then by (A1) it follows 
that $k'\geq 2$, and so for $2\leq k\leq k'$, 
\begin{equation}
\begin{aligned}
\cH\left(\sigma\cup\gamma_{k}\right)-\cH\left(\sigma\right) 
&= \sum_{\substack{w\in\gamma_{k}\backslash\sigma}}
\sum_{v\notin\sigma\cup\gamma_{k}}J_{d\left(w,v\right)}
-\sum_{w\in\gamma_{k}\backslash\sigma}\sum_{\substack{v\in\sigma}}
J_{d\left(w,v\right)}-h\left|\gamma_{k}\backslash\sigma\right|\\
& \leq \sum_{\substack{w\in\gamma_{k}\backslash\sigma}}
\sum_{v\notin\gamma_{k}}J_{d\left(w,v\right)}
-\sum_{w\in\gamma_{k}\backslash\sigma}\sum_{\substack{v\in\sigma\cap\gamma_{k}}}
J_{d\left(w,v\right)}-h\left|\gamma_{k}\backslash\sigma\right|\\
& = \cH\left(\gamma_{k}\right)-\cH\left(\gamma_{k}\cap\sigma\right) 
\leq \cH\left(\gamma_{k}\right)-\cH\left(\gamma_{\left|\gamma_{k}\cap\sigma\right|}\right)
<\cH\left(\gamma_{k}\right)-\cH\left(\boxminus\right),
\end{aligned}
\label{eq:refpath}
\end{equation}
where the last inequality follows from the fact that $\left|\gamma_{k}\cap\sigma\right|<k$  (by 
(\ref{eq:intersectedpath})) because $\gamma$ is uniformly optimal. Taking $k=k'$, 
we get from (\ref{eq:refpath}) that $\cH\left(\sigma\cup\gamma_{k'}\right)<\cH\left(\sigma\right)$, 
and hence that the stability level $V_{\sigma}$ of $\sigma$ defined in ~\ref{eq:Stability} satisfies
\begin{equation}
V_{\sigma}<\max_{1\leq k\leq k'}
\left\{ 0,\left(\cH\left(\gamma_{k}\right)-\cH\left(\boxminus\right)\right)\right\} \leq\Gamma^{\star}.
\end{equation}
This settles the claim because $V_{\boxminus}=\Gamma^{\star}$.
\end{proof}

\begin{remark}
\label{rem:no local minima} 
{\rm Note that if (A1) is not satisfied, or in other words if 
\begin{equation}
\left(1-\frac{1}{N}\right)\sum_{i=1}^{n}J_{i}N^{i}\leq h,
\label{eq:nometastability}
\end{equation}
then it follows from the inequality in ~\ref{eq:refpathcompare} (note that without (A1) this is not a strict inequality) that
\begin{equation}
\cH\left(\sigma\cup\gamma_{1}\right)-\cH\left(\sigma\right) \leq \cH\left(\gamma_{1}\right)-\cH\left(\boxminus\right) \leq 0,
\end{equation}
and hence $\sigma$ is not a local minimum of $\cH$. Since $\sigma$ is arbitrary, it follows that $\cH$ has no local minima. This again illustrates why assumption (A1) is needed.}
\end{remark}

\section{Non-increasing pair potential}
\label{S3}

In Section~\ref{S3.1} we prove a concavity property for the energy profile 
along the reference path inside hierarchical blocks (Lemma~\ref{lem:concave} 
below). In Section~\ref{S3.2} we show that the flucuations of the energy profile 
inside a hierarchical block are relatively small (Lemma~\ref{lem:shift-increment} 
below) and use this to prove Theorem~\ref{thm:Gamma-case1} in the hierarchical 
mean-field limit (Corollary~\ref{cor:Gammastar1} and Remark~\ref{mhat0} below). 
In Section~\ref{S3.3} we identify the critical configurations and check that 
the conditions in Lemma~\ref{lem:variational-lemma} are satisfied 
(Lemmas~\ref{lem:locmaxima}--\ref{lem:circumventiing-path} below). We use 
these results in Section~\ref{S3.4} to prove Theorem~\ref{thm:Kstar-case1} and 
in Section~\ref{S3.5} to prove Theorems~\ref{thm:Cstar-case1}--\ref{thm:Cstar-case2}. 


\subsection{Concavity along the reference path}
\label{S3.1}

From now on we will only consider the case where $\vec{J}$ is non-increasing. 
We will drop the superscript $\mathrm{MD}$ and denote the uniformly optimal path 
$\gamma^{\mathrm{MD}}$ defined in Section~\ref{S2} by $\gamma$. We observe 
that 
\begin{eqnarray}
\cH\left(\gamma_{k}\right)-\cH\left(\boxminus\right) 
& = & \sum_{i=1}^{k}\sum_{j=k+1}^{N^{n}}J_{d\left(v_{i},v_{j}\right)}-hk,\quad 1\leq k \leq N^{n},
\label{eq:HMDheight}
\end{eqnarray}
and it is not difficult to show that (\ref{eq:HMDheight}) can be written as
\begin{equation}
\begin{aligned}
\cH\left(\gamma_{k}\right)-\cH\left(\boxminus\right)
&=\sum_{i=1}^{n}J_{i}N^{i-1}\Bigg(k\,\mathrm{mod}\,N^{i}
\left(N-\left\lfloor \frac{k}{N^{i-1}}\right\rfloor\mathrm{mod}\,N-1\right)\\
&\qquad\qquad +\left(N^{i-1}-k\,\mathrm{mod}\,N^{i-1}\right)
\left\lfloor \frac{k}{N^{i-1}}\right\rfloor \mathrm{mod}\,N\Bigg)-hk.
\label{eq:HMDheight2}
\end{aligned}
\end{equation}
Hence the communication height between $\boxminus$ and $\boxplus$ is given by 
\begin{equation}
\begin{aligned}
\Gamma^{\star}
&=\max_{1\leq k \leq N^{n}}\Bigg\{ \sum_{i=1}^{n}J_{i}N^{i-1}\Bigg(k\,\mathrm{mod}\,N^{i}
\left(N-\left\lfloor \frac{k}{N^{i-1}}\right\rfloor \mathrm{mod}\,N-1\right)\\
&\qquad\qquad \qquad\qquad+\left(N^{i-1}-k\,\mathrm{mod}\,N^{i-1}\right)
\left\lfloor \frac{k}{N^{i-1}}\right\rfloor \mathrm{mod}\,N\Bigg)-hk\Bigg\}. 
\label{eq:gammastarMD}
\end{aligned}
\end{equation}
However, it is not clear from (\ref{eq:gammastarMD}) how $\Gamma^{\star}$ and the 
energy values along the path $\gamma$ depend on $\vec{J}$. We will therefore derive
$\Gamma^{\star}$ in a different way, obtaining a more insightful expression. 

Note that if $j<k$, then
\begin{equation}
\cH\left(\gamma_{k}\right)-\cH\left(\gamma_{j}\right)
=\sum_{i=j+1}^{k}\left(\sum_{s=k+1}^{N^{n}}J_{d\left(v_{i},v_{s}\right)}
-\sum_{s=1}^{j}J_{d\left(v_{i},v_{s}\right)}\right)-h\left(k-j\right).
\label{eq:genincrement}
\end{equation}
In particular, we observe that, for any $0\leq a\leq n-1$,
\begin{equation}
\cH\left(\gamma_{N^{a}}\right)-\cH\left(\gamma_{0}\right)
=\cH\left(\gamma_{N^{a}}\right)-\cH\left(\boxminus\right)
=\left(N-1\right)N^{a}\sum_{i=a}^{n-1}N^{i}J_{i+1}-hN^{a}.
\label{eq:base-increment}
\end{equation}
We are interested in the global maxima of the energy profile. In order to 
locate where these occur, we analyse the geometric properties of the 
sequence $\{\cH(\gamma_{i})\}_{i=0}^{N^n}$. The following result describes 
concave subsequences that appear in $\{\cH(\gamma_{i})\}_{i=0}^{N^n}$ 
(see Fig.~\ref{figplot-0}) and that will be used repeatedly in Section~\ref{S4} 
to locate the global maxima of the energy landscape.

\begin{lemma}
\label{lem:concave}
Suppose that $k=j+N^{a}$ and $l=k+N^{a}$ for some $a\geq0$ 
and $j\geq0$. Suppose that the three vertices $v_{j}$, $v_{k}$ 
and $v_{l}$ all lie in the same $\left(a+1\right)$-block. Then 
\begin{equation}
\left(\cH\left(\gamma_{k}\right)-\cH\left(\gamma_{j}\right)\right)
-\left(\cH\left(\gamma_{l}\right)-\cH\left(\gamma_{k}\right)\right)
=2J_{a+1}N^{2a}.
\end{equation}
\end{lemma}

\begin{proof}
Note that, for any $1\leq s\leq N^{a}$, $b\geq1$, $b\neq a+1$,
\begin{equation}
\left|\left\{t>j+N^{a}\colon\, d\left(v_{j+s},v_{t}\right)=b\right\} \right|
=\left|\left\{t>k+N^{a}\colon\, d\left(v_{k+s},v_{t}\right)=b\right\} \right|,
\end{equation}
while 
\begin{equation}
\left|\left\{t>j+N^{a}\colon\, d\left(v_{j+s},v_{t}\right)=a+1\right\} \right|
= \left|\left\{t>k+N^{a}\colon\, d\left(v_{k+s},v_{t}\right)=a+1\right\} \right|+N^{a}.
\end{equation}
Similarly, for $b\geq1$, $b\neq a+1$, 
\begin{equation}
\left|\left\{t\leq j\colon\, d\left(v_{j+s},v_{t}\right)=b\right\} \right|
=\left|\left\{t\leq k\colon\, d\left(v_{k+s},v_{t}\right)=b\right\} \right|,
\end{equation}
while 
\begin{equation}
\left|\left\{t\leq j\colon\, d\left(v_{j+s},v_{t}\right)=a+1\right\} \right|+N^{a}
=\left|\left\{t\leq k\colon\, d\left(v_{k+s},v_{t}\right)=a+1\right\} \right|.
\end{equation}
Hence, by rewriting the sum in (\ref{eq:genincrement}), we get 
\begin{equation}
\begin{aligned}
&\left(\cH\left(\gamma_{k}^{\mathrm{MD}}\right)-\cH\left(\gamma_{j}^{\mathrm{MD}}\right)\right)
-\left(\cH\left(\gamma_{l}^{\mathrm{MD}}\right)-\cH\left(\gamma_{k}^{\mathrm{MD}}\right)\right)\\
&= \left(\sum_{s=1}^{N^{a}}\sum_{b=1}^{n}J_{b}
\left|\left\{ t>j+N^{a}\colon\, d\left(v_{j+s},v_{t}\right)=b\right\} \right|-\sum_{s=1}^{N^{a}}
\sum_{b=1}^{n}J_{b}\left|\left\{ t\leq j\colon\, d\left(v_{j+s},v_{t}\right)=b\right\} \right|\right)\\
&-\left(\sum_{s=1}^{N^{a}}\sum_{b=1}^{n}J_{b}\left|\left\{ t>k+N^{a}:\, 
d\left(v_{k+s},v_{t}\right)=b\right\} \right|-\sum_{s=1}^{N^{a}}
\sum_{b=1}^{n}J_{b}\left|\left\{ t\leq k\colon\, d\left(v_{k+s},v_{t}\right)=b\right\} \right|\right)\\
& = 2J_{a+1}N^{2a}.
\label{eq:difference-increments}
\end{aligned}
\end{equation}
This shows that the energy profile along the path $\gamma$ is made up of periodic 
segments that are \emph{concave} (see Definition~\ref{def-symmetric&concave} below). 
\end{proof}

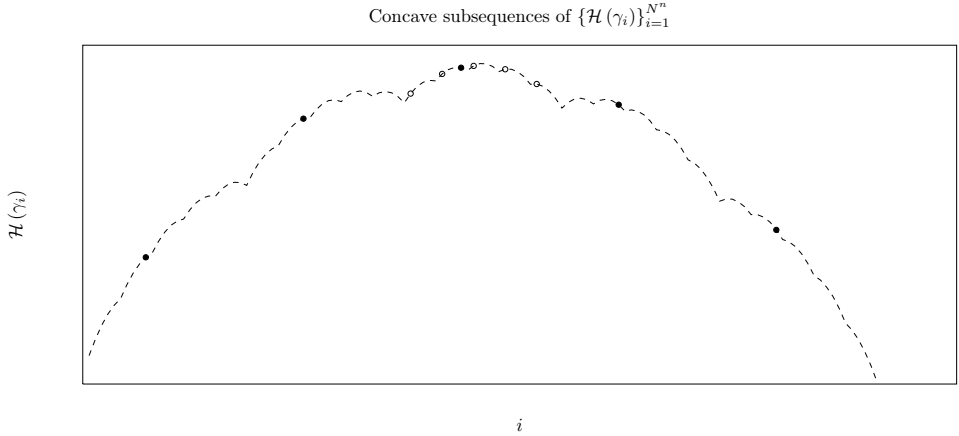
\begin{figure}[htbp]
\begin{tikzpicture} [scale = 0.7]
\begin{axis}[name=plot1,height=8cm,width=18cm,     	
title = {Concave subsequences of $\left\{\cH\left(\gamma_{i}\right)\right\}_{i=1}^{N^n}$}, xlabel = {$i$}, 
ylabel = {$\cH\left(\gamma_{i}\right)$}, ymin = -110, ymax = 1200, xmin = 0, ticks=none,]
\addplot[very thin, dashed, black] table {concavedata.txt};
\draw[fill,black] (axis cs:10,380) circle (0.05cm);
\draw[fill,black] (axis cs:35,916) circle (0.05cm);;
\draw[fill,black] (axis cs:60,1113) circle (0.05cm);
\draw[fill,black] (axis cs:85,970)circle (0.05cm);
\draw[fill,black] (axis cs:110,486)circle (0.05cm);
\draw[black] (axis cs:52,1013) circle (0.05cm);
\draw[black] (axis cs:57,1089) circle (0.05cm);
\draw[black] (axis cs:62,1120) circle (0.05cm);
\draw[black] (axis cs:67,1107) circle (0.05cm);
\draw[black] (axis cs:72,1050) circle (0.05cm);
\end{axis}
\end{tikzpicture}
\caption{\small The solid circles represent a periodic subsequence of 
$\left\{\cH\left(\gamma_{i}\right)\right\}_{i=0}^{N^{n}}$ of period $N^{n-1}$, 
while the hollow circles represent points of period $N^{n-2}$ that are 
contained within the same $\left(n-1\right)$-block.}
\label{figplot-0}
\end{figure}


\subsection{Hierarchical mean-field limit}
\label{S3.2}

The hierarchical mean-field limit corresponds to letting the hierarchical dimension 
$N$ tend to infinity while keeing the hierarchical height $n$ fixed. We will show 
that, under certain assumptions on the rate of decay of the sequence 
$\left\{J_{i}\right\}_{i=1}^{n}$, in the hierarchical mean-field limit the sequence 
$\{\cH(\gamma_{i})\} _{i=0}^{N^n}$ attains its global maximum at a location that 
is close to a multiple (by some factor in $\left\{1,\ldots,N\right\}$) of the largest 
block size where the corresponding configuration has energy larger than 
$\cH\left(\boxminus\right)$. We define this explicitly as follows. 

Recall from (\ref{eq:defm}) that 
\begin{eqnarray}
\hat{m} & = & \max\left\{ 0\leq m\leq n-1\colon\,\left(1-\frac{1}{N}\right)
\sum_{i=m+1}^{n}J_{i}N^{i}>h\right\} \nonumber \\
& = & \max\left\{ 0\leq m\leq n-1\colon\,\cH\left(\gamma_{N^{m}}\right)
\geq\cH\left(\boxminus\right)\right\}, 
\label{eq:defm2}
\end{eqnarray}
where the second line follows from (\ref{eq:base-increment}).  

From Lemma~\ref{lem:concave} it follows that, for all $M>\hat{m}$ and all $1\leq s\leq N-1$, 
$\cH\left(\gamma_{sN^{M}}\right)<\cH\left(\boxminus\right)$. Note also that, by 
Lemma~\ref{lem:concave} and equation (\ref{eq:base-increment}), we define 
\begin{eqnarray}
\alpha_{\hat{m},s} 
&=&\cH\left(\gamma_{sN^{\hat{m}}}\right)-\cH\left(\boxminus\right) 
\nonumber\\
& = & \sum_{i=0}^{s-2}\left(\cH(\gamma_{\left(s-i\right)N^{\hat{m}}})
-\cH(\gamma_{\left(s-i-1\right)N^{\hat{m}}})\right)
+\cH\left(\gamma_{N^{\hat{m}}}\right)-\cH\left(\gamma_{0}\right)
\nonumber \\
& = & s\left[\big(\cH\left(\gamma_{N^{\hat{m}}}\right)-\cH\left(\gamma_{0}\right)\big)
-\left(s-1\right)J_{\hat{m}+1}N^{2\hat{m}}\right]
\nonumber \\
& = & sN^{\hat{m}}\left[\left(1-\frac{1}{N}\right)
\sum_{k=\hat{m}}^{n-1}J_{k+1}N^{k+1}-h-\left(s-1\right)J_{\hat{m}+1}N^{\hat{m}}\right].
\label{eq:alphavalues}
\end{eqnarray}
Increments of values given by (\ref{eq:alphavalues}) are equal to
\begin{eqnarray}
\alpha_{\hat{m},s+1}-\alpha_{\hat{m},s} 
& = & N^{\hat{m}}\left[\left(1-\frac{1}{N}\right)\sum_{k=\hat{m}}^{n-1}J_{k+1}N^{k+1}
-h-2sJ_{\hat{m}+1}N^{\hat{m}}\right].
\label{eq:alphaincrement}
\end{eqnarray}
By the concavity implied by Lemma~\ref{lem:concave}, we have that $\alpha_{\hat{m},s+1}
-\alpha_{\hat{m},s}\leq0$ if and only if $s\geq\hat{s}$, where $\hat{s}$ is defined 
in (\ref{eq:sineq}). Under Assumption (A1)(a) it is easy to see that the sequence 
$\left\{ \cH\left(\gamma_{sN^{n-1}}\right)-\cH\left(\boxminus\right)\right\} _{s=0}^{N}$ attains 
a unique maximum at  $1\leq\left\lceil \hat{s}\right\rceil <N$, with value 
\begin{eqnarray}
\cH(\gamma_{\left\lceil \hat{s}\right\rceil N^{\hat{m}}})
-\cH\left(\boxminus\right) 
& = & \left\lceil \hat{s}\right\rceil \left(2\hat{s}
-\left\lceil \hat{s}\right\rceil +1\right)J_{\hat{m}+1}N^{2\hat{m}}.
\label{eq:maxHsubpath}
\end{eqnarray}
Furthermore, we claim that for any $N<t\leq N^{n-\hat{m}}$, $\cH(\gamma_{tN^{\hat{m}}})
<\cH(\gamma_{\left\lceil \hat{s}\right\rceil N^{\hat{m}}})$. Indeed, define $\bar{d}=
d(v_{\left\lceil \hat{s}\right\rceil N^{\hat{m}}},v_{tN^{\hat{m}}})>\hat{m}$, and note that 
$tN^{\hat{m}}=\eta N^{\dot{d}}+sN^{\hat{m}}$ for some $0\leq\eta,s<N$. Hence 
\begin{equation}
\begin{aligned}
\cH\left(\gamma_{tN^{\hat{m}}}\right)-\cH\left(\boxminus\right)
& = \cH(\gamma_{\eta N^{\bar{d}}})-\cH\left(\boxminus\right)
+\cH\left(\gamma_{tN^{\hat{m}}}\right)-\cH(\gamma_{\eta N^{\bar{d}}})\\
& \leq \cH\left(\gamma_{tN^{\hat{m}}}\right)-\cH(\gamma_{\eta N^{\bar{d}}})\\
& = sN^{\hat{m}}\left[\left(1-\frac{1}{N}\right)\sum_{k=\hat{m}+1}^{n}J_{k}N^{k}
-h-\left(s-1\right)J_{\hat{m}+1}N^{\hat{m}}-\eta J_{\bar{d}+1}N^{\dot{d}}\right]\\
& < \cH\left(\gamma_{sN^{\hat{m}}}\right)-\cH\left(\boxminus\right)
\leq\cH(\gamma_{\left\lceil \hat{s}\right\rceil N^{\hat{m}}})-\cH\left(\boxminus\right),
\end{aligned}
\label{eq:allatlevelm}
\end{equation}
where the first inequality follows from the definition of $\hat{m}$ and the fact that $\bar{d}>\hat{m}$. 

We next show that fluctuations in energy $\left|\cH\left(\gamma_{i}\right)-\cH\left(\gamma_{j}\right)
\right|$ for $\left|i-j\right|\leq N^{\hat{m}}$ are relatively small compared to  
$\cH\left(\gamma_{\left\lceil \hat{s}\right\rceil }\right)-\cH\left(\boxminus\right)$.

\begin{lemma}
\label{lem:shift-increment}
Let $k=\sum_{i=0}^{s} a_{i}N^{i}$ with $0\leq a_{i}\leq N-1$, and let $M=\sum_{i=t}^{n-1} a_{i}N^{i}$ with $0\leq b_{i}\leq N-1$ and $n-1\geq t>s$. Then 
\begin{equation}
\cH\left(\gamma_{M+k}\right)-\cH\left(\gamma_{M}\right)
\leq\cH\left(\gamma_{k}\right)-\cH\left(\boxminus\right)
\end{equation}
and 
\begin{equation}
\left|\cH\left(\gamma_{M+k}\right)-\cH\left(\gamma_{M}\right)\right|
\leq\left|\cH\left(\gamma_{k}\right)-\cH\left(\boxminus\right)\right|+hk.
\end{equation}
\end{lemma}

\begin{proof}
Note that, during the move from $\gamma_{M}$ to $\gamma_{M+k}$, the total change in 
energy due to interacting pairs at distance $i$ is given by $\left(1-\frac{1}{N}\right)k
\sum_{i=s+2}^{t}J_{i}N^{i}$ for $s+2\leq i\leq t$, while for $i\geq t+1$ it is given by 
$k\sum_{i=t}^{n-1}J_{i+1}N^{i}\left(N-2b_{i}-1\right)$. Now, for $1\leq i\leq s+1$, this 
change is equal to 
\begin{equation}
J_{1}N^{0}a_{0}\left(N-a_{0}\right)+\sum_{i=1}^{s}J_{i+1}N^{i}\left(\left(N-a_{i}-1\right)
\left(\sum_{j=0}^{i}a_{j}N^{j}\right)+a_{i}\left(N^{i}-\sum_{j=0}^{i-1}a_{j}N^{j}\right)\right),
\end{equation}
which is also the same during the move from $\gamma_{\boxminus}$ to $\gamma_{k}$.
Thus, we get 
\begin{equation}
\begin{aligned}
&\cH\left(\gamma_{M+k}\right)-\cH\left(\gamma_{M}\right)\\ 
& = \sum_{i=0}^{s}J_{i+1}N^{i}\left(\left(N-a_{i}-1\right)\left(\sum_{j=0}^{i}a_{j}N^{j}\right)
+ a_{i}\left(N^{i}-\sum_{j=0}^{i-1}a_{j}N^{j}\right)\right)\\
&\qquad + \left(1-\frac{1}{N}\right)k\sum_{i=s+2}^{t}J_{i}N^{i}+k\sum_{i=t}^{n-1}J_{i+1}N^{i}
\left(N-2b_{i}-1\right)-hk\\
& \leq \sum_{i=0}^{s}J_{i+1}N^{i}\left(\left(N-a_{i}-1\right)\left(\sum_{j=0}^{i}a_{j}N^{j}\right)
+a_{i}\left(N^{i}-\sum_{j=0}^{i-1}a_{j}N^{j}\right)\right)\\
&\qquad  + \left(1-\frac{1}{N}\right)k\sum_{i=s+2}^{n}J_{i}N^{i}-hk\\
&= \cH\left(\gamma_{k}\right)-\cH\left(\gamma_{\boxminus}\right).
\label{eq:gamma-m+k}
\end{aligned}
\end{equation}
Note, furthermore, that the right-hand side of the first line of (\ref{eq:gamma-m+k}) is 
non-negative, as is the first sum in the second line and both sums in the third line.
Making use of the triangle inequality, we get the second claim of the lemma. 
\end{proof}

We will assume for now that $\hat{m}\geq1$ and consider the case $\hat{m}=0$ in 
Remark~\ref{mhat0}. It follows from Lemma~\ref{lem:shift-increment} and Assumption 
(A3) that, for any $0 \leq k < N^{\hat{m}}$ and $\ell \geq 1$, 
\begin{eqnarray}
\frac{|\cH(\gamma_{k+\ell N^{\hat{m}}})-\cH(\gamma_{\ell N^{\hat{m}}})|}
{|\cH(\gamma_{\left\lceil \hat{s}\right\rceil N^{\hat{m}}})-\cH(\boxminus)|} 
& \leq & \frac{|\cH(\gamma_{k})-\cH(\gamma_{\boxminus})|+hk}
{|\cH(\gamma_{\left\lceil \hat{s}\right\rceil N^{\hat{m}}})
-\cH(\boxminus)|} \to 0 \quad \mbox{as} \quad N\to\infty,
\label{eq:check A2}
\end{eqnarray}
since from (\ref{eq:gamma-m+k}) we see that the numerator in the right-hand 
side of (\ref{eq:check A2}) equals the numerator in the condition of Assumption 
(A3), and from (\ref{eq:alphavalues}) the same follows for the denominator. Thus,
using (\ref{eq:alphavalues}) we conclude the following.

\begin{corollary}[Proof of Theorem \ref{thm:Gamma-case1}]
\label{cor:Gammastar1}
Suppose that Assumption {\rm (A2)} holds. Then 
\begin{eqnarray}
\Gamma^{\star} 
& = & \left[1+o_{N}\left(1\right)\right]\left(\cH\left(\gamma_{\left\lceil \hat{s} \right\rceil N^{\hat{m}}}\right)
-\cH\left(\boxminus\right)\right)\label{eq:Gstarhlimit}\\
& = & \left[1+o_{N}\left(1\right)\right]\left\lceil \hat{s}\right\rceil 
\left(2\hat{s}-\left\lceil \hat{s}\right\rceil +1\right)J_{\hat{m}+1}
N^{2\hat{m}}\nonumber \\
& = & \left[1+o_{N}\left(1\right)\right]\hat{s}^{2}J_{\hat{m}+1}N^{2\hat{m}}.
\nonumber 
\end{eqnarray}
\end{corollary}

\begin{remark}
\label{mhat0}
{\rm The special case $\hat{m}=0$ can be considered seperately. By 
Lemma~\ref{lem:shift-increment} it follows, for any $0\leq t\leq N^{n}$ and with 
\begin{eqnarray}
\hat{s} & = & \left(2J_{1}\right)^{-1}\left[\left(1-\frac{1}{N}\right)
\sum_{i=0}^{n-1}J_{i+1}N^{i+1}-h\right],
\end{eqnarray} 
that
\begin{equation}
\cH\left(\gamma_{t}\right)-\cH\left(\boxminus\right)
\leq\cH\left(\gamma_{\left\lceil \hat{s}\right\rceil }\right)-\cH\left(\boxminus\right)
\end{equation}
and hence $\Gamma^{\star}=\cH(\gamma_{\left\lceil \hat{s}\right\rceil})
-\cH\left(\boxminus\right)=\left\lceil \hat{s}\right\rceil \left(2\hat{s}
-\left\lceil \hat{s}\right\rceil +1\right)J_{1}$.} 
\end{remark}


\subsection{Critical configurations}
\label{S3.3}

It is clear from (\ref{eq:variationalform}) that the prefactor $K^{\star}$ 
is closely related to the set of critical configurations $\cC^{\star}$, in particular, 
the cardinality of this set. The symmetry of $\Lambda_{N}^{n}$ implies that 
the image of any critical configuration under an isometric translation is also a critical configuration. Thus, we have to count the number of isometries 
that result in distinct elements of $\cC^{\star}$, which is a problem related to the 
$N$-ary decomposition of the size of a critical configuration. To do so, we first 
establish a result that determines the $N$-ary decomposition of any global 
maximum subject to Assumption (A3).

The following lemma gives us the asymptotic value of the terms in the $N$-ary 
decomposition of the size of a critical configuration.

\begin{lemma}
\label{lem:locmaxima} 
Suppose that {\rm (A1)--(A4)} holds, and that the path $\gamma$ attains a global 
maximum at $\gamma_{M}$. Let 
\begin{equation}
M=a_{n-1}N^{n-1}+\ldots+k_{1}N+a_{0}
\end{equation}
be the $N$-ary decomposition of the integer $M$. Then 
\begin{equation}
\lim_{N\to\infty}\frac{1}{N}\sum_{i=0}^{n-1}\left|a_{i}-\eta_{i}\right|=0,
\end{equation}
where $\eta_{i}=0$ for $\hat{m}< i \leq n-1$, $\eta_{\hat{m}}=\left\lceil \hat{s}\right\rceil$, 
and $\eta_{\hat{m}-1},\ldots,\eta_{0}$ are defined in \eqref{eq:zetamhat-1} and 
\eqref{eq:zetai+1} below. 
\end{lemma}

\begin{proof}
From the definition of $\hat{m}$ in (\ref{eq:defm}) and the argument leading up to 
(\ref{eq:Gstarhlimit}), it is clear that $\lim_{N\to\infty}\left|a_{i}\right|=0$ for $i>\hat{m}$. 
Let $\eta_{\hat{m}}=\left\lceil \hat{s}\right\rceil $. Then, for any $0\leq\sigma<N$,
\begin{equation}
\begin{aligned}
&\cH\left(\gamma_{\eta_{\hat{m}}N^{\hat{m}}+\sigma N^{\hat{m}-1}}\right)
-\cH\left(\gamma_{\eta_{\hat{m}}N^{\hat{m}}}\right)\\
& = J_{\hat{m}}N^{2\hat{m}-2}\sigma\left(N-\sigma\right)+J_{\hat{m}+1}N^{2\hat{m}-1}
\sigma\left(N-\eta_{\hat{m}}-1\right)\\
& \qquad +\left(1-\frac{1}{N}\right)\sum_{i=\hat{m}+2}^{n}\sigma J_{i}N^{\hat{m}-1+i}
-J_{\hat{m}+1}\eta_{\hat{m}}\sigma N^{2\hat{m}-1}-h\sigma N^{\hat{m}-1}\\
& = J_{\hat{m}}N^{2\hat{m}-2}\sigma\left(N-\sigma\right)+J_{\hat{m}+1}N^{2\hat{m}-1}
\sigma\left(N-2\eta_{\hat{m}}-1\right)\\
& \qquad +\left(1-\frac{1}{N}\right)\sum_{i=\hat{m}+2}^{n}\sigma J_{i}N^{\hat{m}-1+i}
-h\sigma N^{\hat{m}-1}.
\end{aligned}
\label{eq:diffinzetas}
\end{equation}
By the concavity in Lemma~\ref{lem:concave}, $\cH(\gamma_{\eta_{\hat{m}}
N^{\hat{m}}+(\sigma+1)N^{\hat{m}-1}})-\cH(\gamma_{\eta_{\hat{m}}
N^{\hat{m}}+\sigma N^{\hat{m}-1}})\leq 0$ if and only if
\begin{equation}
\begin{aligned}
0\,\vee\,&\Bigg\lceil \frac{1}{2}\Bigg(\Bigg(\Big(\frac{J_{\hat{m}+1}}{J_{\hat{m}}}\Big)
N\big(N-2\eta_{\hat{m}}-1\big)\\
&\qquad +\Big(1-\frac{1}{N}\Big)
\sum_{i=2}^{n-\hat{m}}\Big(\frac{J_{\hat{m}+i}}{J_{\hat{m}}}\Big)N^{i+1}
-\frac{h}{J_{\hat{m}}N^{\hat{m}-1}}\Bigg)+N-1\Bigg)\Bigg\rceil 
=\eta_{\hat{m}-1}\leq\sigma.
\label{eq:zetamhat-1}
\end{aligned}
\end{equation}
Observe that (\ref{eq:zetamhat-1}) is continuous in $\eta_{\hat{m}}$. Hence, if $\varphi_{\hat{m}}
\in \left[\left\lceil \hat{s}\right\rceil \left(1-\epsilon\right),\left\lceil \hat{s}\right\rceil 
\left(1+\epsilon\right)\right]$ for some $\epsilon>0$, and $\varphi_{\hat{m}-1}$ is equal to 
(\ref{eq:zetamhat-1}) with $\eta_{\hat{m}}$ replaced by $\varphi_{\hat{m}}$, then 
\begin{equation}
\frac{1}{N}\left|\eta_{\hat{m}-1}-\varphi_{\hat{m}-1}\right|
\leq\left(\frac{J_{\hat{m}+1}}{J_{\hat{m}}}\right)\left|\eta_{\hat{m}}-\varphi_{\hat{m}}\right|
=\epsilon O\left(1\right)+\frac{2}{N}.
\label{eq:zeta-varphi}
\end{equation}
Since we already know from the reasoning leading up to Corollary~\ref{cor:Gammastar1}
that any global maximum $M$ must satisfy $a_{i}=0$ for $i>\hat{m}$ and $a_{\hat{m}}
\in\left[\left\lceil \hat{s}\right\rceil \left(1-\epsilon\right),\left\lceil \hat{s}\right\rceil
\left(1+\epsilon\right)\right]$, by (\ref{eq:zeta-varphi}) we also have that $a_{\hat{m}-1} \in
\left[\eta_{\hat{m}-1}\left(1-\epsilon'\right),\eta_{\hat{m}-1}\left(1+\epsilon'
\right)\right]$, with $\epsilon'$ allowed to be arbitrarily small as $N\to\infty$.
We can now repeat these computations recursively, to conclude the same for 
$a_{\hat{m}-2},\ldots,a_{0}$. 

Given $\eta_{\hat{m}},\ldots,\eta_{\hat{m}-i}$, let $0\leq\sigma<N$ and $s\left(i,j\right)=\sum_{t=0}^{i}\eta_{\hat{m}-t}N^{\hat{m}-t}+jN^{\hat{m}-i-1}$, and note that 
\begin{eqnarray}
\cH\left(\gamma_{s(i,\sigma)}\right) 
&=&  \cH\left(\gamma_{s(i,0)}\right)
+J_{\hat{m}-i}N^{2\left(\hat{m}-i-1\right)}\sigma\left(N-\sigma\right)\\
&& + \sum_{j=1}^{i+1}J_{\hat{m}-i+j}N^{2\left(\hat{m}-i-1\right)+j}
\sigma\left(N-2\eta_{\hat{m}-i+j-1}-1\right) \nonumber \\
&& + \left(1-\frac{1}{N}\right)\sum_{j=\hat{m}+2}^{n}\sigma J_{j}
N^{\hat{m}-i-1+j}-h\sigma N^{\hat{m}-i-1}. \nonumber
\end{eqnarray}
Thus, we have 
\begin{eqnarray}
\cH\left(\gamma_{s(i,\sigma +1)}\right) 
& = & \cH\left(\gamma_{s(i,\sigma)}\right)
+J_{\hat{m}-i}N^{2\left(\hat{m}-i-1\right)}\left(N-2\sigma-1\right)\\
&& +  \sum_{j=1}^{i+1}J_{\hat{m}-i+j}N^{2\left(\hat{m}-i-1\right)+j}
\left(N-2\eta_{\hat{m}-i+j-1}-1\right) \nonumber \\
&& +  \left(1-\frac{1}{N}\right)\sum_{j=\hat{m}+2}^{n}J_{j}N^{\hat{m}-i-1+j} -hN^{\hat{m}-i-1}, \nonumber
\label{eq:zetadiff2}
\end{eqnarray}
and hence $\cH(\gamma_{s(i,\sigma +1)})
-\cH(\gamma_{s(i,\sigma)})\leq0$ whenever 
\begin{equation}
\begin{aligned}
&0\vee \Bigg\lceil \frac{1}{2}\Bigg(\Bigg(\sum_{j=1}^{i+1}
\Big(\frac{J_{\hat{m}-i+j}}{J_{\hat{m}-i}}\Big)N^{j}\big(N-2\eta_{\hat{m}-i}-1\big)\\
&\qquad +\Big(1-\frac{1}{N}\Big) \sum_{j=2}^{n-\hat{m}}
\Big(\frac{J_{\hat{m}+j}}{J_{\hat{m}-i}}\Big)N^{i+j+1}
-\frac{h}{J_{\hat{m}-i}N^{\hat{m}-i-1}}\Bigg)+N-1\Bigg)\Bigg\rceil 
=\eta_{\hat{m}-i-1}\leq\sigma.
\label{eq:zetai+1}
\end{aligned}
\end{equation}
Again it follows that if $\varphi_{\hat{m}-i}\in\left\{ 0,\ldots,N-1\right\}$ and $\varphi_{\hat{m}-i-1}$ 
is equal to the left-hand side of (\ref{eq:zetai+1}) with $\eta_{\hat{m}-i}$ replaced by 
$\varphi_{\hat{m}-i}$ in (\ref{eq:zetai+1}), then 
\begin{equation}
\left|\eta_{\hat{m}-i-1}-\varphi_{\hat{m}-i-1}\right|
\leq\left(\frac{J_{\hat{m}-i+1}}{J_{\hat{m}-i}}\right)\left|\eta_{\hat{m}-i}
-\varphi_{\hat{m}-i}\right|+\frac{2}{N}.
\end{equation}
This proves the statement of the lemma.
\end{proof}

We need to look at the change in energy when we go from a critical
configuration in the set $\cC^{\star}$ to a neighbouring configuration obtained 
by changing the sign at one vertex. Our next observation concerns the sets 
$U_{\sigma}^{-}$ and $U_{\sigma}^{+}$ defined in the statement of 
Lemma~\ref{lem:variational-lemma}. 

\begin{lemma}
\label{lem:circumventiing-path}
Suppose that (A1) holds and that every $\xi\in\cC^{\star}$ has the same volume $\left|\xi\right|=k^{\star}$, 
and that every configuration of volume $k^{\star}$ has energy at least $\Phi\left(\boxminus,\boxplus\right)$. Suppose furthermore that for every configuration $\sigma \in U_{\xi}^{+}$, $\cH\left(\sigma\right) \neq \Phi\left(\boxminus,\boxplus\right)$. Then \eqref{eq:condition-lemma} is satisfied.
\end{lemma}

\begin{proof}
Let $\xi\in\cC^{\star}$, and suppose that $\sigma\in U_{\xi}^{-}$, so that $\sigma
=\xi\backslash\left\{ v_{a}\right\}$ for some $a<k^{\star}$. If $\sigma$ lies on some 
optimal path, then, by the assumption that this path has a unique maximum, 
(\ref{eq:condition-lemma}) is satisfied. Else, since $\xi$ lies on an optimal path, 
there exists some configuration $\xi'=\xi\backslash\left\{ v_{b}\right\}$ on 
the same path, of volume $\left|\xi'\right|=k^{\star}-1$ (note that by (A1) $k^{\star}>0$) and with $\Phi\left(\xi',\boxminus\right)<\Phi\left(\boxminus,\boxplus\right)$. We claim that the 
path $\sigma\to\sigma\cap\xi'\to\xi'$ stays strictly below $\Phi\left(\boxminus,\boxplus\right)$, which proves the statement of the lemma. Since by definition 
$\cH\left(\sigma\right)<\Phi\left(\boxminus,\boxplus\right)$ and $\cH\left(\xi'\right)<\Phi\left(\boxminus,\boxplus\right)$, we only need to show that 
$\cH\left(\sigma\cap\xi'\right)<\Phi\left(\boxminus,\boxplus\right)$.
However, note that 
\begin{equation}
\begin{aligned}
\cH\left(\sigma\cap\xi'\right)-\cH\left(\sigma\right) 
&= \sum_{\substack{i\leq k^{\star}\\i\neq b,a}}
J_{d\left(v_{b},v_{i}\right)}-\sum_{\substack{i>k^{\star}}}
J_{d\left(v_{b},v_{i}\right)}-J_{d\left(v_{b},v_{a}\right)}+h\\
& \leq \sum_{\substack{i\leq k^{\star}\\ i\neq b}}
J_{d\left(v_{b},v_{i}\right)}-\sum_{\substack{i>k^{\star}}}
J_{d\left(v_{b},v_{i}\right)}+h
= \cH\left(\xi'\right)-\cH\left(\xi\right)
< 0,
\end{aligned}
\end{equation}
where the last inequality uses the fact that $\Phi\left(\xi',\boxminus\right)<\Phi\left(\xi,\boxminus\right)$. This proves the claim for $\sigma \in U_{\xi}^{-}$. The argument for  $\sigma \in U_{\xi}^{+}$ makes use of the fact that by assumption $\cH\left(\sigma\right)\neq \Phi\left(\boxminus,\boxplus\right)$, and is otherwise identical to the argument above.
\end{proof}

Next, let us first consider any configuration $\gamma_{k}$ lying on the path 
$\gamma$, with $k=a_{n-1}N^{n-1}+\ldots+a_{0}$, and let $\sigma_{b}$ be 
a configuration obtained from $\gamma_{k}$ by flipping the sign at a vertex 
$w$ such that $d\left(w,v_{k}\right)=b$ for $b\in\left\{ 1,\ldots,n\right\} $. Note 
that by symmetry it makes no difference which particular vertex we select. If 
$\sigma_{b}\left(w\right)=-\gamma_{k}\left(w\right)=-1$, then for $b=1$ we have 
\begin{equation}
\cH\left(\sigma_{b}\right)-\cH\left(\gamma_{k}\right)=J_{1}\left(2a_{0}-N-1\right)
+\sum_{i=1}^{n-1}J_{i+1}N^{i}\left(2a_{i}-N+1\right)+h,
\label{eq:vertexflip-s1}
\end{equation}
while for $2\leq b\leq n$,
\begin{equation}
\begin{aligned}
&\cH\left(\sigma_{b}\right)-\cH\left(\gamma_{k}\right)\\
&=\sum_{i=1}^{b-1}J_{i}N^{i}\left(1-\frac{1}{N}\right)
+J_{b}\left(2\sum_{i=0}^{b-1}a_{i}N^{i}-N^{b}-N^{b-1}\right)
+\sum_{i=b}^{n-1}J_{i+1}N^{i}\left(2a_{i}-N+1\right)+h.
\label{eq:vertexflip-s2}
\end{aligned}
\end{equation}
Similarly, if $\sigma_{b}\left(w\right)=-\gamma_{k}\left(w\right)=+1$, then for 
$b=1$ we have 
\begin{equation}
\cH\left(\sigma_{b}\right)-\cH\left(\gamma_{k}\right)
=\sum_{i=0}^{n-1}J_{i+1}N^{i}\left(N-2a_{i}-1\right)-h,
\label{eq:vertexflip+s1}
\end{equation}
while for $2\leq b\leq n$,
\begin{equation}
\begin{aligned}
&\cH\left(\sigma_{b}\right)-\cH\left(\gamma_{k}\right)\\
&=\sum_{i=1}^{b-1}J_{i}N^{i}\left(1-\frac{1}{N}\right)+J_{b}\left(N^{b}
-2\sum_{i=0}^{b-1}a_{i}N^{i}-N^{b-1}\right)
+\sum_{i=b}^{n-1}J_{i+1}N^{i}\left(N-2a_{i}-1\right)-h.
\label{eq:vertexflip+s2}
\end{aligned}
\end{equation}

Under Assumption (A5), $\left\{ \cH\left(\gamma_{i}\right)\right\}_{i=0}^{N^{n}}$ 
attains a unique maximum. Indeed, this is immediate from (\ref{eq:genincrement}). 
Furthermore, from Assumption (A4) it follows that $\vec{J}$ is strictly monotone, 
and hence by Lemma~\ref{lem: unif opt. path} the path $\gamma$ is strictly optimal. 
This implies that all $\sigma\in\cC^{\star}$ must have the same volume, and that 
every other configuration of that volume has larger energy. Hence the conditions of 
Lemma~\ref{lem:circumventiing-path} are met. 


\subsection{Proof of Theorem~\ref{thm:Kstar-case1}}
\label{S3.4}

Define 
\begin{equation}
\begin{aligned}
&B_{d} = \Bigg\{ 1\leq b\leq\hat{m}\colon\,
\sum_{i=1}^{b-1}J_{i}N^{i}\Big(1-\frac{1}{N}\Big)
+J_{b}\Bigg(2\sum_{i=0}^{b-1}\eta_{i}N^{i}-N^{b}-N^{b-1}\Bigg)\\
&\qquad\qquad +\sum_{i=b}^{n-1}J_{i+1}N^{i}\big(2\eta_{i}-N+1\big)+h<0\Bigg\}\\
&B_{u} = \Bigg\{ 1\leq b\leq n\colon\,\sum_{i=1}^{b-1}J_{i}N^{i}\Big(1-\frac{1}{N}\Big)
+J_{b}\Bigg(N^{b}-2\sum_{i=0}^{b-1}\eta_{i}N^{i}-N^{b-1}\Bigg)\\
&\qquad\qquad +\sum_{i=b}^{n-1}J_{i+1}N^{i}\big(N-2\eta_{i}-1\big)-h<0\Bigg\}, 
\label{eq:Bd Bu}
\end{aligned}
\end{equation}
where $\left\{\eta_{i}\right\} _{i=0}^{n-1}$ is defined as in the statement of 
Lemma~\ref{lem:locmaxima}. By (\ref{eq:vertexflip-s2}) and (\ref{eq:vertexflip+s2}), 
$B_{d}$ gives the distances to the `critical' vertex of the vertices that are flipped 
in obtaining configurations that result in a lower energy than the critical configuration. 
Thus
\begin{equation}
\begin{aligned}
&N^{-}\left(\sigma\right) = \left|\left\{ \sigma\in U_{\sigma}^{-}\colon\,
\cH\left(\sigma\right)<\cH\left(\sigma\right)\right\} \right|
=\left[1+o_{N}\left(1\right)\right]
\sum_{i\in B_{d}}\eta_{i-1}N^{i-1},\\
&N^{+}\left(\sigma\right) = \left|\left\{ \sigma\in U_{\sigma}^{+}\colon\,
\cH\left(\sigma\right)<\cH\left(\sigma\right)\right\} \right|
=\left[1+o_{N}\left(1\right)\right]
\sum_{i\in B_{u}}\left(N^{i}-\eta_{i-1}N^{i-1}\right).
\end{aligned}
\end{equation}
Hence, by Lemma \ref{lem:variational-lemma}, we have
\begin{equation}
\begin{aligned}
\frac{1}{K^\star} 
&= \left[1+o_{N}\left(1\right)\right] \sum_{\sigma\in\cC^{\star}}
\frac{\Big(\sum_{i\in B_{d}}\eta_{i-1}N^{i-1}\Big)\Big(\sum_{i\in B_{u}}
(N^{i}-\eta_{i-1}N^{i-1})\Big)}{\Big(\sum_{i\in B_{d}}\eta_{i-1}N^{i-1}\Big)
+\Big(\sum_{i\in B_{u}}(N^{i}-\eta_{i-1}N^{i-1})\Big)}\\
& = \left[1+o_{N}\left(1\right)\right] \frac{\Big(\sum_{i\in B_{d}}\eta_{i-1}N^{i-1}\Big)
\Big(\sum_{i\in B_{u}}(N^{i}-\eta_{i-1}N^{i-1})\Big)}
{\Big(\sum_{i\in B_{d}}\eta_{i-1}N^{i-1}\Big)+\Big(\sum_{i\in B_{u}}(N^{i}-\eta_{i-1}N^{i-1})\Big)}\\
&\qquad \times \frac{N^{n-\hat{m}-1}}{N-\eta_{0}}
\prod_{i=0}^{\hat{m}}{N \choose \eta_{i}}\left(N-\eta_{i}\right).
\end{aligned}
\end{equation}


\subsection{Proof of Theorems~\ref{thm:Cstar-case1}--\ref{thm:Cstar-case2}}
\label{S3.5}

Let $\{\tilde{J}_{i}\} _{i=1}^{n}$ be such that $\tilde{J}_{i}/N \to 0$ for all $i\in
\left\{1,\ldots,n\right\}$ as $N\to\infty$, and take $J_{i}=\tilde{J}_{i}/N^{i}$. 
It is easy to check that Assumption (A3) is satisfied given that Assumption 
(A2)(b) is also satisfied. 

\begin{proof}[Proof of Theorem \ref{thm:Cstar-case1}]
From (\ref{eq:sineq}) and (\ref{eq:zetai+1}) we learn that
\begin{eqnarray}
\eta_{\hat{m}} & = & \left\lceil \hat{s}\right\rceil 
=\left\lceil \frac{N}{2\tilde{J}_{\hat{m}+1}}\left(\left(1-\frac{1}{N}\right)
\sum_{i=\hat{m}+1}^{n}\tilde{J}_{i}-h\right)\right\rceil \nonumber \\
& = & \left[1+o_{N}\left(1\right)\right]\frac{1}{2}\frac{N}{\tilde{J}_{\hat{m}+1}}
\left(\sum_{i=\hat{m}+1}^{n}\tilde{J}_{i}-h\right)
\label{eq:special-case-zetam}
\end{eqnarray}
and 
\begin{equation}
\begin{aligned}
&\eta_{\hat{m}-1} =  \left[1+o_{N}\left(1\right)\right]\frac{N}{2},\\
&\eta_{\hat{m}-i} = \left[1+o_{N}\left(1\right)\right]\frac{N}{2}
\left(\sum_{j=1}^{i+1}\left(\frac{\tilde{J}_{\hat{m}-i+j}}{\tilde{J}_{\hat{m}-i}}\right)
\left(1-\frac{2\eta_{\hat{m}-i+1}}{N}\right)+\sum_{j=2}^{n-\hat{m}}
\left(\frac{\tilde{J}_{\hat{m}+j}}{\tilde{J}_{\hat{m}-i}}\right)-\frac{h}
{\tilde{J}_{\hat{m}-i}}+1\right),
\label{eq:special-case-zetam-i}
\end{aligned}
\end{equation}
for $i=1,\ldots,\hat{m}$. This identifies the configurations announced in 
\eqref{eq:thrm - K decomp}.
\end{proof}

\begin{proof}[Proof of Theorem \ref{thm:Cstar-case2}]
Observe from (\ref{eq:sineq}) that 
\begin{eqnarray}
\hat{s} & = & \frac{N}{2\tilde{J}_{\hat{m}+1}}\left(\left(1-\frac{1}{N}\right)
\sum_{i=\hat{m}+1}^{n}\tilde{J}_{i}-h\right),
\end{eqnarray}
and by assumption (A1)(b) we have that 
\begin{equation}
\lim_{N\to\infty}\left(1-\frac{1}{N}\right)
\sum_{i=\hat{m}+1}^{n}\tilde{J}_{i}-h=\lim_{N\to\infty}\sum_{i=\hat{m}+1}^{n}
\tilde{J}_{i}-h>0.
\end{equation} 
Then 
\begin{equation}
\begin{aligned}
&\frac{k\sum_{i=\hat{m}+1}^{n}J_{i}N^{i}}{\left\lceil \hat{s}\right\rceil 
\left(2\hat{s}-\left\lceil \hat{s}\right\rceil +1\right)J_{\hat{m}+1}N^{2\hat{m}}}
=  \frac{k\sum_{i=\hat{m}+1}^{n}\tilde{J}_{i}}{\left\lceil \hat{s}\right\rceil 
\left(2\hat{s}-\left\lceil \hat{s}\right\rceil +1\right)\tilde{J}_{\hat{m}+1}N^{\hat{m}-1}}\\
&\leq \frac{N\sum_{i=\hat{m}+1}^{n}\tilde{J}_{i}}
{\left\lceil \hat{s}\right\rceil \left(2\hat{s}-\left\lceil \hat{s}\right\rceil +1\right)
\tilde{J}_{\hat{m}+1}}
= O\left(N^{-1}\right)\left(\tilde{J}_{\hat{m}+1}\right)^{-1}
\sum_{i=\hat{m}+1}^{n}\tilde{J}_{i},
\end{aligned}
\label{eq:A2 check1}
\end{equation}
and similarly 
\begin{equation}
\begin{aligned}
&\frac{\sum_{i=0}^{\hat{m}-1}J_{i+1}N^{i}\left(\left(N-a_{i}-1\right)
\left(\sum_{j=0}^{i}a_{j}N^{j}\right)+a_{i}\left(N^{i}
-\sum_{j=0}^{i-1}a_{j}N^{j}\right)\right)}{\left\lceil \hat{s}\right\rceil 
\left(2\hat{s}-\left\lceil \hat{s}\right\rceil +1\right)J_{\hat{m}+1}N^{2\hat{m}}}\\ 
&\leq \frac{N\sum_{i=0}^{\hat{m}-1}\tilde{J}_{i+1}}{\left\lceil \hat{s}\right\rceil 
\left(2\hat{s}-\left\lceil \hat{s}\right\rceil +1\right)\tilde{J}_{\hat{m}+1}}
= O\left(N^{-1}\right)\left(\tilde{J}_{\hat{m}+1}\right)^{-1}
\sum_{i=\hat{m}+1}^{n}\tilde{J}_{i}.
\end{aligned}
\label{eq: A2 check2}
\end{equation}
Summing (\ref{eq:A2 check1}) and (\ref{eq: A2 check2}) we get Assumption (A2). 
From (\ref{eq:Gstarhlimit}) we get 
\begin{equation}
\Gamma^{\star}=\left[1+o_{N}\left(1\right)\right]\frac{1}{4}N^{\hat{m}+1}
\left(\tilde{J}_{\hat{m}+1}\right)^{-1}\left(\left(1-\frac{1}{N}\right)
\sum_{i=\hat{m}+1}^{n}\tilde{J}_{i}-h\right)^{2}.
\end{equation}
\end{proof}


\section{Standard interaction}
\label{S4}

In this section we consider the special case 
\begin{equation}
J_{i}=\tilde{J}/N^{i}, \qquad 1 \leq i \leq n,
\label{eq:regulardef}
\end{equation}
for some $\tilde{J}>0$. The Hamiltonian in (\ref{eq:Hamiltonian})
becomes 
\begin{equation}
\cH\left(h;\sigma\right)=-\frac{\tilde{J}}{2}
\sum_{ {a,b \in \Lambda_N^n:} \atop {a\neq b} }N^{-d\left(v_a,v_a\right)}
\sigma(v_a)\sigma(v_a)-\frac{h}{2}\sum_{a \in \Lambda_N^n}\sigma(v_a),
\end{equation}
where  we exhibit the dependence on $h$. In Sections~\ref{S4.1} we show that 
the energy landscape has certain symmetries. In Section~\ref{S4.2} we exploit 
these symmetries to identify the location of the global maximum of the energy 
along the reference path $\gamma$. In Section~\ref{S4.3} we use these results
to prove Theorems~\ref{thm:Gamma-case3} and \ref{thm:Cstar-case3}. In 
Section~\ref{S4.4} we compute the prefactor and prove Theorem~\ref{thm:Kstar-case2}.


\subsection{Symmetries in the energy landscape}
\label{S4.1}

In this section we derive four lemmas
(Lemmas~\ref{lem:concavesymmetricsequence}--\ref{lem:2ndsymmetries} below) 
exhibiting certain symmetries in the energy landscape for the case of standard 
interaction (see Fig.~\ref{figplot}). These symmetries will be crucial later on.

For any $h_{1},h_{2}>0$ and $0\leq a,b\leq N^{n}$, 
\begin{equation}
\cH\left(h_{1};\gamma_{a}\right)-\cH\left(h_{1};\gamma_{b}\right)
=\cH\left(h_{2};\gamma_{a}\right)-\cH\left(h_{2};\gamma_{b}\right)
+\left(h_{2}-h_{1}\right)\left(a-b\right).
\label{eq:differenthdifference}
\end{equation}

\begin{definition}
\label{def-symmetric&concave}
A sequence $\left\{a_{i}\right\} _{i=1}^{M} \in \R^M$ is called symmetric when
\begin{equation}
a_{i}=a_{M-i+1}, \qquad 1\leq i\leq M,
\end{equation}
and concave when 
\begin{equation}
a_{i}-a_{i-1}\geq a_{i+1}-a_{i}, \qquad 2\leq i\leq M-1.
\label{eq:concavityproperty}
\end{equation}
\end{definition}

\noindent
The following lemma is elementary. 

\begin{lemma}
\label{lem:concavesymmetricsequence} 
Suppose that the sequence $\{a_{i}\} _{i=1}^{M}$ is symmetric and concave. Then 
\begin{equation}
\max_{1 \leq i \leq M} a_{i} = a_{\left\lceil \frac{M}{2}\right\rceil}.
\end{equation}
\end{lemma}

Recall the definition of $\hat{m}$ from (\ref{eq:defm}), and note that now 
\begin{equation}
\hat{m}_{h}=\left\lfloor n-\frac{h}{\tilde{J}}\left(1-\frac{1}{N}\right)^{-1}\right\rfloor, 
\label{eq:mhatregular}
\end{equation}
where again we exhibit the dependence on $h$. It was shown in Section~\ref{S3.2} 
that, in the hierarchical limit $N\to\infty$, $\hat{m}_{h}$ gives the order of magnitude 
of a critical configuration (in particular, the asymptotic size of a critical configuration 
was shown to be $\hat{s}N^{\hat{m}}$). We will now show that for the standard 
interaction in (\ref{eq:regulardef}), $\hat{m}_{h}$ plays a similar role. 

Let $\gamma\colon\,\boxminus\to\boxplus$ be the optimal path defined in Section~\ref{S1.4}.
We begin by considering the Hamiltonian $i \mapsto \cH(h;\gamma_{i})$ for certain 
special values of $h$. Recall $h^{\left(m,s\right)}$ defined in (\ref{eq:hdagger}). In 
terms of this quantity, we have
\begin{equation}
\begin{aligned}
&\cH\left(h^{\left(m,s\right)};\gamma_{sN^{\hat{m}}}\right)
-\cH\left(h^{\left(m,s\right)};\boxminus\right)\\ 
&= \frac{\tilde{J}}{N}sN^{m}\left(N-s\right)+\tilde{J}sN^{m}\sum_{i=m+2}^{n}\left(1-\frac{1}{N}\right)
-sh^{\left(m,s\right)}N^{m}\\
&= sN^{m}\left(\tilde{J}\left(N-s\right)\frac{1}{N}+\tilde{J}\left(1-\frac{1}{N}\right)\left(n-m-1\right)
-h^{\left(m,s\right)}\right)=0
\end{aligned}
\end{equation}
and 
\begin{equation}
\hat{m}_{h^{\left(m,s\right)}}=\left\lfloor m+\left(s-1\right)\frac{1}{N}
\left(1-\frac{1}{N}\right)^{-1}\right\rfloor =m.
\label{eq:mhat-hms}
\end{equation}
A magnetic field that takes the form $h^{\left(m,s\right)}$ gives rise to symmetries 
in the energy landscape along the path $\gamma$, which we can exploit in order 
to find the values at which $i \mapsto \cH\left(h^{\left(m,s\right)};\gamma_{i}\right)$ 
attains its global maximum. Later we will use this information to find the location of 
the global maxima for general values of $h$. First we show that the global maximum 
of $i \mapsto \cH\left(h^{\left(m,s\right)};\gamma_{i}\right)$ is attained in the interval 
$\left[0,sN^{m}\right]$.

\begin{lemma}
\label{lem:maxisinsN}
For any $1\leq s\leq N$ and $0\leq m\leq n-1$,
\begin{equation}
\max_{1 \leq i \leq N^n} \cH\left(h^{\left(m,s\right)};\gamma_{i}\right)
=\max_{i\leq sN^{m}}\cH\left(h^{\left(m,s\right)};\gamma_{i}\right).
\label{eq:Hineq}
\end{equation}
\end{lemma}

\begin{proof}
Let $K=a_{n-1}N^{n-1}+\ldots+a_{0}$ and $u(i)=a_{n-1}N^{n-1}+\ldots+a_{i}N^{i}$, and note that, by Lemma \ref{lem:shift-increment},
\begin{equation}
\cH\left(h^{\left(m,s\right)};\gamma_{u(m+1)}\right)
\leq\cH\left(h^{\left(m,s\right)};\gamma_{u(m+2)}\right)+\cH\left(h^{\left(m,s\right)};\gamma_{a_{m+1}N^{m+1}}\right)
-\cH\left(h^{\left(m,s\right)};\boxminus\right).
\end{equation}
By Lemma \ref{lem:concave} and the definition of $m=\hat{m}$ in (\ref{eq:defm2}), we 
have, for $0 \leq m<n-1$, 
\begin{equation}
\begin{aligned}
&\cH\left(h^{\left(m,s\right)};\gamma_{a_{m+1}N^{m+1}}\right)
-\cH\left(h^{\left(m,s\right)};\gamma_{\left(a_{m+1}-1\right)N^{m+1}}\right)\\ 
&\qquad \leq \cH\left(h^{\left(m,s\right)};\gamma_{N^{m+1}}\right)
-\cH\left(h^{\left(m,s\right)};\boxminus\right)\leq0.
\end{aligned}
\end{equation}
Hence, by induction, 
\begin{equation}
\begin{aligned}
&\cH\left(h^{\left(m,s\right)};\gamma_{a_{m+1}N^{m+1}}\right)
\leq\cH\left(h^{\left(m,s\right)};\gamma_{\left(a_{m+1}-1\right)N^{m+1}}\right)\\
&\qquad \leq\ldots\leq\cH\left(h^{\left(m,s\right)};\gamma_{N^{m+1}}\right)
\leq\cH\left(h^{\left(m,s\right)};\boxminus\right)
\label{eq:inductsetp}
\end{aligned}
\end{equation}
and therefore
\begin{equation}
\cH\left(h^{\left(m,s\right)};\gamma_{u(m+1)}\right)
\leq\cH\left(h^{\left(m,s\right)};\gamma_{u(m+2)}\right).
\end{equation}
Once again it follows from inductive reasoning that 
\begin{equation}
\cH\left(h^{\left(m,s\right)};\gamma_{u(m+1)}\right)
\leq\cH\left(h^{\left(m,s\right)};\gamma_{a_{n-1}N^{n-1}}\right).
\end{equation}
By the same reasoning as in (\ref{eq:inductsetp}), we have
\begin{equation}
\cH\left(h^{\left(m,s\right)};\gamma_{a_{n-1}N^{n-1}}\right)
\leq\cH\left(h^{\left(m,s\right)};\gamma_{N^{n-1}}\right)
\leq\cH\left(h^{\left(m,s\right)};\boxminus\right)
\end{equation}
and hence 
\begin{equation}
\cH\left(h^{\left(m,s\right)};\gamma_{u(m+1)}\right)
\leq\cH\left(h^{\left(m,s\right)};\boxminus\right).
\end{equation}
Thus 
\begin{equation}
\begin{aligned}
&\cH\left(h^{\left(m,s\right)};\gamma_{K}\right)
-\cH\left(h^{\left(m,s\right)};\boxminus\right)\\
&= \cH\left(h^{\left(m,s\right)};\gamma_{K}\right)
-\cH\left(h^{\left(m,s\right)};\gamma_{u(m+1)}\right)
+ \cH\left(h^{\left(m,s\right)};\gamma_{u(m+1)}\right)
-\cH\left(h^{\left(m,s\right)};\boxminus\right)\\
& \leq \cH\left(h^{\left(m,s\right)};\gamma_{K}\right)
-\cH\left(h^{\left(m,s\right)};
\gamma_{u(m+1)}\right)
\leq \cH\left(h^{\left(m,s\right)};\gamma_{a_{m}N^{m}+\ldots+a_{0}}\right)
-\cH\left(h^{\left(m,s\right)};\boxminus\right),
\end{aligned}
\label{eq:slowing2-inequality}
\end{equation}
where the last inequality again follows from Lemma \ref{lem:shift-increment}. Moreover, for 
$m=n-1$ the inequality in (\ref{eq:slowing2-inequality}) is immediate. If $a_{m}<s$,  then the 
claim in \eqref{eq:Hineq} follows immediately. Otherwise we have $\cH\left(h^{\left(m,s\right)};
\gamma_{a_{m}N^{m}}\right)\leq\cH\left(h^{\left(m,s\right)};\boxminus\right)$ and hence, by 
Lemma~\ref{lem:shift-increment} and using the abbreviation $v(i)=a_{i}N^{i}+\ldots+a_{0}$, 
\begin{equation}
\begin{aligned}
&\cH\left(h^{\left(m,s\right)};\gamma_{v(m)}\right)
-\cH\left(h^{\left(m,s\right)};\boxminus\right)\\ 
&\leq \cH\left(h^{\left(m,s\right)};\gamma_{v(m)}\right)
-\cH\left(h^{\left(m,s\right)};\gamma_{a_{m}N^{m}}\right)
+ \cH\left(h^{\left(m,s\right)};\gamma_{a_{m}N^{m}}\right)
-\cH\left(h^{\left(m,s\right)};\boxminus\right) \\
&\leq \cH\left(h^{\left(m,s\right)};\gamma_{v(m-1)}\right)
-\cH\left(h^{\left(m,s\right)};\boxminus\right)\\
&\leq \max_{1 \leq i\leq sN^{m}}\cH\left(h^{\left(m,s\right)};\gamma_{i}\right)
-\cH\left(h^{\left(m,s\right)};\boxminus\right),
\end{aligned}
\end{equation}
which settles the claim.
\end{proof}

We next derive two results stating $\{\cH(h^{\left(m,s\right)},\gamma_{i})\} _{i=1}^{sN^n}$ 
(illustrated in Fig.~\ref{figplot}) is symmetric and fractal-like, which is used later to locate the 
global maxima of this sequence. 

\begin{lemma}
\label{lem: symmetry}
The sequence $\{\cH(h^{(m,s)};\gamma_{i})\} _{i=0}^{sN^{m}}$ is symmetric, i.e., 
\begin{equation}
\cH\left(h^{\left(m,s\right)};\gamma_{K}\right)=\cH\left(h^{\left(m,s\right)};
\gamma_{sN^{m}-K}\right), \qquad 0\leq K\leq sN^{m}. 
\end{equation}
\end{lemma}

\begin{proof}
Let $K=k_{n-1}N^{n-1}+\ldots+k_{0}$ , so that 
\begin{equation}
\begin{aligned}
&\cH\left(h^{\left(m,s\right)};\gamma_{K}\right)-\cH\left(h^{\left(m,s\right)};\boxminus\right)
+ h^{\left(m,s\right)}K\\ 
& = \sum_{i=0}^{n-1}J_{i+1}N^{i}\left(\left(\sum_{j=0}^{i}k_{j}N^{j}\right)
\left(N-k_{i}-1\right)+k_{i}\left(N^{i}-\sum_{j=0}^{i-1}k_{j}N^{j}\right)\right)\\
& = \sum_{i=0}^{n-1}J_{i+1}N^{i}\left(\left(\sum_{j=0}^{i-1}k_{j}N^{j}\right)
\left(N-k_{i}-1\right)+k_{i}N^{i}\left(N-k_{i}-1\right)+k_{i}N^{i}-k_{i}\sum_{j=0}^{i-1}k_{j}N^{j}\right)\\
& = \sum_{i=0}^{n-1}J_{i+1}N^{i}\left(\left(\sum_{j=0}^{i-1}k_{j}N^{j}\right)
\left(N-2k_{i}-1\right)+k_{i}N^{i}\left(N-k_{i}\right)\right)\\
& = \sum_{i=0}^{n-1}\frac{\tilde{J}}{N}\left(\left(\sum_{j=0}^{i-1}k_{j}N^{j}\right)
\left(N-2k_{i}-1\right)+k_{i}N^{i}\left(N-k_{i}\right)\right).
\end{aligned}
\end{equation}
Since $k_{i}=0$ for $i>m$ and $k_{m}<s$, this simplifies to 
\begin{equation}
\begin{aligned}
&\cH\left(h^{\left(m,s\right)};\gamma_{K}\right)-\cH\left(h^{\left(m,s\right)};\boxminus\right)\\ 
& = \sum_{i=0}^{m}\frac{\tilde{J}}{N}\left(\left(\sum_{j=0}^{i-1}k_{j}N^{j}\right)
\left(N-2k_{i}-1\right)+k_{i}N^{i}\left(N-k_{i}\right)\right)\\
&\qquad  + K\left(\tilde{J}\left(1-\frac{1}{N}\right)\left(n-m-1\right)-h^{\left(m,s\right)}\right).
\label{eq:HKregsymetric}
\end{aligned}
\end{equation}
Note that if $\tilde{K}=sN^{m}-K$, then the number of interacting pairs at distance 
$i=0,\ldots,m$ in the configuration $\gamma_{\tilde{K}}$ (i.e., vertices $v_{a}$, $v_{b}$
such that $\gamma_{\tilde{K}}\left(v_{a}\right)=-\gamma_{\tilde{K}}\left(v_{b}\right)$ 
and $d\left(v_{a},v_{b}\right)=i$) is the same as in the configuration $\gamma_{K}$. 
At distance $m+1$ this number is equal to 
\begin{equation}
N^{m}\left(K\left(s-k_{m}-1\right)+\left(N^{m}-\sum_{j=0}^{m-1}k_{j}N^{j}\right)k_{m}
+\left(sN^{m}-K\right)\left(N-s\right)\right)
\end{equation}
and therefore we conclude that 
\begin{equation}
\begin{aligned}
&\cH\left(h^{\left(m,s\right)};\gamma_{\tilde{K}}\right)-\cH\left(h^{\left(m,s\right)};\boxminus\right)\\ 
& = \sum_{i=0}^{m-1}\frac{\tilde{J}}{N}\left(\left(\sum_{j=0}^{i-1}k_{j}N^{j}\right)
\left(N-2k_{i}-1\right)+k_{i}N^{i}\left(N-k_{i}\right)\right)\\
&\qquad + \frac{\tilde{J}}{N}\left(K\left(s-k_{m}-1\right)+\left(N^{m}
-\sum_{j=0}^{m-1}k_{j}N^{j}\right)k_{m}+\left(sN^{m}-K\right)\left(N-s\right)\right)\\
&\qquad + \sum_{i=m+1}^{n-1}\tilde{J}\left(1-\frac{1}{N}\right)
\left(sN^{m}-\sum_{j=0}^{m}k_{j}N^{j}\right)-h^{\left(m,s\right)}\tilde{K}.
\end{aligned}
\end{equation}
Thus, we have
\begin{equation}
\begin{aligned}
&\cH\left(h^{\left(m,s\right)};\gamma_{\tilde{K}}\right)-\cH\left(h^{\left(m,s\right)};\gamma_{K}\right)
= \sum_{i=m+1}^{n-1}\tilde{J}\left(1-\frac{1}{N}\right)\left(sN^{m}-2K\right)\\
&\quad + \frac{\tilde{J}}{N}\left(K\left(s-k_{m}-1\right)+\left(sN^{m}-K\right)\left(N-s\right)
-K\left(N-k_{i}-1\right)\right)
- h^{\left(m,s\right)}\left(sN^{m}-2K\right),
\end{aligned}
\end{equation}
which is equal to $0$ if and only if
\begin{equation}
\begin{aligned}
&h^{\left(m,s\right)}\left(sN^{m}-2K\right)\\ 
& = \tilde{J}\left(1-\frac{1}{N}\right)\left(sN^{m}-2K\right)\left(n-m-1\right)\\
&\quad + \frac{\tilde{J}}{N}\left(K\left(s-k_{m}-1\right)+\left(sN^{m}-K\right)\left(N-s\right)
-K\left(N-k_{m}-1\right)\right)\\
& = \tilde{J}\left(1-\frac{1}{N}\right)\left(sN^{m}-2K\right)\left(n-m-1\right)
+\frac{\tilde{J}}{N}\left(K\left(s-N\right)+\left(sN^{m}-K\right)\left(N-s\right)\right)\\
& = \tilde{J}\left(sN^{m}-2K\right)\left(\left(1-\frac{1}{N}\right)\left(n-m\right)
-\left(s-1\right)\frac{1}{N}\right),
\end{aligned}
\end{equation}
which indeed is true by the definition of $h^{\left(m,s\right)}$ in (\ref{eq:hdagger}). 
\end{proof}

To state the second result we need some more notation. Let $Q\colon\,\mathbb{N}_{0}
\to\left\{ 0,1\right\}$ be defined by 
\begin{equation}
Q\left(a\right)=a\,\mathrm{mod}\,2.
\label{eq:def Q}
\end{equation}
For all integers $k\in\left\{1,\ldots,m\right\} $ taking the form $k=a(1+Q(N+1))-Q((N+1)(s+1))$ 
for some $a\in\left\{1,\ldots,m\right\}$, define the integer intervals $S_{k}=[S_{k}^{-},S_{k}^{+}]$, 
where
\begin{equation}
\begin{aligned}
S_{k}^{-} & =  \left(\left\lfloor \frac{s}{2}\right\rfloor -1+Q\left(s\left(N+1\right)\right)\right)N^{m}
+\sum_{j=1}^{k-1}a_{m-j}N^{m-j}+\left(1+Q\left(sN\right)\right)N^{m-k},\\
S_{k}^{+} & = \left(\left\lfloor \frac{s}{2}\right\rfloor -1+Q\left(s\left(N+1\right)\right)\right)N^{m}
+\sum_{j=1}^{k-1}a_{m-j}N^{m-j}+N^{m-k+1}, 
\end{aligned}
\label{eq:SetSk}
\end{equation}
and 
\begin{equation}
a_{m-j}=\left\lfloor \frac{N}{2}-Q\left(\left(j+s+1\right)\left(N+1\right)\right)\right\rfloor.
\end{equation}
The following clarification regarding (\ref{eq:SetSk}) is in order. For odd values of $N$, 
(\ref{eq:SetSk}) defines the sets $S_{1},\ldots,S_{m}$, and the coefficients $a_{m-j}$ 
are all equal to $\left\lfloor \frac{N}{2}\right\rfloor =\frac{N-1}{2}$. For even values of 
$N$ and even values of $s$, (\ref{eq:SetSk}) defines the odd-indexed sets $S_{1},S_{3},
\ldots,S_{2\left\lfloor \frac{m}{2}\right\rfloor +1}$ and the coefficients $a_{m-j}$ are given 
by $a_{m-1}=\frac{N}{2}$, $a_{m-2}=\frac{N}{2}-1$, etc. For even values of $N$ and 
odd values of $s$, (\ref{eq:SetSk}) defines the even-indexed sets $S_{2},S_{4},\ldots,
S_{2\left\lfloor \frac{m}{2}\right\rfloor }$ and the coefficients $a_{m-j}$ are given by 
$a_{m-1}=\frac{N}{2}-1$, $a_{m-2}=\frac{N}{2}$, etc.

\begin{lemma}
\label{lem:2ndsymmetries} 
For every $k\in\left\{1,\ldots,m\right\}$ that takes the form 
\begin{equation}
k=a\left(1+Q\left(N+1\right)\right)+Q\left(\left(N+1\right)\left(s+1\right)\right)
\end{equation}
for some $a\in\mathbb{N}_{0}$, the sequence $\{\cH(h^{(m,s)};\gamma_{i})\} _{i\in S_{k}}$ 
is symmetric.
\end{lemma}

\begin{proof}
Suppose that $K\in S_{k}$, so that 
\begin{equation}
K = \sum_{i=0}^{m}a_{i}N^{i}
= \left(\left\lfloor \frac{s}{2}\right\rfloor -1+Q\left(s\left(N+1\right)\right)\right)N^{m}
+\sum_{j=1}^{k-1}a_{m-j}N^{m-j}+R,
\end{equation}
where 
\begin{equation}
R=a_{m-k}N^{m-k}+a_{m-k-1}N^{m-k-1}+\ldots+a_{0}
\end{equation}
for $1+Q\left(sN\right)\leq a_{m-k}\leq N-1$ and $0\leq a_{i}\leq N-1$
for $0\leq i<m-k$. Also let 
\begin{equation}
\begin{aligned}
\tilde{K} &= \left(\left\lfloor \frac{s}{2}\right\rfloor -1+Q\left(s\left(N+1\right)\right)\right)N^{m}\\
&\qquad +\sum_{j=1}^{k-1}a_{m-j}N^{m-j}+N^{m-k+1}-R+\left(1+Q\left(sN\right)\right)N^{m-k}\\
& = K+N^{m-k+1}-2R+\left(1+Q\left(sN\right)\right)N^{m-k},
\end{aligned}
\end{equation}
so that $K$ and $\tilde{K}$ are mirrored points in $S_{k}$ (i.e., if $K$ is the $i^{th}$ point in 
$S_{k}$, then $\tilde{K}$ is the $(\left|S_{k}\right|-i)^{th}$ point). Note that, by 
(\ref{eq:HKregsymetric}),
\begin{equation}
\begin{aligned}
&\cH\left(h^{\left(m,s\right)};\gamma_{K}\right)-\cH\left(h^{\left(m,s\right)};\boxminus\right) 
= \sum_{i=0}^{m}\frac{\tilde{J}}{N}\left(\left(\sum_{j=0}^{i-1}a_{j}N^{j}\right)
\left(N-2a_{i}-1\right)+a_{i}N^{i}\left(N-a_{i}\right)\right)\\
&\qquad\qquad\qquad\qquad  
+ K\left(\tilde{J}\left(1-\frac{1}{N}\right)\left(n-m-1\right)-h^{\left(m,s\right)}\right).
\end{aligned}
\end{equation}
Observe that, for $1\leq i\leq m-k$, the total number of interacting pairs at distance 
$i$ in $\gamma_{\tilde{K}}$ (i.e., vertices $v,w$ such that $d\left(v,w\right)=i$ and 
$\gamma_{\tilde{K}}\left(v\right)=-\gamma_{\tilde{K}}\left(w\right)$), is the same as 
in $\gamma_{K}$. At distance $m-k+1$, the number of interacting pairs in 
$\gamma_{\tilde{K}}$ is equal to the number of interacting pairs in $\gamma_{K}$ 
plus $(1+Q(sN))N^{m-k}(R-(1+Q(sN))N^{m-k})$ minus $(1+Q(sN))N^{m-k}(N^{m-k+1}
-R)$. For $m-k+2\leq i$, the number of interacting pairs at distance $i$ in 
$\gamma_{\tilde{K}}$ is equal to the number of interacting pairs in $\gamma_{K}$ 
plus $a_{i}N^{i}(R-(1+Q(sN))N^{m-k})$ minus $a_{i}N^{i}(N^{m-k+1}-R)$, and 
plus $(N-a_{i}-1)N^{i}(N^{m-k+1}-R)$ minus $(N-a_{i}-1)N^{i}(R-(1+Q(sN))N^{m-k})$. 
Thus, we have
\begin{equation}
\begin{aligned}
&\left(\cH\left(h^{\left(m,s\right)};\gamma_{\tilde{K}}\right)
-\cH\left(h^{\left(m,s\right)};\boxminus\right)\right)\left(\frac{\tilde{J}}{N}\right)^{-1}\\
& = \sum_{i=0}^{m-k}\left(\left(\sum_{j=0}^{i-1}a_{j}N^{j}\right)
\left(N-2a_{i}-1\right)+a_{i}N^{i}\left(N-a_{i}\right)\right)\\
&\qquad + \left(1+Q\left(sN\right)\right)\left(2R-N^{m-k+1}
-\left(1+Q\left(sN\right)\right)N^{m-k}\right)\\
&\qquad + \sum_{i=m-k+1}^{m}\left(\left(\sum_{j=0}^{i-1}a_{j}N^{j}\right)
\left(N-2a_{i}-1\right)+a_{i}N^{i}\left(N-a_{i}\right)\right)\\
&\qquad + \sum_{i=m-k+1}^{m}\left(N-2a_{i}-1\right)\left(N^{m-k+1}-2R
+\left(1+Q\left(sN\right)\right)N^{m-k}\right)\\
&\qquad + \sum_{i=m+1}^{n-1}\tilde{J}\left(1-\frac{1}{N}\right)
\left(\sum_{j=0}^{m}a_{j}N^{j}+\left(N^{m-k+1}-2R
+\left(1+Q\left(sN\right)\right)N^{m-k}\right)\right)\\
&\qquad - h^{\left(m,s\right)}\tilde{K}.
\end{aligned}
\end{equation}
Hence it follows that 
\begin{equation}
\begin{aligned}
&\cH\left(h^{\left(m,s\right)};\gamma_{\tilde{K}}\right)
-\cH\left(h^{\left(m,s\right)};\gamma_{K}\right) 
= \frac{\tilde{J}}{N}\left(1+Q\left(sN\right)\right)\left(2R-N^{m-k+1}
-\left(1+Q\left(sN\right)\right)N^{m-k}\right)\\
&\qquad + \sum_{i=m-k+1}^{m}\frac{\tilde{J}}{N}\left(N-2a_{i}-1\right)\left(N^{m-k+1}-2R
+\left(1+Q\left(sN\right)\right)N^{m-k}\right)\\
&\qquad + \sum_{i=m+1}^{n-1}\tilde{J}\left(1-\frac{1}{N}\right)\left(N^{m-k+1}-2R
+\left(1+Q\left(sN\right)\right)N^{m-k}\right)-h^{\left(m,s\right)}\left(\tilde{K}-K\right).
\label{eq:HKtildeKdifference}
\end{aligned}
\end{equation}
Note that (\ref{eq:HKtildeKdifference}) is equal to zero if and only if
\begin{equation}
\begin{aligned}
&h^{\left(m,s\right)}\left(\tilde{K}-K\right) 
= h^{\left(m,s\right)}\left(N^{m-k+1}-2R+N^{m-k}\right)\\
& =\frac{\tilde{J}}{N}
\left(1+Q\left(sN\right)\right)\left(2R-N^{m-k+1}-\left(1+Q\left(sN\right)\right)N^{m-k}\right)\\
&\qquad + \sum_{i=m-k+1}^{m}\frac{\tilde{J}}{N}\left(N-2a_{i}-1\right)\left(N^{m-k+1}-2R
+\left(1+Q\left(sN\right)\right)N^{m-k}\right)\\
&\qquad + \sum_{i=m+1}^{n-1}\tilde{J}\left(1-\frac{1}{N}\right)\left(N^{m-k+1}-2R
+\left(1+Q\left(sN\right)\right)N^{m-k}\right),
\end{aligned}
\end{equation}
which holds whenever
\begin{equation}
\begin{aligned}
&h^{\left(m,s\right)} 
= -\frac{\tilde{J}}{N}\left(1+Q\left(sN\right)\right)+\sum_{i=m-k+1}^{m}\frac{\tilde{J}}{N}
\left(N-2a_{i}-1\right)+\sum_{i=m+1}^{n-1}\tilde{J}\left(1-\frac{1}{N}\right)\\
& = -\frac{\tilde{J}}{N}\left(1+Q\left(sN\right)\right)+\frac{\tilde{J}}{N}\sum_{i=m-k+1}^{m-1}
\left(N-2a_{i}-1\right)\\
&\qquad +\frac{\tilde{J}}{N}\left(N-2\left\lfloor \frac{s}{2}\right\rfloor 
+1-Q\left(s\left(N+1\right)\right)\right)
+ \left(n-m-1\right)\tilde{J}\left(1-\frac{1}{N}\right).
\end{aligned}
\end{equation}
If $N$ is odd, then $\left(N-2a_{i}-1\right)=(N-2\lfloor \frac{N}{2}\rfloor -1)=0$, and hence 
$\sum_{i=m-k+1}^{m-1}$ $\left(N-2a_{i}-1\right)=0$. If $N$ is even, then the terms 
$\left(N-2a_{i}-1\right)$ alternate between $-1$ and $1$. Thus, if $s$ is even, $k$ is odd 
and $\sum_{i=m-k+1}^{m-1}\left(N-2a_{i}-1\right)=0$ because the sum has an even number 
of terms, while if $s$ is odd, then the sum adds up to $\frac{\tilde{J}}{N}$. We can encode 
this as 
\begin{equation}
\frac{\tilde{J}}{N}\sum_{i=m-k+1}^{m-1}\left(N-2a_{i}-1\right)
=\frac{\tilde{J}}{N}Q\left(s\left(N+1\right)\right).
\end{equation}
Recalling \eqref{eq:hdagger}, it remains to show that 
\begin{equation}
\begin{aligned}
&\left(1-\frac{1}{N}\right)\left(n-m\right)-\left(s-1\right)\frac{1}{N} 
= -\frac{1}{N}\left(1+Q\left(sN\right)\right)+\frac{1}{N}Q\left(s\left(N+1\right)\right)\\
&\qquad + \frac{1}{N}\left(N-2\left\lfloor \frac{s}{2}\right\rfloor 
+1-2Q\left(s\left(N+1\right)\right)\right)+\left(n-m-1\right)\left(1-\frac{1}{N}\right),
\end{aligned}
\end{equation}
or equivalently
\begin{equation}
-\left(s-1\right)\frac{1}{N}=-\frac{1}{N}\left(2\left\lfloor \frac{s}{2}\right\rfloor 
+Q\left(sN\right)-1+Q\left(s\left(N+1\right)\right)\right)=-\left(s-1\right)\frac{1}{N},
\end{equation}
which is trivially true.
\end{proof}

The symmetries in Lemmas~\ref{lem:concavesymmetricsequence}--\ref{lem:2ndsymmetries}
are depicted in Fig.~\ref{figplot}.

\begin{figure}[htbp]
\begin{tikzpicture}[scale=0.7]
\begin{axis}[name=plot1,height=8cm,width=8cm,     	
title = {Energy profile along an optimal path},
xlabel = {$i$}, ylabel = {$\cH\left(\gamma_{i}\right)$}, ymin = -70000, ymax = 140000,
xmin = 0, minor y tick num = 1,]     	
\addplot[very thin, black] table {dataH1.txt};
\addplot[very thin, dashed, gray] table {dataH2.txt};
\draw[red,dashed,thin] (axis cs:19500,110000) rectangle (axis cs:27000,137000);
\end{axis}
\draw[red, dashed,->] (7,4.5)--(9,4.5);
\begin{axis}[name=plot2,at={($(plot1.east)+(3cm,0)$)},anchor=west,height=8cm,width=8cm,     	
title = {$\cH\left(\gamma_{i}\right)$ for $i\in S_1$},
xlabel = {$i$}, ylabel = {$\cH\left(\gamma_{i}\right)$}, ymin = 112000, ymax = 140000,    		
minor y tick num = 1,]     	
\addplot[very thin, black] table {dataS1.txt};
\addplot[very thin, dashed, gray] table {dataS1p.txt};
\addplot[very thin, dashed, gray] table {dataS1q.txt};
\end{axis}
\end{tikzpicture}
\label{figplot}
\caption{\small Plot of $i \mapsto \cH(\gamma_{i})$ for $\Lambda_{5}^{9}$, 
with $\tilde{J}=10.3$ and $h=h^{\left(m,s\right)}=\tilde{J}((1-\frac{1}{N})(n-m)
-\frac{(s-1)}{N})$ with $m=4$ and $s=8$. The solid-line in the left plot corresponds 
to values $i=0,1,\ldots,sN^{m}$, and is symmetric as shown in Lemma~\ref{lem: symmetry}. 
The solid-line in the right plot shows symmetry of $\cH(\gamma_{i})$ for values 
$i\in S_{1}$, as shown in Lemma~\ref{lem:2ndsymmetries}.}
\end{figure}
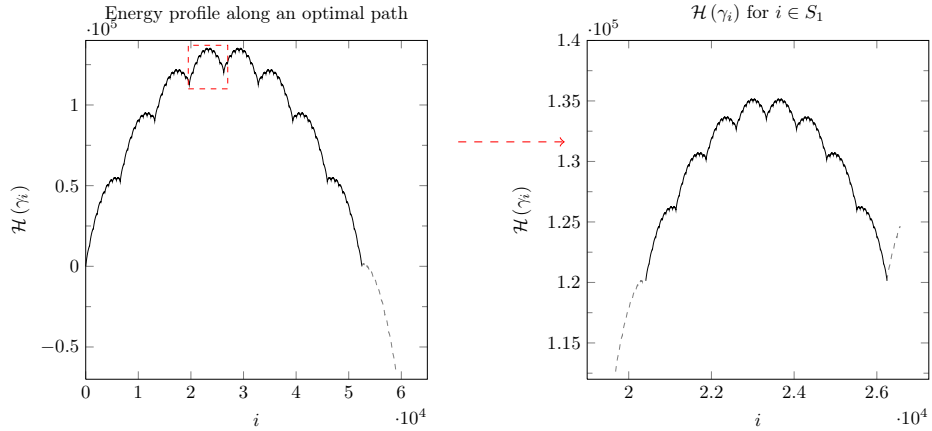


\subsection{Global maximum along the reference path}
\label{S4.2}

In this section we derive two propositions 
(Propositions~\ref{prop:hs-max-location}--\ref{prop:hs-max-location-even} below) 
identifying the location of the global maximum of $i\mapsto\cH(h^{(m,s)};\gamma_{i})$.

\begin{proposition}
\label{prop:hs-max-location}
Suppose that $N$ is odd. If $s$ is odd, then 
\begin{equation}
\begin{aligned}
&\cH\left(h^{\left(m,s\right)};\gamma_{\left\lfloor sN^{m}/2\right\rfloor }\right)
=\cH\left(h^{\left(m,s\right)};\gamma_{\left\lfloor sN^{m}/2\right\rfloor +1}\right)\\
&=\max_{1 \leq i\leq sN^{m}}\cH\left(h^{\left(m,s\right)};\gamma_{i}\right)
=\max_{1 \leq i \leq N^n}\cH\left(h^{\left(m,s\right)};\gamma_{i}\right),
\end{aligned}
\label{eq:maxloc}
\end{equation}
and for all $i<\left\lfloor sN^{m}/2\right\rfloor$, 
\begin{equation}
\cH\left(h^{\left(m,s\right)};\gamma_{i}\right)<\cH\left(h^{\left(m,s\right)};
\gamma_{\left\lfloor sN^{m}/2\right\rfloor }\right).\label{eq:1st max s odd}
\end{equation}
If $s$ is even, then 
\begin{equation}
\cH\left(h^{\left(m,s\right)};\gamma_{\left\lfloor \left(s-1\right)N^{m}/2\right\rfloor +1}\right)
=\max_{1 \leq i\leq sN^{m}}\cH\left(h^{\left(m,s\right)};\gamma_{i}\right)
=\max_{1 \leq i \leq N^n}\cH\left(h^{\left(m,s\right)};\gamma_{i}\right)
\label{eq:maxloc2}
\end{equation}
and for all $i<\left\lfloor \left(s-1\right)N^{m}/2\right\rfloor +1$,
\begin{equation}
\cH\left(h^{\left(m,s\right)};\gamma_{i}\right)<\cH\left(h^{\left(m,s\right)};
\gamma_{\left\lfloor \left(s-1\right)N^{m}/2\right\rfloor +1}\right).
\label{eq: 1st max s even}
\end{equation}
\end{proposition}

\begin{proof}
The first equality in (\ref{eq:maxloc}) is immediate from Lemma~\ref{lem: symmetry} 
since $\left\lfloor sN^{m}/2\right\rfloor +1=sN^{m}-\left\lfloor sN^{m}/2\right\rfloor$, 
while the third equality follows from Lemma \ref{lem:maxisinsN}. We claim that the 
second equality in (\ref{eq:maxloc}) follows from both Lemma \ref{lem: symmetry}
and Lemma \ref{lem:concavesymmetricsequence}. Indeed, note that by
Lemma~\ref{lem:concave} the sequence
\begin{equation}
\left\lbrace \cH\left(h;\gamma_{i}\right)\right\rbrace, \quad \left\lfloor \frac{sN^{m}}{2}\right\rfloor -\left\lfloor \frac{N}{2}\right\rfloor +1 \leq i \leq \left\lfloor \frac{sN^{m}}{2}\right\rfloor +\left\lfloor \frac{N}{2}\right\rfloor +2.
\end{equation}
is concave, and by Lemma~\ref{lem: symmetry} is also symmetric. Therefore, by 
Lemma~\ref{lem:concavesymmetricsequence}, we have that $\cH(\gamma_{i})
\leq\cH(\gamma_{\lfloor sN^{m}/2\rfloor})$ for all $i$ such that $d(v_{i},v_{\lfloor sN^{m}/2
\rfloor})=d(v_{i},v_{\lfloor sN^{m}/2\rfloor +1})=1$. In fact, from Lemma~\ref{lem:concave} 
we have a strict form of concavity,
\begin{equation}
\begin{aligned}
&\cH\left(h^{(m,s)};\gamma_{\left\lfloor sN^{m}/2\right\rfloor }\right)
-\cH\left(h^{(m,s)};\gamma_{\left\lfloor sN^{m}/2\right\rfloor }\right)\\
&\qquad =\cH\left(h^{(m,s)};\gamma_{\left\lfloor sN^{m}/2\right\rfloor +1}\right)
-\cH\left(h^{(m,s)};\gamma_{\left\lfloor sN^{m}/2\right\rfloor }\right)+2\tilde{J}
=2\tilde{J},
\end{aligned}
\end{equation}
which shows that 
\begin{equation}
\cH(\gamma_{\lfloor sN^{m}/2\rfloor})>\cH(\gamma_{i}) 
\qquad \forall\, i<\lfloor sN^{m}/2\rfloor\colon\,d(v_{i},v_{\lfloor sN^{m}/2\rfloor})
=d(v_{i},v_{\lfloor sN^{m}/2\rfloor +1})=1.
\end{equation}
Suppose that this is also true for all $i$ such that $d(v_{i},v_{\lfloor sN^{m}/2\rfloor})=r$, 
and let $z$ be such that $d(v_{z},v_{\lfloor sN^{m}/2\rfloor})=r+1$. Note that if $r+1<m+1$, 
then $z$ belongs to a sequence of the form $\{z_{0}+tN^{r}\}_{t=0}^{N-1}$
for some $z_{0}$ such that all $N$ terms in the sequence belong to the same 
$\left(r+1\right)$-block, while if $r+1=m+1$, then $z\in\{z_{0}+tN^{r}\}_{t=0}^{s-1}$
such that again all $s$ terms belong to the first $\left(m+1\right)$-block. Observe 
that the sequence $\{\cH(h^{(m,s)};\gamma_{i})\} _{i\in A}$ is concave by 
Lemma~\ref{lem:concave} and symmetric by Lemma~\ref{lem: symmetry},
where 
\begin{equation}
A=\left\{ \left\{ z_{0}+tN^{r}\right\}_{t=0}^{N-1}\bigcap\left[0,\left\lfloor sN^{m}/2\right\rfloor \right],
\left\{ sN^{m}-z_{0}-tN^{r}\right\}_{t=0}^{N-1}\bigcap\left[\left\lfloor sN^{m}/2\right\rfloor +1,
sN^{m}\right]\right\} 
\end{equation}
if $r+1<m+1$, and 
\begin{equation}
A=\left\{ \left\{ z_{0}+tN^{r}\right\} _{t=0}^{s-1}\bigcap\left[0,\left\lfloor sN^{m}/2\right\rfloor \right],
\left\{ sN^{m}-z_{0}-tN^{r}\right\} _{t=0}^{s-1}\bigcap\left[\left\lfloor sN^{m}/2\right\rfloor +1,
sN^{m}\right]\right\} 
\end{equation}
if $r+1=m+1$. Hence it attains its maximum only at the two midpoints of the sequence $A$ 
(which has $N+1$ terms in total). At least one of these two points is at distance $r$ from 
$v_{\lfloor sN^{m}/2\rfloor}$. Thus, by the inductive hypothesis we have that $\cH(h^{(m,s)};
\gamma_{z})<\cH(h^{(m,s)}; \gamma_{\lfloor sN^{m}/2\rfloor})$. 

Next, we look at the case when $s$ is even. By (\ref{eq:differenthdifference}) and the above 
result for the odd value $s-1$, we have that, for $t<\lfloor (s-1)N^{m}/2\rfloor +1$,
\begin{equation}
\begin{aligned}
&\cH\left(h^{\left(m,s\right)};\gamma_{\left\lfloor \left(s-1\right)N^{m}/2\right\rfloor +1}\right)
-\cH\left(h^{\left(m,s\right)};\gamma_{t}\right)\\
&\geq\cH\left(h^{\left(m,s-1\right)};
\gamma_{\left\lfloor \left(s-1\right)N^{m}/2\right\rfloor +1}\right)
-\cH\left(h^{\left(m,s-1\right)};\gamma_{t}\right)>0,
\end{aligned}
\end{equation}
and thus we only need to show that 
\begin{equation}
\cH\left(h^{\left(m,s\right)};\gamma_{\left\lfloor \left(s-1\right)N^{m}/2\right\rfloor +1}\right)
\geq\cH\left(h^{\left(m,s\right)};\gamma_{t}\right)\quad \forall \,
\left\lfloor \left(s-1\right)N^{m}/2\right\rfloor +1\leq t\leq\left\lfloor 
sN^{m}/2\right\rfloor +1.
\label{eq:otherside}
\end{equation}
To do this, recall first that by Lemma~\ref{lem:2ndsymmetries} the sequence $\{\cH(h^{(m,s)};
\gamma_{i})\} _{i\in S_{m}}$ is symmetric, concave and of odd cardinality. Furthermore, 
$\cH(h^{(m,s)}; \gamma_{\lfloor (s-1)N^{m}/2\rfloor +1})$ is the midpoint of this sequence, 
and for all $i,j\in S_{m}$ we have $d(v_{i},v_{j})=1$. Hence
\begin{equation}
\cH\left(h^{\left(m,s\right)};\gamma_{\left\lfloor \left(s-1\right)N^{m}/2\right\rfloor +1}\right)
>\cH\left(h^{\left(m,s\right)};\gamma_{i}\right)
\end{equation} 
for all $i<\lfloor (s-1)N^{m}/2\rfloor +1$ such that $d(v_{i},v_{\lfloor (s-1)N^{m}/2\rfloor +1})=1$. 
Now observe that $S_{m}\subset S_{m-1} \subset\ldots\subset S_{1}$, and suppose that 
$\cH(h^{(m,s)}; \gamma_{\lfloor (s-1)N^{m}/2\rfloor +1})>\cH(h^{(m,s)}; \gamma_{i})$ for 
all $i<\lfloor (s-1)N^{m}/2\rfloor +1$ such that $d(v_{i},v_{\lfloor (s-1)N^{m}/2\rfloor +1})=r$. 
If $i$ is such that $d(v_{i},v_{\lfloor (s-1)N^{m}/2\rfloor +1})=r+1$, then like in the $s$-odd 
case we can construct a concave and symmetric sequence such that the midpoint (and 
hence maximum) of this sequence is at distance $r$ or less from $v_{\lfloor (s-1)N^{m}/2
\rfloor +1}$. It follows that (\ref{eq:maxloc2}) and (\ref{eq: 1st max s even}) hold. 
\end{proof}

\begin{proposition}
\label{prop:hs-max-location-even}
Suppose that $N$ is even, and let
\begin{equation}
r=\left(\frac{s-1}{2}\right)N^{m}+\sum_{j=1}^{m-1}a_{m-j}N^{m-j}
+\frac{N}{2},
\label{eq:max-index-r}
\end{equation}
where 
\begin{equation}
a_{m-j}=\frac{N}{2}-Q\left(j+s+1\right).
\label{eq:amj-Neven}
\end{equation}
If $s=2a+1$ for some $a\in\left\{ 0,\ldots,\frac{N}{2}-1\right\}$, then
\begin{equation}
\cH\left(h^{\left(m,s\right)};\gamma_{r}\right)
=\max_{1 \leq i\leq sN^{m}}\cH\left(h^{\left(m,s\right)};\gamma_{i}\right)
=\max_{1 \leq i \leq N^n}\cH\left(h^{\left(m,s\right)};\gamma_{i}\right)
\label{eq:maxloc-1}
\end{equation}
and, for all $i<r$,
\begin{equation}
\cH\left(h^{\left(m,s\right)};\gamma_{i}\right)
<\cH\left(h^{\left(m,s\right)};\gamma_{r}\right).
\end{equation}
Similarly, 
\begin{equation}
\cH\left(h^{\left(m,s+1\right)};\gamma_{r}\right)
=\max_{1 \leq i\leq\left(s+1\right)N^{m}}\cH\left(h^{\left(m,s+1\right)};\gamma_{i}\right)
=\max_{1 \leq i \leq N^n}\cH\left(h^{\left(m,s+1\right)};\gamma_{i}\right)
\label{eq:maxloc2-1}
\end{equation}
and, for all $i<r$,
\begin{equation}
\cH\left(h^{\left(m,s+1\right)};\gamma_{i}\right)
<\cH\left(h^{\left(m,s\right)};\gamma_{r}\right).
\end{equation}
\end{proposition}

\begin{proof}
The coordinates $a_{m-j}$ are defined below (\ref{eq:SetSk}). Noting that $r$ 
is the midpoint of the smallest of the sets $\left\{S_{k}\right\}$ (for odd or even 
indices $k$ depending on the case in question), we see that the claim follows
from similar computations as those in the proof of Proposition~\ref{prop:hs-max-location}. 
\end{proof}


\subsection{Proof of Theorems~\ref{thm:Gamma-case3} and \ref{thm:Cstar-case3}}
\label{S4.3}

We now use Propositions~\ref{prop:hs-max-location} and \ref{prop:hs-max-location-even} 
to determine the size of the critical configurations and prove Theorems~\ref{thm:Cstar-case3} 
and \ref{thm:Gamma-case3} (Propositions~\ref{prop:h-general-max-location} and 
\ref{prop:Gamma-star} below). Recall the definition of the index set $\mathbb{I}$ in ~\eqref{eq:indexI}.

\begin{proposition}[Proof of Theorem \ref{thm:Cstar-case3}]
\label{prop:h-general-max-location} 
Let $h>0$, and take let $\left(m,s\right)\in\mathbb{I}$ be such that 
\begin{equation}
h^{\left(m,s\right)}\leq h < h^{\left(m,s-1\right)}.
\label{eq:h-between-hs}
\end{equation}
If $s$ is odd, then for $N$ odd
\begin{equation}
\max_{1 \leq i \leq N^n}\cH\left(h;\gamma_{i}\right)
=\cH\left(h;\gamma_{\left\lfloor sN^{m}/2\right\rfloor }\right)
\label{eq:max-odd}
\end{equation}
and for $N$ even 
\begin{equation}
\max_{1 \leq i \leq N^n}\cH\left(h;\gamma_{i}\right)=\cH\left(h;\gamma_{r}\right),
\label{eq:max-even,even}
\end{equation}
where $r$ is given in \eqref{eq:max-index-r}. If $s$ is even, then for $N$ odd
\begin{equation}
\max_{1 \leq i \leq N^n}\cH\left(h;\gamma_{i}\right)
=\cH\left(h;\gamma_{\left\lfloor \left(s-1\right)N^{m}/2\right\rfloor +1}\right),
\label{eq:max-odd,even}
\end{equation}
and for $N$ even
\begin{equation}
\max_{1 \leq i \leq N^n}\cH\left(h;\gamma_{i}\right)
=\cH\left(h;\gamma_{r'}\right),
\label{eq:max-even,odd}
\end{equation}
where $r'$ is obtained by replacing $s$ by $s-1$ in the leading term 
in \eqref{eq:max-index-r}. If $h\geq \tilde{J}(1-\frac{1}{N})n$, then 
$\max_{1 \leq i \leq N^n}\cH(h;\gamma_{i})=\cH(h;\gamma_{0})$.
If the inequality on the left side of $h$ in \eqref{eq:h-between-hs} is also 
strict, then these are the unique maxima.
\end{proposition}

\begin{proof}
We give the proof for $N$ odd and $s$ even, the proof for all other cases being similar. From 
(\ref{eq:differenthdifference}) and Proposition~\ref{prop:hs-max-location}
we have that, for all $i\leq\left\lfloor \left(s-1\right)N^{m}/2\right\rfloor +1$,
\begin{equation}
\cH\left(h;\gamma_{\left\lfloor \left(s-1\right)N^{m}/2\right\rfloor +1}\right)
-\cH\left(h;\gamma_{i}\right)\geq\cH\left(h^{\left(m,s-1\right)};
\gamma_{\left\lfloor \left(s-1\right)N^{m}/2\right\rfloor +1}\right)
-\cH\left(h^{\left(m,s-1\right)};\gamma_{i}\right)\geq0
\label{eq:h-gen-max-ineq}
\end{equation}
and, for $i\geq\left\lfloor \left(s-1\right)N^{m}/2\right\rfloor +1$,
\begin{equation}
\cH\left(h;\gamma_{\left\lfloor \left(s-1\right)N^{m}/2\right\rfloor +1}\right)
-\cH\left(h;\gamma_{i}\right)\geq\cH\left(h^{\left(m,s\right)};
\gamma_{\left\lfloor \left(s-1\right)N^{m}/2\right\rfloor +1}\right)
-\cH\left(h^{\left(m,s\right)};\gamma_{i}\right)\geq0.
\label{eq:h-gen-max-ineq2}
\end{equation}
This proves the first claim. If the inequalities in (\ref{eq:h-between-hs}) are both 
strict, then both (\ref{eq:h-gen-max-ineq}) and (\ref{eq:h-gen-max-ineq2})
are strict whenever $i\neq\left\lfloor \left(s-1\right)N^{m}/2\right\rfloor +1$.
\end{proof}

\begin{remark}
{\rm It is easy to check that if we take $h=\tilde{J}((1-\frac{1}{N})(n-m)
-(s-1)\frac{1}{N})=h^{(m,s)}$ or $h=h^{(m,s-1)}=\tilde{J}((1-\frac{1}{N})
(n-m)-(s-2)\frac{1}{N})$ in (\ref{eq:h-between-hs}), then (\ref{eq:max-odd,even})
and (\ref{eq:max-odd}) remain true.}
\end{remark}

\begin{proposition}[Proof of Theorem \ref{thm:Gamma-case3}]
\label{prop:Gamma-star}
Let $h>0$, and let $m$ and $s$ satisfy \eqref{eq:h-between-hs}.\\ 
{\rm (1)} Suppose that $N$ is odd. For $s$ even,  
\begin{equation}
\begin{aligned}
\Gamma^{\star} 
& = \frac{\tilde{J}}{4N}\left(N^{m}\left[2s\left(N-\frac{s}{2}+1\right)-N
-1\right]+N-2s+1\right)\\
&\qquad +\frac{1}{2}\left(\tilde{J}\left(1-\frac{1}{N}\right)
\left(n-m-1\right)-h\right)\left(N^{m}\left(s-1\right)+1\right)
\end{aligned}
\end{equation}
while for $s$ odd
\begin{equation}
\begin{aligned}
\Gamma^{\star}
& =\frac{\tilde{J}}{4N}\left(N^{m}\left[2s\left(N-\frac{s}{2}\right)+N\right]+N-2s-1\right)\\
&\qquad +\frac{1}{2}\left(\tilde{J}\left(1-\frac{1}{N}\right)\left(n-m-1\right)-h\right)
\left(sN^{m}+1\right).
\end{aligned}
\end{equation}
{\rm (2)} Suppose that $N$ is even. For $s$ odd,
\begin{equation}
\begin{aligned}
&\Gamma^{\star} = \Gamma_{s}^{\star}\\
& = \frac{\tilde{J}}{2}N^{1+Q\left(m+1\right)}\left(\frac{N^{m-2+Q\left(m\right)}-1}{N^{2}-1}\right)
+ \tilde{J}\left(\frac{1}{2}\left(\frac{N^{m}-1}{N-1}\right)
- N^{Q\left(m\right)}\left(\frac{N^{m-Q\left(m\right)}-1}{N^{2}-1}\right)\right)\\
&\qquad \qquad \times \left(N-s\right)\\
&\qquad + \left[\frac{N}{4}\left(\frac{N^{m}-1}{N-1}\right)-N^{Q\left(m\right)}
\left(\frac{N^{m-Q\left(m\right)}-1}{N^{2}-1}\right)+N^{m-1}\left(\frac{s-1}{2}\right)
\left(N-\frac{s-1}{2}\right)\right]\\
&\qquad + \left(\left(\frac{s-1}{2}\right)N^{m}+\frac{N}{2}\left(\frac{N^{m}-1}{N-1}\right)
-N^{1+Q\left(m\right)}\left(\frac{N^{m-Q\left(m\right)}-1}{N^{2}-1}\right)\right)\\
&\qquad \qquad \times \left(\tilde{J}\left(1-\frac{1}{N}\right)\left(n-m-1\right)-h\right),
\end{aligned}
\end{equation}
while for $s$ even, 
\begin{equation}
\begin{aligned}
&\Gamma^{\star}=\Gamma_{s-1}^{\star}+\left(h^{\left(s-1\right)}-h\right)\\
&\times \left(sN^{m}
-\left(\frac{s-1}{2}\right)N^{m}-\left(\frac{N}{2}\right)\left(\frac{N^{m}-1}{N-1}\right)
+N^{1+Q\left(m\right)}\left(\frac{N^{m-Q\left(m\right)}-1}{N^{2}-1}\right)\right).
\end{aligned}
\label{eq:shifted-Gamma}
\end{equation}
\end{proposition}

\begin{proof}
From Proposition~\ref{prop:h-general-max-location} we have that, for $N$ odd 
and $s$ even, 
\begin{equation}
\Gamma^{\star}=\cH\left(h;\gamma_{\left\lfloor \left(s-1\right)N^{m}/2\right\rfloor +1}\right)
-\cH\left(h;\boxminus\right),
\label{eq:GammastarNoddseven}
\end{equation}
where we also note that 
\begin{equation}
\left\lfloor \left(s-1\right)N^{m}/2\right\rfloor +1 
= \left\lfloor \left(s-1\right)/2\right\rfloor N^{m}+1
+\sum_{i=0}^{m-1}\left\lfloor \frac{N}{2}\right\rfloor N^{i}
= \left(\frac{s}{2}-1\right)N^{m}+\frac{1}{2}\left(N^{m}+1\right).
\end{equation}
We can now use this decomposition together with (\ref{eq:HKregsymetric}) to calculate 
$\Gamma^{\star}$ (after a fair deal of tedious computations). For the case where $N$ 
is odd and $s$ is odd, $\Gamma_{s}^{\star}$ is calculated in the same manner, while 
(\ref{eq:shifted-Gamma}) follows immediately from (\ref{eq:differenthdifference}).
\end{proof}


\subsection{Proof of Theorem~\ref{thm:Kstar-case2}}
\label{S4.4}

In this section we identify the configurations in the sets $U_{\sigma}^{-}$ and $U_{\sigma}^{+}$ 
defined in Lemma~\ref{lem:variational-lemma} and compute the prefactor $K^{\star}$ 
(Corollary ~\ref{cor:Prefactor simple} and Proposition ~\ref{prop:Prefactor other case} below).  

Let $M$ be the volume of the critical configurations, whose value was determined 
in Proposition \ref{prop:h-general-max-location} (i.e., $M=\lfloor sN^{\hat{m}}/2\rfloor$ 
if $N$ is odd and $s$ is odd, etc.). Recall that $v_{M}$ is the last vertex flipped (from 
$-1$ to $+1$) in obtaining the configuration $\gamma_{M}$ along the path $\gamma$. 
Let $b\geq1$ and let $w$ be any vertex such that $d\left(w,v_{M}\right)=b$. Define the 
configuration $\sigma_{b}$ by
\begin{equation}
\sigma_{b}\left(v\right)
= \begin{cases}
\gamma_{M}\left(v\right), & v\neq w,\\
-\gamma_{M}\left(v\right), & v=w.
\end{cases}\label{eq:sigma_b}
\end{equation}
Assuming that $h$ satisfies (\ref{eq:h-between-hs}) with strict inequalities, we know 
from Proposition \ref{prop:h-general-max-location} that any uniformly optimal path 
attains a unique global maximum. Hence if $b=1$, then $\cH\left(\sigma_{b}\right)
<\cH\left(\gamma_{M}\right)$, since $\cH\left(\sigma_{b}\right)\in\left\{\cH\left(\gamma_{M-1}\right),
\cH\left(\gamma_{M+1}\right)\right\}$. The following lemma shows that if $N\not\neq2,4$ 
and $m\geq1$, then the only neighbours of  $\gamma_{M}$ with lower energy are those 
obtained by flipping a vertex at distance $b=1$.

\begin{lemma}
\label{lem:energy sigma_b}Let $\sigma_{b}$ be defined as in \eqref{eq:sigma_b}.
Suppose that $N \neq 2,4$ and $m=\hat{m}\geq1$. Then $\cH\left(\sigma_{b}\right)
>\cH\left(\gamma_{M}\right)$ whenever $b\geq2$. 
\end{lemma}

\begin{proof}
We first consider $\sigma_{b}\left(w\right)=-1$, where $w$ is the vertex at which 
$\sigma_{b}$ differs from $\gamma_{M}$. Note that $b\leq m+1$, since there are no 
$+1$-valued vertices in $\gamma_{M}$ that are at distance larger than $m+1$ from 
each other. By (\ref{eq:vertexflip-s2}) we have that
\begin{eqnarray}
\nonumber
\cH\left(\sigma_{b}\right)-\cH\left(\gamma_{M}\right) 
& = & \tilde{J}\left(b-1\right)\left(1-\frac{1}{N}\right)
+\tilde{J}N^{-b}\left(2\sum_{i=0}^{b-1}a_{i}N^{i}-N^{b}-N^{b-1}\right)\\
&& \qquad + \frac{\tilde{J}}{N}\sum_{i=b}^{n-1}\left(2a_{i}-N+1\right)+h.
\end{eqnarray}
If $b=m+1$, then the right-hand side gives
\begin{eqnarray}
\nonumber
&\tilde{J}\left(\left(b-1\right)\left(1-\frac{1}{N}\right)
+N^{-b}\left(2\sum_{i=0}^{b-1}a_{i}N^{i}-N^{b}-N^{b-1}\right)
-\left(1-\frac{1}{N}\right)\left(n-m-1\right)+\frac{h}{\tilde{J}}\right)\\ 
&\geq \tilde{J}\left(\left(m+1\right)\left(1-\frac{1}{N}\right)+N^{-m-1}
\left(2M-N^{m+1}-N^{m}\right)-\left(s-1\right)\frac{1}{N}\right),
\end{eqnarray}
where the inequality follows from the bounds on $h$ in (\ref{eq:h-between-hs}).
It is easy to see from the value of $M$ in Theorem~\ref{thm:Cstar-case3}
that the above is strictly larger than 
\begin{equation}
\begin{aligned}
&\tilde{J}\left(\left(m+1\right)\left(1-\frac{1}{N}\right)
+\frac{1}{N}\left(\left(s-1-Q\left(s+1\right)\right)-N-1\right)
-\left(s-1\right)\frac{1}{N}\right)\\
&\qquad \geq \tilde{J}\left(m\left(1-\frac{1}{N}\right)-\frac{2}{N}\right) \geq 0.
\end{aligned}
\end{equation}
Hence we conclude that for $b=m+1$ and $\sigma_{b}\left(w\right)=-1$ the claim of the 
lemma holds. 

Now assume that $b\leq m$. If $N$ is odd, then $a_{0}=\frac{N-1}{2}+Q\left(s+1\right)$ 
and $a_{i}=\frac{N-1}{2}$ for $1\leq i\leq m-1$, while $a_{m}=\lfloor \frac{s-1}{2}\rfloor 
=\frac{s-1-Q\left(s+1\right)}{2}$ and $a_{i}=0$ for $i>m$. If $h$ satisfies (\ref{eq:h-between-hs}) 
for some $1\leq s\leq N-1$ and $2\leq m\leq n-1$ (we do not need to consider the case $m=1$ 
because $m\geq b\geq2$), then this implies 
\begin{equation}
\begin{aligned}
&\cH(\sigma_{b})-\cH(\gamma_{M})
=\tilde{J}\Big((b-n+m)\left(1-\frac{1}{N}\right)+N^{-b}\left(2Q(s+1)-1\right)\\
&\qquad -\frac{1}{N}\left(N-s+1+Q\left(s+1\right)\right)+\frac{h}{\tilde{J}}\Big)
\end{aligned}
\end{equation}
and hence $\cH\left(\sigma_{b}\right)\leq\cH\left(\gamma_{k}\right)$
if and only if 
\begin{equation}
\begin{aligned}
&b \leq 1+ \left(1-\frac{1}{N}\right)^{-1}\Bigg(\frac{1}{N}\left(N-s+1+Q\left(s+1\right)\right)\\
&\qquad + \left(1-\frac{1}{N}\right)\left(n-m-1\right)-N^{-b}\left(2Q\left(s+1\right)-1\right)
-\frac{h}{\tilde{J}}\Bigg).
\end{aligned}
\label{eq:bineq-Nodd}
\end{equation}
From (\ref{eq:h-between-hs}) we have that the right-hand side of (\ref{eq:bineq-Nodd}) is 
less than or equal to 
\begin{equation}
\frac{Q\left(s+1\right)+1}{N}-N^{-b}\left(2Q\left(s+1\right)-1\right)
\end{equation}
and hence is less than $2$ when $N\geq3$. This implies that $\cH(\sigma_{b})>\cH(\gamma_{k})$ 
when $b\geq2$. If $N$ is even, then 
\begin{equation}
\begin{aligned}
\left(\cH\left(\sigma_{b}\right)-\cH\left(\gamma_{M}\right)\right)\tilde{J}^{-1} 
& = \left(b-1\right)\left(1-\frac{1}{N}\right)
+N^{-b}\left(2\sum_{i=0}^{b-1}a_{i}N^{i}-N^{b}-N^{b-1}\right)\\
&\qquad -\frac{1}{N}\left(N-s+1+Q\left(s+1\right)\right)
-\frac{1}{N}\left(N-1\right)\left(n-m-1\right)\\
&\qquad + \frac{1}{N}\left(1-Q\left(s+m\right)-Q\left(s+b+1\right)\right)+\frac{h}{\tilde{J}},
\end{aligned}
\end{equation}
and hence $\cH\left(\sigma_{b}\right)\leq\cH\left(\gamma_{k}\right)$ if and only if
\begin{equation}
\begin{aligned}
\left(b-1\right)\left(1-\frac{1}{N}\right) 
& \leq  1-\frac{s}{N}+\frac{Q\left(s+1\right)}{N}+\left(1-\frac{1}{N}\right)\left(n-m-1\right)
\label{eq:bineq-Neven}\\
&\qquad +  \frac{Q\left(s+m\right)}{N}+\frac{Q\left(s+b+1\right)}{N}-\frac{h}{\tilde{J}}
-N^{-b}\left(2\sum_{i=0}^{b-1}a_{i}N^{i}-N^{b}-N^{b-1}\right).
\end{aligned} 
\end{equation}
Since $h$ satisfies (\ref{eq:h-between-hs}), this implies 
\begin{equation}
\begin{aligned}
b & \leq 1+\left(1-\frac{1}{N}\right)^{-1}
\Bigg(\frac{Q\left(s+1\right)+Q\left(s+b+1\right)+Q\left(s+m\right)}{N}\\
&\qquad -N^{-b}\left(2\sum_{i=0}^{b-1}a_{i}N^{i}-N^{b}-N^{b-1}\right)\Bigg)\\
& \leq  1+\left(1-\frac{1}{N}\right)^{-1}\left(\frac{Q\left(s+1\right)
 +Q\left(s+b+1\right)+Q\left(s+m\right)+2Q\left(s+m-b\right)}{N}-R_{b}\right),
\end{aligned}
\end{equation}
where $R_{b}=N^{-b}((\frac{1}{N-1})(N^{b-1}-N-2N^{b-2}))$. The right-hand side is less than 
$2$ when $N \geq 6$, in which case $\cH\left(\sigma_{b}\right)>\cH\left(\gamma_{k}\right)$
when $b \geq 2$. 

Now suppose that $\sigma_{b}\left(w\right)=+1$. Let us first consider the case when $N$ is 
odd. Suppose that $b>m$. Then by (\ref{eq:vertexflip+s2})
\begin{equation}
\begin{aligned}
&\cH\left(\sigma_{b}\right)-\cH\left(\gamma_{k}\right)\\
& = \tilde{J}\left(1-\frac{1}{N}\right)\left(b-1\right)
+\tilde{J}N^{b-1}\left(N^{b}-2\sum_{i=0}^{b-1}a_{i}N^{i}-N^{b-1}\right)
+\sum_{i=b}^{n-1}\tilde{J}_{i+1}N^{i}\left(N-2a_{i}-1\right)-h\\
& = \tilde{J}\Bigg(\left(1-\frac{1}{N}\right)\left(b-1\right)+N^{b-1}\left(N^{b}
-\left(s-Q\left(s+1\right)\right)N^{m}+1-2Q\left(s+1\right)-N^{b-1}\right)\\
&\qquad  + \left(1-\frac{1}{N}\right)\left(n-b\right)-\frac{h}{\tilde{J}}\Bigg).
\end{aligned}
\end{equation}
From (\ref{eq:h-between-hs}) it follows that this is larger than or equal to 
\begin{equation}
\begin{aligned}
&\tilde{J}\Bigg(N^{-b}\left(N^{b}-\left(s-Q\left(s+1\right)\right)N^{m}
+1-2Q\left(s+1\right)-N^{b-1}\right)\\
&\qquad +\left(1-\frac{1}{N}\right)\left(m-1\right)
+\left(s-2\right)\frac{1}{N}\Bigg) > 0.
\end{aligned}
\end{equation}
Hence, the inequality $\cH\left(\sigma_{b}\right)\leq\cH\left(\gamma_{k}\right)$ is at most 
possible for $b\leq m$. In this case we get that $\cH\left(\sigma_{b}\right)\leq\
cH\left(\gamma_{k}\right)$ if and only if
\begin{equation}
\begin{aligned}
b &\leq 1+\left(1-\frac{1}{N}\right)^{-1}\Bigg(\frac{h}{\tilde{J}}
-\left(1-\frac{1}{N}\right)\left(n-m-1\right)\\
&\qquad -\frac{1}{N}\left(N-s+Q\left(s+1\right)-1\right)
-N^{-b}\left(1-2Q\left(s+1\right)\right)\Bigg).
\end{aligned}
\label{eq:bineq+Nodd}
\end{equation}
Once again, from (\ref{eq:h-between-hs}) it follows that (\ref{eq:bineq+Nodd}) is satisfied 
if and only if
\begin{equation}
b\leq1+\left(1-\frac{1}{N}\right)^{-1}\left(-\frac{1}{N}\left(Q\left(s+1\right)-1\right)
-N^{-b}\left(1-2Q\left(s+1\right)\right)\right)<2 \qquad \forall\,N \geq 2.
\end{equation}
Similarly, if $N$ is even, then for $b>m$ we get
\begin{equation}
\begin{aligned}
\cH\left(\sigma_{b}\right)-\cH\left(\gamma_{M}\right) 
& \geq \tilde{J}\Bigg(\left(1-\frac{1}{N}\right)\left(b-1\right)-\frac{1}{N\left(N-1\right)}
+\frac{Q\left(s+1\right)}{N}+\frac{Q\left(s+m\right)}{N}\\
& \qquad + \frac{Q\left(s+b+1\right)}{N}+\frac{Q\left(s+m-b\right)}{N}-\frac{2}{N}\Bigg),
\end{aligned}
\end{equation}
which is larger than $0$ when $N\geq4$. Thus, once again we only need to consider $b\leq m$, 
for which $\cH\left(\sigma_{b}\right)-\cH\left(\gamma_{k}\right) \leq 0$ if and only if 
\begin{equation}
\begin{aligned}
b & \leq 1+\left(1-\frac{1}{N}\right)^{-1}\Bigg(-\frac{1}{N}
\left(N-s+1+Q\left(s+1\right)\right)-\left(1-\frac{1}{N}\right)\left(n-m-1\right)\\
&\qquad  + \frac{1}{N}\left(1-Q\left(s+m\right)-Q\left(s+b+1\right)\right)
+\frac{h}{\tilde{J}}-N^{-b}\left(N^{b}-2\sum_{i=0}^{b-1}a_{i}N^{i}-N^{b-1}\right)\Bigg)\\
& \leq 1+\left(1-\frac{1}{N}\right)^{-1}\Bigg(\frac{1}{N}
\Big(1-Q\left(s+m\right)-Q\left(s+b+1\right)-Q\left(s+1\right)\\
&\qquad -2Q\left(s+m-b\right)\Big)+\frac{1}{N\left(N-1\right)}\Bigg),
\end{aligned} 
\label{eq:bineq+Neven}
\end{equation}
which is less than $2$ when $N \geq 4$.
\end{proof}

The prefactor $K^\star$ can now be easily computed.

\begin{corollary}
\label{cor:Prefactor simple}
Suppose that $N \neq 2,4$ and $m\geq1$. Then
\begin{equation}
\frac{1}{K^\star} = a_{0}N^{n-\hat{m}-2}\prod_{i=0}^{\hat{m}}
{N \choose a_{i}}\left(N-a_{i}\right).
\label{eq:corollary 1/K}
\end{equation}
\end{corollary}

\begin{proof}
By Lemma~\ref{lem:energy sigma_b} we have that, for all $\sigma\in\cC^{\star}$, 
\begin{equation}
\begin{aligned}
\left|U_{\sigma}^{-}\right| 
&= \left|w\in\Lambda_{N}^{n}\colon\, d\left(w,v_{M}\right)=1,\,\gamma_{M}\left(w\right)=-1\right|
= a_{0},\\
\left|U_{\sigma}^{-}\right| 
& = \left|w\in\Lambda_{N}^{n}\colon\, d\left(w,v_{M}\right)=1,\,\gamma_{M}\left(w\right)=-1\right|
= N-a_{0}.
\end{aligned}
\end{equation}
Furthermore, it is a simple combinatorial fact that
\begin{equation}
\begin{aligned}
\left|\cC^{\star}\right| 
& = \left|\left\{ \sigma\in\Omega\colon\,\sigma=\varphi\left(\gamma_{M}\right)
\mbox{ for some isometric bijection }\varphi\colon\Lambda_{N}^{n}\to
\Lambda_{N}^{n}\right\} \right|\\
& = N^{n-\hat{m}-1}\left(N-a_{0}\right)^{-1}
\prod_{i=0}^{\hat{m}}{N \choose a_{i}}\left(N-a_{i}\right).
\end{aligned}
\end{equation}
Equation \eqref{eq:corollary 1/K} now follows from Lemmas~\ref{lem:variational-lemma} and \ref{lem:circumventiing-path}. 
\end{proof}

We can also investigate what the prefactor amounts to when the precondition of 
Corollary~\ref{cor:Prefactor simple} is not satisfied. For this, we define 
\begin{eqnarray}
O_{d} & = & \left\{ 1\right\} \bigcup\left\{ 2\leq b\leq\hat{m}\colon\, 
b\mbox{ satisfies inequality (\ref{eq:bineq-Nodd}) }\right\}, \nonumber \\
O_{u} & = & \left\{ 1\right\} \bigcup\left\{ 2\leq b\leq\hat{m}\colon\, 
b\mbox{ satisfies inequality (\ref{eq:bineq+Nodd}) }\right\}, 
\label{eq:OdOu}
\end{eqnarray}
and 
\begin{eqnarray}
E_{d} & = & \left\{ 1\right\} \bigcup\left\{ 2\leq b\leq\hat{m}\colon\, 
b\mbox{ satisfies inequality (\ref{eq:bineq-Neven}) }\right\}, \nonumber \\
E_{u} & = & \left\{ 1\right\} \bigcup\left\{ 2\leq b\leq\hat{m}\colon\, 
b\mbox{ satisfies inequality (\ref{eq:bineq+Neven}) }\right\}. 
\label{eq:EdEu}
\end{eqnarray}
Then the following is immediate from Lemmas~\ref{lem:variational-lemma}
and \ref{lem:circumventiing-path}.

\begin{proposition}
\label{prop:Prefactor other case}
Suppose that $h$ satisfies 
\[
h^{(m,s)}<h<h^{(m,s-1)}
\]
for some $(m,s) \in \mathbb{I}$. If $N$ is odd, then the prefactor $K^\star$ is given
by 
\begin{equation}
\frac{1}{K^{\star}} = 
\frac{\left[\sum_{i\in O_{d}} a_{i-1}N^{i-1}\right]
\left[\sum_{i\in O_{u}}\left(N^{i}-a_{i-1}N^{i-1}\right)\right]}
{\left[\sum_{i\in O_{d}}a_{i-1}N^{i-1}\right]
+\left[\sum_{i\in O_{u}}\left(N^{i}-a_{i-1}N^{i-1}\right)\right]}
\,\frac{N^{n-\hat{m}-1}}{N-a_{0}}\,\prod_{i=0}^{\hat{m}}
{N \choose a_{i}}\left(N-a_{i}\right),
\label{eq:1/Kstar}
\end{equation}
where $a_{0}=\frac{N-1}{2}+1$, $a_{i}=\frac{N-1}{2}$ for $i=1,\ldots,\hat{m}-1$ and 
$a_{\hat{m}}=\frac{s-1-\left(s+1\right)\mod2}{2}$. If $N$ is even, then the summations 
in \eqref{eq:1/Kstar} are over $E_{d}$ and $E_{u}$, respectively, and the terms $a_{i}$ 
are replaced by $\bar{a}_{i}$ defined in \eqref{eq:max-index-r}.
\end{proposition}



\end{document}